\documentclass[a4paper,10pt]{amsart}

\usepackage[letterpaper,left=1.5in,right=1.5in,top=1in,bottom=1in]{geometry}

\newcommand{\myauthor}{Elden Elmanto, Marc Levine, Markus Spitzweck, Paul Arne {\O}stv{\ae}r}
\newcommand{\mytitle}{Algebraic Cobordism and \'{E}tale Cohomology}

\title{Algebraic Cobordism and \'{E}tale Cohomology}
\author{\myauthor}
\date{}

\usepackage[pdfstartview=FitH,
            pdfauthor={\myauthor},
            pdftitle={\mytitle},
            colorlinks,
            linkcolor=reference,
            citecolor=citation,
            urlcolor=e-mail,
            backref]{hyperref}

\usepackage{amsmath}

\usepackage{tikz}
\usetikzlibrary{matrix,arrows}
\usepackage{amscd}
\usepackage{amsbsy}
\usepackage{amssymb}
\usepackage{verbatim}
\usepackage{eufrak}
\usepackage{microtype}
\usepackage{hyperref}
\usepackage{mathrsfs}
\usepackage{amsthm}
\usepackage[all,cmtip]{xy}
\usepackage[abbrev,lite]{amsrefs}
\usepackage{mathpazo}
\usepackage{dsfont}

\usepackage{color}
\definecolor{todo}{rgb}{1,0,0}
\definecolor{conditional}{rgb}{0,1,0}
\definecolor{e-mail}{rgb}{0,.40,.80}
\definecolor{reference}{rgb}{.20,.60,.22}
\definecolor{mrnumber}{rgb}{.80,.40,0}
\definecolor{citation}{rgb}{0,.40,.80}

\let\oldmarginpar\marginpar
\renewcommand\marginpar[1]{\-\oldmarginpar[\raggedleft\footnotesize #1]%
{\raggedright\footnotesize #1}}

\newcommand{\Fscr}{\mathscr{F}}

\newcommand{\Oscr}{\mathscr{O}}

\newcommand{\Mrm}{\mathrm{M}}

\DeclareMathOperator{\limone}{lim^1}
\newcommand{\proj}{\mathsf{proj}}

\newcommand{\hyp}{\mathsf{hyp}}
\newcommand{\compl}{{}^{{\kern -.5pt}\wedge}_{\ell}}

\newcommand{\C}{\mathrm{C}}
\newcommand{\D}{\mathrm{D}}
\newcommand{\F}{\mathsf{F}}
\renewcommand{\H}{\mathrm{H}}
\newcommand{\SH}{\mathrm{SH}}
\newcommand{\M}{\mathsf{M}}

\renewcommand{\sc}{\mathrm{sc}}

\newcommand{\Shv}{\mathrm{Shv}}
\newcommand{\conn}{\mathrm{conn}}

\newcommand{\Sp}{\mathsf{Sp}}
\newcommand{\LL}{\mathrm{L}}

\newcommand{\TriCat}{\mathrm{TriCat}}

\newcommand{\cell}{\mathsf{cell}}

\renewcommand{\AA}{\mathbf{A}}

\newcommand{\GG}{\mathbf{G}}

\newcommand{\NN}{\mathds{N}}
\newcommand{\PP}{\mathbf{P}}
\newcommand{\QQ}{\mathbb{Q}}
\newcommand{\RR}{\mathbf{R}}
\newcommand{\ZZ}{\mathbf{Z}}

\renewcommand{\P}{\mathrm{P}}
\newcommand{\gr}{\mathsf{gr}}

\DeclareMathOperator{\CAlg}{CAlg}

\newcommand{\op}{\mathrm{op}}

\DeclareMathOperator{\Gal}{Gal}

\DeclareMathOperator{\id}{id}

\DeclareMathOperator{\Char}{{char}\,}

\DeclareMathOperator{\Pic}{Pic}

\newcommand{\Nis}{\mathrm{Nis}}

\newcommand{\Spc}{\mathrm{Spc}}

\newcommand{\Sm}{\mathrm{Sm}}

\newcommand{\KU}{\mathrm{KU}}
\newcommand{\tr}{\mathrm{tr}}

\newcommand{\eff}{\mathrm{eff}}

\DeclareMathOperator{\Aut}{Aut}

\DeclareMathOperator{\Hom}{Hom}
\DeclareMathOperator{\Maps}{Maps}
\DeclareMathOperator{\maps}{maps}
\DeclareMathOperator{\fib}{fib}

\newcommand{\mapint}{\underline{\mathrm{map}}}
\newcommand{\Mod}{\mathrm{Mod}}
\renewcommand{\Pr}{\mathrm{Pr}}
\renewcommand{\CAlg}{\mathrm{CAlg}}

\newcommand{\cd}{\mathsf{cd}}

\newcommand{\sspt}{\mathbf{1}}

\newcommand{\G}{\mathrm{G}}

\DeclareMathOperator{\Ho}{Ho}

\newcommand{\Cat}{\mathrm{Cat}}

\newcommand{\Sing}{\mathrm{Sing}}

\newcommand{\stab}{\mathsf{stab}}
\newcommand{\Ind}{\mathrm{Ind}}
\newcommand{\KGL}{\mathsf{KGL}}
\newcommand{\MGL}{\mathsf{MGL}}
\newcommand{\MU}{\mathsf{MU}}

\newcommand{\E}{\mathsf{E}}
\newcommand{\unit}{\mathbf{1}}

\newcommand{\s}{\mathbf{S}}

\newcommand{\pre}{\mathrm{pre}}

\newcommand{\ho}{\mathsf{ho}}

\newcommand{\Fun}{\mathrm{Fun}}

\DeclareMathOperator*{\colim}{colim}

\newcommand{\et}{\mathrm{\acute{e}t}}
\newcommand{\Et}{\mathrm{\acute{E}t}}

\newcommand{\Ab}{\mathrm{Ab}}

\newcommand{\e}{\mathrm{e}}

\newcommand{\mot}{\mathsf{m}}

\DeclareMathOperator{\Spec}{Spec}

\DeclareMathOperator{\DM}{DM}

\DeclareMathOperator{\Sq}{\mathrm{Sq}}

\newcommand{\Spt}{\mathrm{Spt}}

\DeclareMathOperator{\Stab}{Stab}
\DeclareMathOperator{\Sch}{Sch}

\newcommand{\iso}{\cong}
\renewcommand{\gr}{\mathrm{gr}}

\theoremstyle{plain}
\newtheorem{theorem}{Theorem}[section]
\newtheorem*{theorem*}{Theorem}
\newtheorem{lemma}[theorem]{Lemma}

\newtheorem{proposition}[theorem]{Proposition}

\newtheorem{corollary}[theorem]{Corollary}

\newtheoremstyle{named}{}{}{\itshape}{}{\bfseries}{.}{.5em}{#1 \thmnote{#3}}
\theoremstyle{named}

\theoremstyle{definition}
\newtheorem{definition}[theorem]{Definition}

\newtheorem{example}[theorem]{Example}

\newtheorem{question}[theorem]{Question}

\theoremstyle{remark}
\newtheorem{remark}[theorem]{Remark}

\subjclass[2010]{14F20, 14F42, 19E15, 55P42, 55P43}
\keywords{Motivic homotopy theory, \'{e}tale motives, slice spectral sequence, algebraic cobordism, algebraic $K$-theory, motivic cohomology.}
\begin{document}

\begin{abstract}
Thomason's  \'{e}tale descent theorem for Bott periodic algebraic $K$-theory \cite{aktec} is generalized to any $\MGL$ module over a regular Noetherian scheme of finite dimension. Over arbitrary Noetherian schemes of finite dimension, this generalizes the analog of Thomason's theorem for Weibel's homotopy $K$-theory. This is achieved by amplifying the effects from the case of motivic cohomology, 
using the slice spectral sequence in the case of the universal example of algebraic cobordism. 
We also obtain integral versions of these statements: Bousfield localization at \'etale motivic cohomology is the universal way to impose \'etale descent for these theories.
As applications, we describe the \'etale local objects in modules over these spectra and show that they satisfy the full six functor formalism, 
construct an \'etale descent spectral sequence converging to Bott-inverted motivic Landweber exact theories, 
and prove cellularity and effectivity of the \'{e}tale versions of these motivic spectra.
\end{abstract}
\maketitle
\tableofcontents

\section{Introduction}
In constructing the motivic homotopy category $\H(S)$ over a base scheme $S$, 
Morel and Voevodsky \cite{morel-voevodsky} use the Nisnevich topology. 
One inspiration for this choice is the fact that many interesting invariants of schemes such as algebraic $K$-theory and Chow groups satisfy Nisnevich (hyper)descent but not \'{e}tale (hyper)descent (we refer the reader to~\S\ref{sect:hyp-remove} for a self-contained discussion on hyperdescent). This failure for $K$-theory has both arithmetic and geometric sources. For example, if $K$-theory did satisfy \'etale hyperdescent then $K_0$ of many fields will admit contributions from its (generally non-zero) Brauer group which is simply not true since this group is always just $\ZZ$; see \cite{mitchell}*{Section 4} for a discussion. Even assuming that the field is algebraically closed, \'etale hyperdescent combined with the comparison theorem in \'etale cohomology would prove that profinitely completed algebraic $K$-theory of a complex variety would simply be the profinitely completed topological $K$-theory of its complex points which is again not true; see \cite{eventuallysurjects} for the precise relationship.

As a result, 
the venture of motivic homotopy theory has focused on algebro-geometric invariants which are Nisnevich local in nature. Besides the fact that the aforementioned invariants are representable in the motivic homotopy category, the Nisnevich topology enjoys some ``finiteness" properties. If $S$ is noetherian, then the Nisnevich cohomological dimension of $S$ is bounded above by its Krull dimension \cite{kato-saito};
this is important for dualizability and compactness questions, 
and also for the convergence of spectral sequences derived from motivic homotopy theory. Recent work of Clausen-Mathew has extended greatly the scope of these ``finiteness" properties of the Nisnevich site \cite{clausen-mathew}.


However, 
there are good reasons to study versions of $\H(S)$ and its stable counterpart $\SH(S)$ using the \'{e}tale topology. 
Indeed, 
$\SH(S)$ suffers from the fact that its ``linearization", 
i.e., 
Voevodsky's big category of motives $\DM$ \cite{ro} and subcategories thereof, 
do not
admit the elusive motivic $t$-structure \cite{VSF}*{Proposition 4.8}.
In recent years, 
a theory of \'{e}tale motives has been developed by Ayoub \cite{ayoub-etale} and by Cisinski-Deglise \cite{etalemotives}, 
including a full-fledged six functors formalism;
for an overview, 
see \cite{ayoub-icm}. 
The \'{e}tale theory has many good properties ---
from the construction of an integral \'{e}tale cycle class map \cite{etalemotives}*{Section 7.1} to providing an environment for the conservativity conjecture \cite{ayoub-icm}*{Section 5.1}.

This article is a contribution to our understanding of the \emph{homotopical} analogue of the above \emph{homological} picture. 
We are interested in a specific question which is, in some sense, quite classical: 

\begin{question}
\label{qn:main} 
What is the difference between $\SH(S)$ and $\SH_{\et}(S)$, 
the stable motivic homotopy categories constructed in the Nisnevich and \'{e}tale topologies?
\end{question}
These categories agree upon rationalization, 
at least if $-1$ is a sum of squares (see \cite{cisinski-deglise}*{\S12.3, Corollary 16.2.14}, and the discussion in Appendix~\S\ref{et-q}),
so our question is basically one of torsion.
The main difference concerns the descent condition;
whereas the Mayer-Vietoris property with respect to distinguished Nisnevich squares is sufficient for descent in $\SH(S)$,  the more involved \'{e}tale hyperdescent condition
is required 
 in $\SH_{\et}(S)$.
An \'{e}tale hypercover $U_{\bullet}\rightarrow X$ is essentially a generalized Cech nerve construction, 
where at each stage one allows a further \'{e}tale covering, 
and the descent condition amounts to a weak equivalence between $\mathcal{F}(X)$ and the homotopy limit of the diagram $\mathcal{F}(U_{\bullet})$ \cite{MR2034012}*{Definition 4.3}.

\subsubsection{Thomason's work}     
\label{sec:thomasonwork}                                                                                                                                                                                                                   
A version of Question \ref{qn:main} was answered by Thomason in his seminal paper \cite{aktec}.
Suppose $X$ is a separated, regular, noetherian scheme of finite Krull dimension.  
For $\ell$ a prime, we say that $X$ is a \emph{$T_{\ell}$-scheme} if:
\begin{enumerate}
\item $\ell$ invertible in $X$, 
\item the residue fields of $X$ have a uniformly bounded mod-$\ell$ \'{e}tale cohomological dimension, and admit Tate-Tsen filtrations, 
\item if $\ell=2$, $\sqrt{-1} \in \H^0(X, \Oscr_X)$.
\end{enumerate}

With these technical hypotheses at hand, 
Thomason \cite{aktec}*{Theorem 2.45} proved the following 
beautiful
result for Bott periodic mod-$\ell^{\nu}$ algebraic $K$-theory.
Let $\Et_X$ denote the small \'{e}tale site on $X$ and $\Sp$ is the category of spectra.
\begin{theorem} 
\label{thm:thomasonoriginal} 
Suppose $X$ is a $T_{\ell}$-scheme and let 
$\beta \in\epsilon_*(K(-)/\ell^{\nu})$ 
be a Bott element \cite{aktec}*{A.7}. 
Then the presheaf of spectra: 
$$
K(-)/\ell^{\nu}[\beta^{-1}]: \Et_X^{\op}  
\rightarrow 
\Sp; U \mapsto K(U)/\ell^{\nu}[\beta^{-1}],
$$ 
satisfies \'{e}tale hyperdescent. 
\end{theorem}


One concrete manifestation of \'{e}tale hyperdescent is the strongly convergent descent spectral sequence for $X$ a $T_{\ell}$-scheme:
$$
H^{*}_{\et}(X, \underline{\pi}^{\et}_{*}K/\ell^{\nu}[\beta^{-1}]) \Rightarrow \pi_{*}K/\ell^{\nu}(X)[\beta^{-1}],
$$ 
which by Gabber rigidity allows for concrete computations of the target groups, 
see e.g., 
\cite[Examples 4.4, 4.8, 4.18]{aktec}.

In the language of motivic homotopy theory, 
we may reformulate Theorem~\ref{thm:thomasonoriginal} for the motivic spectrum $\KGL/\ell^{\nu}_X$ representing mod-$\ell^{\nu}$ $K$-theory over $X$\footnote{Since $X$ is regular, we note that homotopy $K$-theory agrees with Thomason-Trobaugh $K$-theory.}: 
the unit of the adjunction between stable motivic homotopy categories:
\begin{equation} \label{eq:pi-adjunction}
\epsilon^*: \SH(X) \rightleftarrows \SH_{\et}(X): \epsilon_*,
\end{equation}
furnished by \'{e}tale localization induces an isomorphism in $\SH(X)$:
$$
\KGL/\ell^{\nu}_X[\beta^{-1}] \stackrel{\simeq}{\rightarrow} \epsilon_*\epsilon^*\KGL/\ell^{\nu}_X.
$$
Phrased this way, 
motivic homotopy theory also provides a universal property of $\beta$-inverted $K$-theory as an \'{e}tale hyperdescent localization.

\subsubsection{Bott-inverted motivic cohomology} 
The story in Section \ref{sec:thomasonwork} has a parallel in motivic cohomology, considered by the third author in \cite{levine-bott}. We write $H^{p,q}_{\mot}(X,\ZZ/n)$ for mod $n$ motivic cohomology of a smooth $k$ scheme $X$.

Suppose $\ell$ is prime to the exponential characteristic of the field $k$ and that $k$ contains a primitive $\ell$th root of unity $\zeta_\ell$. Via the isomorphism $H^{0,1}_{\mot}(\Spec\,k;\ZZ/\ell)\cong \mu_\ell(k)$
we have the element $\tau_{\ell}^{\M\ZZ} \in H^{0,1}_{\mot}(\Spec\,k;\ZZ/\ell)$ corresponding to $\zeta_\ell$. Even without supposing that $\zeta_\ell$ is in $k$, one can construct  (see \S \ref{sec:bott-mz}) a family of  ``Bockstein-compatible'' generators $\tau_{\ell^{\nu}}^{\M\ZZ} \in H^{0,\e(\ell^{\nu})}_{\mot}(\Spec\,k; \ZZ/\ell^{\nu}) \simeq \ZZ/\ell^{\nu}$, 
where $\e(\ell^{\nu})$ is short for the exponent of the group $(\ZZ/\ell^{\nu})^{\times}$.  This is analogous to Thomason's construction of the Bott elements $\beta_\nu\in\pi_{n(\nu)}(K(-)/\ell^{\nu})$ for suitable $n(\nu)$; corresponding to Thomason's theorem on Bott inverted $K$-theory one has the following result for motivic cohomology.

\begin{theorem} 
\label{thm:IntroMotcoh} 
Let $k$ be a field with exponential characteristic prime to $\ell$ and let $X$ be in $\Sm_{k}$. Suppose that $\cd_{\ell}(k) < \infty$. 
For all $\nu \geq 1$, the natural map:
$$
H^{p,q}_{\mot}(X; \ZZ/\ell^{\nu}) 
\to 
H^p_{\et}(X; \mu_{\ell^{\nu}}^{\otimes q})
$$ 
induces an isomorphism
$$
H^{p,*}_{\mot}(X; \ZZ/\ell^{\nu})[(\tau_{\ell^{\nu}}^{\M\ZZ})^{-1}] 
\cong 
H^p_{\et}(X; \mu_{\ell^{\nu}}^{\otimes *}).
$$
\end{theorem}	
In case $\Char k>0$, or $\Char k=0$ and $\ell$ is odd,  or $\Char k=0$, $\ell=2$ and $\sqrt{-1}\in k$,  this is  \cite[Theorem 6.2]{levine-bott}, proved using essentially the same argument as for  Thomason's theorem. The case $\Char k=0$, $\ell=2$ and $\sqrt{-1}\not\in k$ is discussed in \S\ref{sect:bott-motcoh}.

\begin{remark}\label{rem:BK} The Milnor conjecture (the case $\ell=2$) and the Bloch-Kato conjecture (the case $\ell>2$), proven by Voevodsky, gives a finer statement about the relationship between motivic mod $\ell^\nu$ cohomology and \'etale cohomology with $\mu_{\ell^\nu}^{\otimes *}$-coefficients:
\begin{theorem}  
\label{thm:BLconjecture}
(Voevodsky \cite{voevodsky-2003, voevodsky-2011}) 
For $p \leq q$ and $X\in\Sm_{k}$, 
the \'{e}tale sheafification functor induces an isomorphism between motivic and \'{e}tale cohomology groups:
$$
H^{p,q}_{\mot}(X; \ZZ/\ell^{\nu}) 
\overset{\cong}{\rightarrow} 
H^p_{\et}(X; \mu_{\ell^{\nu}}^{\otimes q}).
$$
\end{theorem}
This directly implies Theorem~\ref{thm:IntroMotcoh}, even without the additional assumption about the finiteness of the \'etale-$\ell$-cohomological dimension. However, as we will need these extra conditions for other portions of our argument, we will not need the full power of the Bloch-Kato conjecture for our main results, and can rely on the more elementary arguments that go into the proof of Theorem~\ref{thm:IntroMotcoh}.
\end{remark} 

In motivic terms, 
we may reformulate the isomorphism of Theorem~\ref{thm:IntroMotcoh}  in a more structured way by using the motivic spectrum $\M\ZZ/\ell^{\nu}_X$ over $X$ representing mod-$\ell$ motivic cohomology.
That is, the unit of the adjunction~\eqref{eq:pi-adjunction}
induces an equivalence in $\SH(X)$:
\begin{equation}
\label{equation:BottBLiso}
\M\ZZ/\ell^{\nu}_X[(\tau_{\ell^{\nu}}^{\M\ZZ})^{-1}] 
\stackrel{\simeq}{\rightarrow} 
\epsilon_*\epsilon^*\M\ZZ/\ell^{\nu}_X.
\end{equation} 
We interpret the above equivalence as an answer to the linearized version of Question~\ref{qn:main}: 
the categories $\DM(X, \ZZ/\ell^{\nu})$ and $\DM_{\et}(X, \ZZ/\ell^{\nu})$ differ only by inversion of the element $\tau_{\ell^{\nu}}^{\M\ZZ}$. 
Indeed, this was formalized and verified by Haesemayer-Hornbostel in \cite{horn-hae}, 
at least when $X = \Spec\,k$ and for the subcategories generated by motives of $k$-varieties;
we refine their result in Theorem \ref{thm:categorical-mz} and generalize it to non-linear situations in Theorem \ref{thm:3}.

 \subsubsection{Main result} We can state a version of our main result as follows:

\begin{theorem} \label{thm:1} Let $\ell$ be a prime and $S$ be a Noetherian $\ZZ[\frac{1}{\ell}]$-scheme of finite dimension and assume that for all $x \in S$, $\cd_{\ell}(k(x)) < \infty$, and let $\E_S$ be an $\MGL_S$-module. Then, there exists an element $\tau_{\ell^{\nu}} \in \MGL/\ell^{\nu}_{0,-\e_{\MGL}(\ell^{\nu})}(S)$ such that:
\begin{enumerate}
\item The $\MGL_S$-module obtained by inverting the action of $\tau_{\ell^{\nu}}$ on $\E_S$, 
\[
\E_S/{\ell^{\nu}}[(\tau_{\ell^{\nu}})_{S}^{-1}] 
\]
satisfies \'etale hyperdescent.
\item The spectrum $\E_S/{\ell^{\nu}}[(\tau_{\ell^{\nu}})_{S}^{-1}]$ identifies with the \'etale localization of $\E_S$.
%
\end{enumerate}
\end{theorem}

%
\begin{remark}
Theorem~\ref{thm:1} has a long and interesting history starting with Thomason's work \cite{aktec}. Two complementary accounts are Mitchell's survey \cite{mitchell} and Jardine's book \cite{jardine}. The insight that one can deduce Thomason's theorem from its analogue for motivic cohomology comes from \cite{levine-motvsk}; 
it is also executed in \cite{ro1} and \cite{aro2}. 
A similar argument for algebraic cobordism over algebraically closed fields was the subject of Quick's thesis in \cite{quick}.
\end{remark}

\begin{remark} 
When $\ell =2$ and $\nu =1$, 
the mod-$2$ Moore spectrum does not have a unital multiplication so the Bott-inverted spectrum in Theorem~\ref{thm:1} has to be defined using (a motivic version) of Oka's module action of the 
mod-$4$ Moore spectrum on the mod-$2$ Moore spectrum~\eqref{equation:okapairing1}. 
This point is discussed in~\S\ref{bott-choice}.
\end{remark}

\subsubsection{Integral statement} While the above statement is interesting, an integral statement is of course preferable. The reason for doing this is more than just aesthetics. 
For example, 
such a statement would give a lift of the above theorem to the level of categories of modules over highly structured motivic ring spectra as localization usually preserves such structures, 
while passing to some homotopy cofiber does not.  

In algebraic $K$-theory, 
this was achieved by localizing algebraic $K$-theory spectra at the topological $K$-theory spectrum; 
see \cite{aktec}*{Theorem 2.50}. 
For algebraic cobordism over the complex numbers, 
shades of an integral statement appeared first in Heller's work \cite{heller} on semitopological cobordism. 
The right thing to consider turns out to be the completion (in the sense of, say, \cite{mnn}*{Section 2}) at the \'{e}tale motivic cohomology spectrum $\M\ZZ^{\et} := \epsilon_*\epsilon^*\M\ZZ$.
If $J$ is a set of primes, we denote by $\E_{(J)}$ the localization of $\E$ at $J$.

\begin{theorem} 
\label{thm:2} Under the same assumptions as in Theorem~\ref{thm:1}, the completion of $\E$ at the \'etale motivic cohomology spectrum $\M\ZZ^{\et}$ coincides with the \'etale localization of $\E$ up to localizing at the primes which are invertible in $S$.

%
\end{theorem}
%

Denote by $L_{\M\ZZ^{\et}}$ the left adjoint to the inclusion of $\M\ZZ^{\et}$-local objects, i.e., a Bousfield localization functor. Theorem \ref{thm:2} helps us deduce our main result for modules over highly structured ring spectra using the language of stable $\infty$-categories:
\begin{theorem}  
\label{thm:3}   
Under the 
hypotheses
of Theorem~\ref{thm:2}, assume further that $\E$ is a motivic $\mathcal{E}_{\infty}$ ring spectrum (such as $\MGL$ itself).
Then there are equivalences of stable $\infty$-categories: 
$$
\Mod_{L_{\M\ZZ^{\et}}\E_{(J)}} 
\simeq 
\Mod_{\epsilon_*\epsilon^*\E_{(J)}}
\simeq 
\Mod^{\et}_{\E_{(J)}}.
$$
\end{theorem}
Here, by a motivic $\mathcal{E}_{\infty}$-ring spectrum we mean a commutative algebra object of the symmetric monoidal $\infty$-category of motivic spectra over a Noetherian scheme of finite dimension $S$. Hence Theorem~\ref{thm:2} gives a description of the \'etale-local subcategory of $\E_{(J)}$-modules as modules over a certain motivic ring spectrum.

\subsubsection{Functoriality of \'etale algebraic cobordism} Another application concerns the functoriality of \'etale algebraic cobordism. If we let $\Sch^{(J)}$ be the full subcategory of schemes for which the 
hypotheses
 of Theorem~\ref{thm:2} 
 hold
 then, restricted to this subcategory of schemes, we show that $\MGL_{(J)}^{\et}$ forms a Cartesian section of $\SH$. The point is that $f^*$ and $\epsilon_*$ have no \emph{a priori} reason to commute, while $f^*$ commutes with colimits and Bott elements (by construction). Therefore, $f^*$ preserves Bott-inverted algebraic cobordism. In conjunction with Theorem~\ref{thm:3}, we obtain a six functors formalism for $\Mod^{\et}_{\MGL_{(J)}}$.

\begin{corollary} \label{thm:4} The functor
\begin{equation} \label{thm:4}
\Mod^{\et}_{\MGL_{(J)}}:(\Sch^{(J)})^{\op} \rightarrow \Cat_{\infty},
\end{equation}
satisfies the full six functors formalism in the sense of \cite{ayoub}, \cite{cisinski-deglise}. In particular, $\Mod^{\et}_{\MGL_{(J)}}$ satisfies localization: for any closed immersion $i:Z \hookrightarrow X$ with open complement $j: U \rightarrow X$ and $\M \in \Mod^{\et}_{\MGL_{(J)}}$, then we have a cofiber sequences
\[
j_!j^*\M \rightarrow \M \rightarrow i_!i^*\M,
\]
and
\[
i_*i^!\M \rightarrow \M \rightarrow j_*j^*\M.
\]
\end{corollary}

\subsubsection{Subsequent work} While this paper was under revision, the first and last author, in a joint project with Bachmann, have managed to obtain a stronger \'etale descent result for the motivic sphere spectrum in \cite{beo}. While the general ideas are very similar --- via reduction to the case of motivic cohomology --- the execution is rather different. In this paper, our technical ingredients are 1) a study of the convergence of localized spectral sequences and 2) a thorough study of \'etale localization in the context of oriented theories. We believe that these techniques and the results they yield are of independent interest to motivic homotopy theory as well. 

\subsubsection{Organization} 
Using the language of $\infty$-categories we start Section~\ref{sec:preliminaries} by reviewing motivic homotopy theory with respect to some topology $\tau$.
We recall the adjunctions comparing motivic homotopy theories across different topologies along with Voevodsky's category of motives. 
In Section~\ref{sect:slice}, we introduce the ``slice comparison paradigm" which is a factorization of the unit map of the adjunction comparing different topologies. We recall, in a slightly more general setting, how to associate to $\E \in \SH(S)$ the filtered motivic spectra $\{ f_q\E \}_{q \in \ZZ}$ of its slice covers. 
The ``slice comparison paradigm" then takes the form:
$$
\eta: 
\colim f_q\E \simeq \E  
\stackrel{\eta_{\infty}}{\longrightarrow} 
\colim \epsilon_*\epsilon^*f_q\E 
\stackrel{\eta^{\infty}}{\longrightarrow} 
\epsilon_*\epsilon^*\E.
$$
We explain conditions for when $\eta^{\infty}$ and $\eta_{\infty}$ can be turned into equivalences. This is addressed in Sections~\ref{sect:some-conn} and~\ref{sect:commuting}. In more details, Section~\ref{sect:some-conn} studies the effect $\epsilon_*\epsilon^*$ has on connectivity of motivic spectra and thus addresses the convergence properties of the spectral sequence arising from the tower of spectra $\{\epsilon_*\epsilon^*f_q\E \}.$ These connectivity results are of independent interest. In Section~\ref{sect:commuting}, we give conditions for $\eta^{\infty}$ to be an equivalence, which boils down to understanding situations where the $\epsilon_*$ functor commutes past colimits. In Section~\ref{proof-main} we apply this paradigm to prove Theorem \ref{thm:1}. 
After dealing with the motivic cohomology spectrum $\M\ZZ$ we turn to algebraic cobordism $\MGL$, and finally arbitrary Landweber exact theories. 
The main point is that the result for $\M\ZZ$ proves the theorem on the level of slices and the ``slice comparison paradigm" bridges the gap between the slices and the spectra of interest. 
In Section~\ref{sect:applications} we give applications of our main theorems and promote Theorem~\ref{thm:1} to the level of module categories as in Theorem~\ref{thm:2}. We also prove that \'etale local algebraic cobordism defines a Cartesian section of $\SH$ and thus the category of modules over it satisfy the six functors formalism. In the first appendix (Section~\ref{appendix:invert}) we review the basics of periodization (after \cite{hoyois-cdh}) and localization in suitable $\infty$-categories,  the short second appendix (Section~\ref{appendix:e-locals}) records a well-known result about $\E$-locality, while the third appendix (Section~\ref{mgl-trsfs}) constructs transfers for $\MGL$ necessary to pick out Bott elements over general fields.
 
\subsubsection{Conventions}
Categorical terminologies are to be interpreted in the $\infty$-categorical context;
``functor" always means a functor of $\infty$-categories (i.e., a morphism of quasicategories).
We freely use the notions of algebras and modules in higher algebra \cite{higheralgebra}.
\begin{itemize}
\item We use the following 
vocabulary
of higher category theory:
\begin{itemize}
\item $\Spc$ is the $\infty$-category of $\infty$-groupoids (Kan complexes give a concrete model).
\item $\Maps(X,Y)$ is the Kan complex of maps between objects $X, Y \in \C$ and $\mapint(X,Y) \in \C$ is the internal mapping object.
\item $[X,Y] := \pi_0 \Maps(X, Y)$ denotes homotopy classes of maps between $X$, $Y\in \C$.
\item $\Pr^{L, \otimes}$ ---the symmetric monoidal $\infty$-category of presentable $\infty$-categories and colimit preserving functors ---
has a full subcategory $\Pr^L_{\stab}$ of \emph{stable} presentable $\infty$-categories.
\end{itemize}
\item If $p: X \rightarrow S$ is a smooth $S$-scheme, 
we always mean that $X$ is of \emph{finite type}. 
An essentially smooth scheme over $k$ is an inverse limit of smooth $S$-schemes with affine and dominant transition maps.
\item If $F$ is a presheaf of abelian groups we denote its $\tau$-sheafification by $a_{\tau}(F)$, 
$\tau$ any topology.
\item Some conventions in motivic homotopy theory:
\begin{itemize}
\item We denote unstable motivic spheres by $\s^{p,q} := (S^{1})^{p-q} \wedge \GG_m^q$ for $p\geq q$, 
the motivic sphere spectrum by $\sspt$,
and motivic suspensions by $\Sigma^{p,q}$. 
\item For a motivic spectrum $\E$ over a base scheme $S$ we write $\E \in \SH(S)$, or $\E_S$ for emphasis. If the base is clear in the context, we will sometimes drop the subscript and simply write ``$\E$".
\end{itemize}
\end{itemize}

\subsubsection{Acknowledgements} 
Work on this paper took place at the Institut Mittag-Leffler in Djursholm, the Hausdorff Research Institute for Mathematics in Bonn and the Centre for Advanced Study at the Norwegian Academy of Science and Letters in Oslo which hosted the project ``Motivic Geometry" during the 2020/21 academic year;
we thank all institutions for providing excellent working conditions, hospitality, and support.
Oliver R\"{o}ndigs has been a helpful interlocutor offering advice on this work, and Glen Wilson provided very helpful spectral sequence guidance.
Elmanto thanks Marc Hoyois for discussions on $\infty$-categories, 
and Mike Hill for explaining the possible apocalyptic effects of inverting elements in a spectral sequence and why the ones considered in this paper stick in dry land. 
He is very grateful to his adviser John Francis ---techniques from discussions with him throughout the years have made their way into this paper. We would like to thank Niko Naumann for helpful comments on the first draft of this paper and Peter Haine for help with some references. We would like to especially thank Tom Bachmann and Denis-Charles Cisinski for comments that substantially improved some arguments in this paper, and for patiently answering our questions. Last but certainly not least, the referee has kindly offered an extensive list of comments which improved the readability and the mathematical content of the paper and we are very grateful to them.
Elmanto was supported by a Mittag-Leffler postdoctoral fellowship grant from Varg stiftelsen. Spitzweck was supported by DFG grants ``Applications of Motivic Filtrations" LE 2259/7-2
and ``Operads in Algebraic Geometry and their Realisations", both within the SPP 1786. \O stv\ae r was supported by a Friedrich Wilhelm Bessel Research Award from the Humboldt Foundation and the RCN Frontier Research Group Project no.~250399 ``Motivic Hopf Equations." Spitzweck and \O stv\ae r were
also supported by the RCN project no. 312472 "Equations in Motivic Homotopy
Theory." Levine was supported by the DFG via grant ``Applications of Motivic Filtrations" LE 2259/7-2 within the SPP 1786 and from the ERC project QUADAG. This paper is part of a project that has received funding from the European Research Council (ERC) under the European Union's Horizon 2020 research and innovation programme (grant agreement No. 832833).

\includegraphics[scale = 0.1]{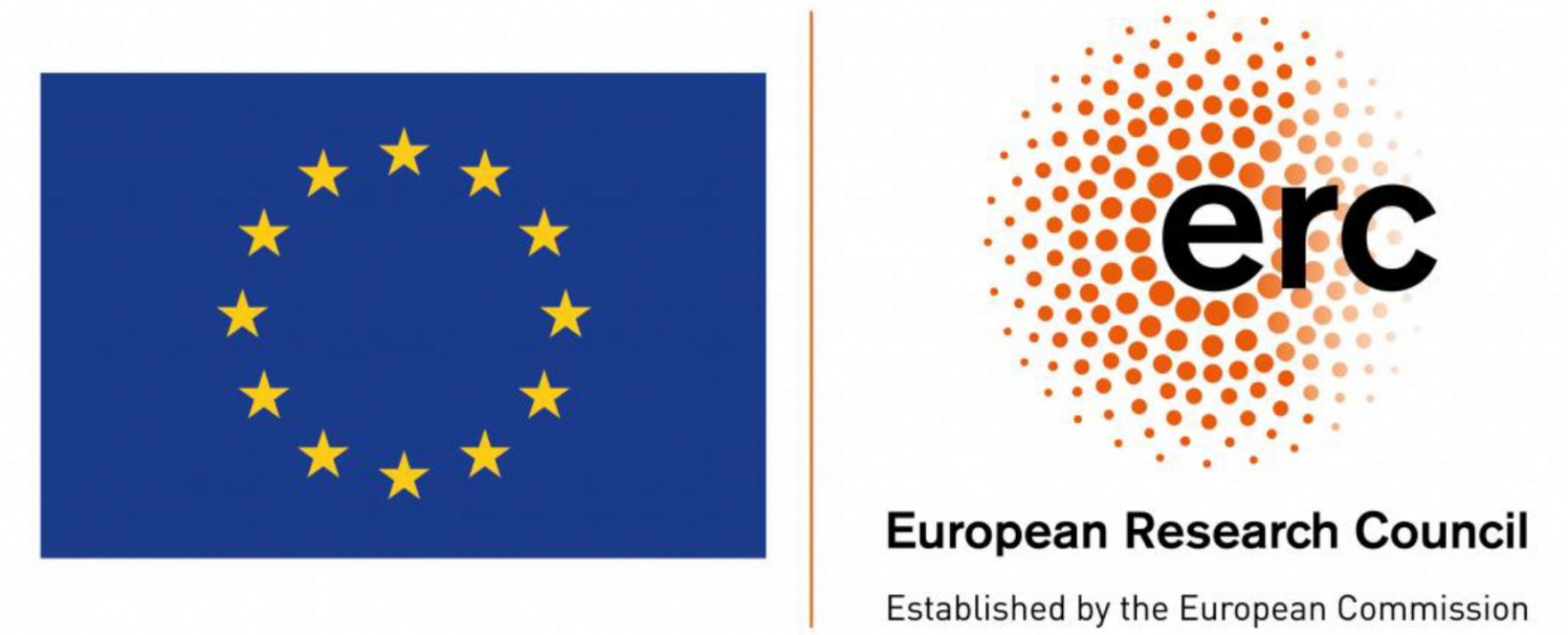}

\section{Preliminaries} 
\label{sec:preliminaries}
In this section, 
we review the basics of motivic homotopy theory ---
mainly setting up notation and checking certain statements in the generality of arbitrary topologies; 
the reader familiar with this theory may feel free to skip this section and return as necessary. 
In Section~\ref{sect:SHtau} we review the construction of the presentably symmetric monoidal stable $\infty$-category of motivic spectra in the generality of arbitrary Grothendieck topologies. 
In Section~\ref{sect:compare} we compare motivic homotopy categories across different topologies, 
and in Section~\ref{sect:motives} we review Voevodsky motives in our language.
For the following,  
$S$ is a quasicompact quasiseparated (qcqs) base scheme.

\subsection{$\tau$-motivic homotopy theory} 
\label{sect:SHtau}
We first construct the presentably symmetric monoidal $\infty$-category of \emph{$\tau$-motivic spaces}, $\H_{\tau}(S)$. 
For our purposes, 
this $\infty$-category has \emph{$\tau$-hyperdescent} built into it as we will explain later.
We start with the (discrete) category $\Sm_S$ of smooth $S$-schemes equipped with a Grothendieck topology $\tau$. 
Let $\P(\Sm_S)$ be the $\infty$-category of presheaves over $S$;
it is obtained from $\Sm_S$ by freely adjoining arbitrary small colimits \cite{htt}*{Theorem 5.1.5.6}.
We denote by $j: \Sm_S \rightarrow \P(\Sm_S)$ the Yoneda embedding.

\subsubsection{}
\label{sec:pshv} 
Recall that $F \in \P(\Sm_S)$ is called \emph{homotopy invariant} or \emph{$\AA^{1}$-invariant} if the projection induces an equivalence $F(X) \rightarrow F(X \times_S \AA^{1})$. 
We denote by 
$
\P_{\AA^1}(\Sm_S)
$ 
the $\infty$-category of homotopy invariant presheaves. 
The embedding $\P_{\AA^1}(\Sm_S) \hookrightarrow \P(\Sm_S)$ preserves limits, 
and thus admits a left adjoint:
$$
\LL_{\AA^1}: \P(\Sm_S) \rightarrow \P_{\AA^1}(\Sm_S),
$$ 
which witnesses $\P_{\AA^1}(\Sm_S)$ as an accessible localization \cite{htt}*{Proposition 5.2.7.4}. 
A map $f: F \rightarrow G$ in $\P(\Sm_S)$ is an \emph{$\AA^{1}$-weak equivalence} if $\LL_{\AA^1}f: \LL_{\AA^1}F \rightarrow \LL_{\AA^1}G$ is an equivalence.

\subsubsection{} \label{sect:hypersheaves}
One difference between doing motivic homotopy theory with an arbitrary topology $\tau$ and the usual Nisnevich localization is the choice of whether or not to hypercomplete; 
recall that if $S$ is a base scheme which is quasi-compact and has finite Krull dimension then the Nisnevich $\infty$-topos is automatically hypercomplete (see \cite{clausen-mathew}*{Theorem 3.17} for a general version of this statement). 
In the main body of this paper, 
we work with the \emph{hypercompletion} of the \'etale topology, 
mainly because we want to obtain spectral sequences with convergent properties to calculate \'etale versions of motivic Landweber exact theories as in \S\ref{sect:hyp-ss}. However, as we will prove in the appendix, we can remove this hypothesis in \S\ref{sect:hyp-remove}.

We consider the full subcategory $\P_{\tau, \hyp}(\Sm_S) \hookrightarrow \P(\Sm_S)$ spanned by \emph{$\tau$-hypersheaves} (also called \emph{$\tau$-local}), i.e., those presheaves $F$ such that for any $\tau$-hypercover $U_{\bullet} \rightarrow U$, 
the map $F(U) \rightarrow \lim_{\Delta} F(U_{\bullet})$ is an equivalence.
As above, 
the embedding $\P_{\tau, \hyp}(\Sm_S) \hookrightarrow \P(\Sm_S)$ preserves limits and admits a left adjoint:
$$
\LL_{\tau}: \P(\Sm_S) \rightarrow \P_{\tau,\hyp}(\Sm_S),
$$ 
which witnesses $\P_{\AA^1}(\Sm_S)$ as an accessible localization. 
A map $f: F \rightarrow G$ is a \emph{$\tau$-weak equivalence} if the induced map $\LL_{\tau}f: \LL_{\tau}F \rightarrow \LL_{\tau}G$ is an equivalence.

For the purposes of this paper, we work with $\tau$-hypersheaves. However, we could also consider the $\infty$-category $\P_{\tau}(\Sm_S)$ of \emph{$\tau$-sheaves}, 
see \cite{htt}*{Section 6.2.2}, 
and its localization functor $\LL'_{\tau}: \P(\Sm_S) \rightarrow \P_{\tau}(\Sm_S)$. 
The $\infty$-category $\P_{\tau, \hyp}(\Sm_S)$ is obtained as a further accessible localization by \cite{htt}*{Theorem 6.5.3.13}  (usually called \emph{hypercompletion}):
$$
\widehat{(-)}: \P_{\tau}(\Sm_S) \rightarrow \P_{\tau,\hyp}(\Sm_S).
$$ 
Note that this is the underlying $\infty$-category of the Brown-Joyal-Jardine model structure on simplicial presheaves \cite{htt}*{Proposition 6.5.2.14} and $\tau$-hypersheaves 
in our sense are equivalent to hypercomplete objects in the sense of \cite{htt} by \cite{htt}*{Theorem 6.5.3.12}.

\subsubsection{}
\label{sec:ltau}
The $\infty$-category $\H_{\tau}(S)$ of \emph{$\tau$-motivic spaces} is the full subcategory of $\P(\Sm_S)$ spanned by homotopy invariant $\tau$-hypersheaves. 
Combining the localizations above, 
it is characterized by a universal property whose proof follows the one for $\tau= \Nis$ verbatim:

\begin{proposition}  
\label{prop:univprop-unstab} 
\label{prop:uniproph} (\cite{robalo}*{Theorem 2.30}, \cite{robalo}*{Remark 2.31}) 
For any $\infty$-category $\C$ with small colimits, 
the composite of localizations $\LL_{\AA^1} \LL_{\tau}: \P(\Sm_S) \rightarrow \P_{\tau,\hyp}(\Sm_S) \rightarrow \H_{\tau}(S)$ induces a fully faithful functor 
$\Fun^L(\H_{\tau}(S), \C) \rightarrow \Fun(\Sm_S, \C)$, 
whose essential image is spanned by functors which are homotopy invariant and satisfy $\tau$-hyperdescent. 
Furthermore, 
the Cartesian monoidal structure on $\P(\Sm_S)$ descends to a Cartesian monoidal structure on $\H_{\tau}(S)$, 
and the functor:
$$
\LL_{\AA^1} \circ \LL_{\tau, \hyp} \circ j: \Sm_S^{\times} \rightarrow \P(\Sm_S)^{\times} \rightarrow \P_{\tau,\hyp}(\Sm_S)^{\times} \rightarrow \H_{\tau}(S)^{\times},
$$ 
is monoidal.
\end{proposition}
Let $\LL_{\tau,\mot}: \P(\Sm_S) \rightarrow \H_{\tau}(S)$ be short for the \emph{motivic localization} functor $\LL_{\AA^1}\LL_{\tau}$. 
If $X \in \Sm_S$ we write $\LL_{\tau,\mot}(X)$ or, if the context is clear, simply $X$ for the motivic localization $\LL_{\tau,\mot}(j(X))$ of $X$ under the Yoneda embedding.

\subsubsection{} 
From here on, we impose the following assumption on $\tau$:
\begin{itemize}
\item A $\tau$-hypersheaf takes coproducts to products, 
i.e., 
for any two $S$-schemes $X$, $Y$, 
the natural map $F(X \coprod Y) \rightarrow F(X) \times F(Y)$ is an equivalence.
\end{itemize}
The Zariski (and hence Nisnevich) topology satisfies the above condition (use the distinguished square obtained by the cover $\{X \rightarrow X \coprod Y, Y \rightarrow X \coprod Y\}$), 
and thus also all finer topologies. 
With this assumption, 
we see that $\P_{\tau, \hyp}(\Sm_S) \subset \P_{\Sigma}(\Sm_S)$, 
where $\P_{\Sigma}(\Sm_S)$ is the full subcategory of presheaves which transforms coproducts into products. 
In \cite{htt}*{Section 5.5.5} this is called the \emph{nonabelian derived category}, 
while (in the context of model categories) these are called \emph{simplicial radditive functors} in \cite{voe10b}. 
The universal property of $\P_{\Sigma}(\Sm_S)$ is that it is freely generated by the Yoneda image of $\Sm_S$ under sifted colimits and furthermore an object in the Yoneda image is compact 
\cite{htt}*{Proposition 5.5.8.10}. 
In general, 
$\P_{\Sigma}(\C)$ is \emph{not} the category of sheaves for some Grothendieck topology \cite{voe10b}*{Example 3.7}.

\begin{proposition} 
\label{prop:generation} 
\begin{itemize}
\item The $\infty$-category $\H_{\tau}(S)$ is generated under sifted colimits by smooth  $S$-schemes. 
\item If $\tau$ is at least as fine as the Zariski topology, then the smooth $S$-schemes which are affine generate $\H_{\tau}(S)$.
\item If the inclusion $\H_{\tau}(S) \hookrightarrow \P_{\Sigma}(\Sm_S)$ preserves filtered colimits, 
then the generators are compact.
\end{itemize}
\end{proposition}
\begin{proof} 
By the above discussion, we see that since $\H_{\tau}(S) \subset \P_{\Sigma}(\Sm_S),$ it is generated by $\LL_{\tau,\mot}(X)$ under sifted colimits, 
where $X$ is a smooth $S$-scheme. 
The second statement follows by the argument in \cite{adeel}*{Lemma 4.3.5}. In more detail, we 
first note that, by definition:
$$
\Sm_S \rightarrow \H_{\tau}(S);
X \mapsto \LL_{\tau,\mot}X,
$$ 
is a $\tau$-cosheaf. 
Thus if $\tau$ is at least as fine as the Zariski topology we may assume that $X$ is separated by picking a Zariski cover of $X$ consisting of affine schemes, 
where the pairwise intersections are all separated. 
We may further pick a Zariski cover of $X$ where the pairwise intersections are affine. 
Applying $\LL_{\tau,\mot}$ to the nerve of this cover produces a simplicial object in $\P_{\Sigma}(\Sm_S)$ with colimit $\LL_{\tau,\mot}X$ (its terms are $\LL_{\tau,\mot}$ applied to affine schemes).
The last statement is immediate.
\end{proof}

\begin{remark} 
The inclusion $\H_{\tau}(S) \hookrightarrow \P_{\Sigma}(\Sm_S)$ need not preserve filtered colimits, 
and hence $\LL_{\tau,\mot}(X)$ may be noncompact. 
But $\H_{\Nis}(S) \hookrightarrow \P(\Sm_S)$ does preserve filtered colimits. 
Indeed, 
being a $\Nis$-hypersheaf is the same as being Nisnevich excisive.
Thus we need to check that if $F_i$ is a filtered diagram of Nisnevich excisive presheaves, 
then the colimit $\colim F_i$ taken in $\P_{\Sigma}(\Sm_S)$ is Nisnevich excisive. 
If $Q$ is an elementary distinguished square, 
then taking the pointwise colimit in $\P_{\Sigma}(\Sm_S)$ we see that $(\colim_i F)(Q)$ is a Cartesian square since colimits commute pullbacks in $\Spc$; 
so $\colim F_i$ is Nisnevich excisive. 
\end{remark}

\subsubsection{} 
We may consider the $\infty$-category of pointed objects $\H_{\tau,\bullet}(S) := \H(S)_{\star/}$ for the final object $\star$ of $\H_{\tau}(S)$.
It is also presentable since $\H_{\tau, \bullet}(S) \simeq \H_{\tau}(S) \otimes \Spc_{\bullet}$, 
where the $\otimes$-product is taken in $\Pr^L$. 
There is an adjunction: $$(-)_{+}: \H_{\tau}(S) \rightleftarrows \H_{\tau, \bullet}(S): u_{\E}.$$
The symmetric monoidal structure on $\H_{\tau}(S)$ extends to a monoidal structure on $\H_{\tau, \bullet}(S)$ so that $(-)_{+}: \H_{\tau}(S) \rightarrow \H_{\tau, \bullet}(S)$ is a monoidal functor.

\subsubsection{} 
Let $\CAlg(\Pr^L)$ be the $\infty$-category of presentably symmetric monoidal $\infty$-categories and symmetric monoidal colimit preserving functors. 
To $\C^{\otimes} \in \CAlg(\Pr^L)$ we can associate the $\infty$-category of $\C^{\otimes}$-modules $\Mod_{\C}:=\Mod_{\C}(\Pr^L)$ and $\C^{\otimes}$-algebras $\CAlg_{\C}:=\CAlg(\Mod_{\C})$.
The latter can be identified with $\CAlg(\Pr^L)_{\C/}$ \cite{higheralgebra}*{Corollary 3.4.1.7}.
Hence a $\C^{\otimes}$-algebra object in $\C^{\otimes}$-modules is the data of a presentably symmetric monoidal category $\D$ admitting a colimit preserving symmetric monoidal functor $F:\C^{\otimes} \rightarrow \D^{\otimes}$. 
If $X \in \C^{\otimes}$, 
we let $X$ ``act" on $\D$ via the functor $F$. 
Let $\CAlg_{\C}[X^{-1}]$ denote the full subcategory of $\CAlg_{\C}$ spanned by those $\C^{\otimes}$-algebras on which $X$ acts invertibly. 

\begin{theorem} 
\label{thm:robalo-stab1} (Robalo, \cite{robalo}*{Proposition 2.9}) 
Let $\C^{\otimes} \in \CAlg(\Pr^L)$ and $X \in \C^{\otimes}$ be an object. 
\begin{itemize}
\item There exists a functor:
$$
L_X: \CAlg_{\C} \rightarrow \CAlg_{\C}[X^{-1}],
$$ 
which witnesses $\CAlg_{\C}[X^{-1}]$ as a localization of $\CAlg_{\C}$. 
Thus for any $\D^{\otimes} \in \CAlg_{\C}$ there is a canonical map $\D^{\otimes} \rightarrow L_X(\D^{\otimes})$ in $\CAlg_{\C}$ sending $X$ to an invertible object.
\item The functor $L_X$ fits into the commutative diagram:
\begin{equation}
\xymatrix{ 
\CAlg_{\C} \ar[rr]^-{L_X} \ar[d]_-{U} &  & \CAlg_{\C}[X^{-1}] \ar[d]^-{U} \\ 
\Mod_{\C} \ar[rr]^-{-\otimes_{\C^{\otimes}} \C^{\otimes}[X^{-1}]} &  &\Mod_{\C^{\otimes}[X^{-1}]},
}
\end{equation}
where the horizontal arrows are forgetful functors. 
Thus if $\D^{\otimes} \in \CAlg_{\C}$ the underlying $\C^{\otimes}$-module of $L_X(\D^{\otimes})$ is equivalent to $\D^{\otimes} \otimes_{\C^{\otimes}} \C^{\otimes}[X^{-1}]$.
\end{itemize}
\end{theorem}
In light of the second point above, 
we write $L_X(\D^{\otimes}):=\D[X^{-1}]$.

\subsubsection{} 
Recall that an object $X \in \C^{\otimes}$ is \emph{$n$-symmetric} if there is an invertible two cell that witnesses an equivalence between the cyclic permutation on 
$X \otimes X \otimes \cdots \otimes X$ and the identity (see \cite{robalo}*{Definition 2.16}). 
For $\M \in \Mod_{\C}$ we set:
\begin{equation} \label{eq:colim-stab}
\Stab_X(\M) := \colim  \M \stackrel{-\otimes X}{\longrightarrow} \M \stackrel{-\otimes X}{\longrightarrow} \M \stackrel{-\otimes X}{\longrightarrow} \cdots. 
\end{equation}
\begin{corollary} 
\label{thm:robalo-stab2} 
(Robalo, \cite{robalo}*{Corollary 2.22}]) 
Let $\C^{\otimes} \in \CAlg(\Pr^L)$ and let $X$ be an $n$-symmetric object in $\C$. 
Then for any $\M \in \Mod_C$ there is a natural equivalence:
$$
\M[X^{-1}] \rightarrow \Stab_X(M).
$$ 
\end{corollary}

\subsubsection{}
\label{sec:stab}
$\H_{\tau}(S)^{\times}_{\bullet} \in \CAlg(\Pr^L)$ and we invert the $3$-symmetric object $(\PP^{1}, \infty) \in \H_{\tau}(S)_{\bullet}$ \cite{voevodsky-icm}*{Lemma 4.4} to form:
$$
\SH_{\tau}(S):=
\H_{\tau}(S)[ (\PP^{1},\infty)^{-1}].
$$ 
The sum of our discussion is the following universal property of $\SH_{\tau}(S)$.
\begin{theorem} 
\label{thm:univprop-stab} (Robalo, \cite{robalo}*{Corollary 2.39}) 
Let $S$ be a qcqs base scheme and let:
$$
\theta:\Sm_S^{\times} \rightarrow \SH_{\tau}(S)^{\otimes},
$$ 
be the symmetric monoidal functor obtained from $\LL_{(\PP^{1},\infty)} \circ (-)_{+} \circ \LL_{\tau,\mot} \circ j$.
Then for $\D \in \CAlg(\Pr^L)$, $\theta$ induces a fully faithful functor:
$$
\Fun^{\otimes, L}(\SH_{\tau}(S), \D) \rightarrow \Fun^{\otimes}(\Sm_S, \D),
$$ 
whose essential image is spanned by those symmetric monoidal functors $F: \Sm_S \rightarrow \D$ which satisfies $\tau$-hyperdescent, 
homotopy invariance, 
and such that the cofiber of $F(S) \rightarrow F(\PP^{1})$ induced by the $\infty$-section acts invertibly.
\end{theorem}
Unwinding the above theorem we get the standard adjunction: 
\begin{equation} \label{eqn:stab-adjunction}
\Sigma^{\infty}_{T,+}
: 
\H_{\tau}(S) 
\leftrightarrows 
\SH_{\tau}(S)
: 
\Omega_T^{\infty}.
\end{equation}

\subsubsection{} The stable analogue of Proposition~\ref{prop:generation} is:
\begin{proposition} 
\label{prop:gen-stab} 
\begin{itemize}
\item The presentably symmetric monoidal stable $\infty$-category $\SH_{\tau}(S)$ is generated under sifted colimits by $\{\Sigma^{-2n,-n}\Sigma^{\infty}_T X_+\}_{n \in \ZZ}$, 
where $X$ is a smooth $S$-scheme.  
\item If $\tau$ is at least as fine as the Zariski topology, then the smooth $S$-schemes which are affine generate $\SH_{\tau}(S)$ under sifted colimits.
\item If the inclusion $\H_{\tau}(S) \hookrightarrow \P_{\Sigma}(\Sm_S)$ preserves filtered colimits,
then the generators are compact.
\end{itemize}
\end{proposition}
\begin{proof} 
According to \cite{robalo}*{Proposition 2.19}, the symmetric monoidal stable $\infty$-category $\SH_{\tau}(S)$ is calculated as the colimit in $\Pr^{L, \otimes}$ of the diagram:
$$
\H_{\tau,\bullet}(S) \stackrel{- \wedge (\PP^{1}, \infty) \simeq \Sigma^{2,1}}\longrightarrow \H_{\tau,\bullet}(S) \stackrel{- \wedge (\PP^{1}, \infty) \simeq \Sigma^{2,1}(-)}\longrightarrow \cdots,
$$
and is thus generated under filtered colimits by $\Sigma^{-2n,-n} \Sigma^{\infty}X$ where $X \in \H_{\bullet}(S)$ from the formula for filtered colimits of presentable categories \cite{htt}*{Lemma 6.3.3.6}. 
The adjunction $\H_{\tau}(S)\rightleftarrows \H(S)_{\tau,\bullet}$ is monadic and in particular any object in $\H(S)_{\tau,\bullet}$ can be written as a sifted colimit of objects in $\H_{\tau}(S)$. 
This implies the first statement since the $\infty$-category $\H_{\tau}(S)$ is generated by smooth $S$-schemes. 
The next two statements follows immediately from their analogues in Proposition~\ref{prop:generation}.
\end{proof}
For a general topology $\tau$, suspension spectra of schemes need not be compact in the category $\SH_{\tau}(S)$. In fact its unit can already be noncompact, as the following example illustrates.
\begin{example} 
\label{ex:r} 
We claim the sphere spectrum is noncompact in $\SH_{\et}(\RR)$.
Suppose for contradiction $\sspt_{\RR}$ is compact and let $\M_i$ be a countable collection of objects in $\DM_{\et}(\RR;\ZZ/2)$ viewed as objects of $\SH(\RR)$ via 
$u_{\tr}: \DM_{\et}(\RR; \ZZ/2) \rightarrow \SH_{\et}(\RR)$. 
Then $\Maps(\sspt_{\RR}, \oplus u_{\tr}(\M_i)) \simeq \oplus \Maps(\sspt_{\RR}, u_{\tr}(\M_i))$, 
which cannot be the case as explained in~\cite{etalemotives}*{Remark 5.4.10} since $\cd_{2}(\RR)=\infty$ 
(the remark uses $\DM_h(\RR;\ZZ/2)$ which agrees with $\DM_{\et}(\RR;\ZZ/2)$ after \cite{etalemotives}*{Corollary 5.5.5}).
\end{example}
\subsubsection{}
\label{sect:hptysheaves1}
If $\E \in \SH_{\tau}(S)$ and $p,q \in \ZZ$, 
the \emph{homotopy presheaf} $\underline{\pi}^{\pre}_{p,q}(\E)$ on $\Sm_S$ is defined by:
$$
U 
\mapsto 
[\Sigma^{p,q} \Sigma^{\infty}_{T}U_+,\E ]_{\SH_{\tau}(S)}.
$$ 
The \emph{$\tau$-homotopy sheaf} of $\E$ is the $\tau$-sheafification $\underline{\pi}^{\tau}_{p,q}(\E) := a_{\tau}( \underline{\pi}^{\pre}_{p,q}(\E))$. 
When $\tau=\Nis$ we refer to $\underline{\pi}^{\tau}_{p,q}(\E)$ as the homotopy sheaf.

\subsection{Comparing motivic categories}
\label{sect:compare}

\subsubsection{} \label{subcanonical}
Let $S$ be a qcqs scheme and consider two topologies $\sigma, \tau$ on $\Sm_S$ such that $\tau$ is finer than $\sigma$ (we write $\sigma \leq \tau$ to signify this); 
the examples that concern us are $\sigma = \Nis, \tau = \et$. 
The identity functor $id: \Sm_S \rightarrow \Sm_S$ induces a geometric morphism of $\infty$-topoi $\epsilon_*: \P_{\tau, \hyp}(\Sm_S) \rightarrow \P_{\sigma, \hyp}(\Sm_S)$; 
that is, 
$\epsilon_*$ admits a left adjoint, denoted by $\epsilon^*$, that preserves finite limits \cite{htt}*{Definition 6.3.1.1}. 
The functor $\epsilon_*$ is obtained by precomposition $F \mapsto F \circ id$ while $\epsilon^*$ is obtained as localization at $\tau$-hypercovers. 
This functor can be computed via the left Kan extension:
$$
\xymatrix{
\Sm_S \ar[d]_{y_{\sigma}} \ar[r]^-{y_{\tau}}   & \P_{\tau,\hyp}(\Sm_S) \\
\P_{\sigma,\hyp}(\Sm_S), \ar@{-->}[ur]_{\epsilon^*} 
}$$
where $y_{\sigma}$ and $y_{\tau}$ are the composite of the Yoneda embedding with $\sigma$ and $\tau$-sheafifications respectively. 
Hence, if $\sigma, \tau$ are subcanonical, $\epsilon^*$ preserves representable sheaves. 
Since $\epsilon^*$ preserves finite products, 
and the categories are equipped with the Cartesian monoidal structures, 
it is symmetric monoidal.

\subsubsection{} 
We note that the adjunction above descends to the motivic categories.
\begin{lemma} 
\label{lem:a1preserve} 
In the adjunction: 
$$
\epsilon^*:  
\P_{\sigma, \hyp}(\Sm_S)  
\rightleftarrows 
\P_{\tau, \hyp}(\Sm_S)
:\epsilon_*,
$$ 
$\epsilon^*$ preserves $\AA^{1}$-weak equivalences and $\epsilon_*$ preserves homotopy invariant objects. 
\end{lemma}
\begin{proof} 
If $F \in  \P_{\tau, \hyp}(\Sm_S)$ is homotopy invariant, 
$\epsilon_*(F)(X \times \AA^{1}) \simeq F(X \times \AA^{1}) \simeq F(X)$ for $X \in \Sm_S$, 
so $\epsilon_*$ preserves homotopy invariant objects. 
To show $\epsilon^*$ preserves $\AA^{1}$-weak equivalences, we note that by general considerations (\cite{htt}*{Proposition 5.5.4.15}), the class of $\AA^1$-weak equivalences is the strong saturation of the maps $\{ X \times \AA^1 \rightarrow X \}$. By \cite{bachmann-hoyois}*{Lemma 2.10} and the fact that $\epsilon^*$ preserves colimits, it suffices to prove that $\epsilon^*(X \times \AA^1 \rightarrow X)$ is an $\AA^1$-weak equivalence, which follows from the fact that $\epsilon^*$ preserves products.
\end{proof}

\begin{proposition} 
\label{proposition:unstableadjoints}
The adjunction of Lemma~\ref{lem:a1preserve} descends to an adjunction:  
$$
\epsilon^*: \H_{\sigma}(S) \rightleftarrows \H_{\tau}(S): \epsilon_*,
$$ 
such that the following diagram of adjunctions commutes:
\begin{equation} \label{eq:adj-unstab}
\xymatrix{
\H_{\sigma}(S) \ar@/^1.2pc/[rr]^{\epsilon^*}\ar[dd]^{i}
&&   
\H_{\tau}(S) \ar[dd]_{i}\ar[ll]_{\epsilon_*}\\\\
\P_{\sigma, \hyp}(\Sm_S) \ar@/_1.2pc/[rr]^-{\epsilon^*}\ar@/^1.2pc/[uu]^-{\LL_{\tau,\mot}} 
&&  
\P_{\tau, \hyp}(\Sm_S) \ar@/_1.2pc/[uu]_-{\LL_{\tau,\mot}}\ar[ll]_-{\epsilon_*}.
}
\end{equation}
\end{proposition}
\begin{proof} 
Since $\epsilon_*$ preserves homotopy invariant presheaves by Lemma~\ref{lem:a1preserve}, the diagram of right adjoints commute. 
The corresponding left adjoints commutes by uniqueness of adjoints.
\end{proof}

Upon taking stabilization we obtain the following comparison result.
\begin{proposition} 
\label{prop:stab}  
The adjunction of Proposition \ref{proposition:unstableadjoints} induces an adjunction:
$$
\epsilon^*: \SH_{\sigma}(S) \rightleftarrows \SH_{\tau}(S): \epsilon_*,
$$
which fits into the following commutative diagram of adjunctions:
\begin{equation} \label{eq:adj-stab}
\xymatrix{
\SH_{\sigma}(S)  \ar@/^1.2pc/[rr]^{\epsilon^*}\ar[dd]^{\Omega^{\infty}_{T}}
&&  
\SH_{\tau}(S)  \ar[dd]_{\Omega^{\infty}_{T}}\ar[ll]_{\epsilon_*}\\\\
\H_{\sigma}(S) \ar@/_1.2pc/[rr]^-{\epsilon^*}\ar@/^1.2pc/[uu]^-{\Sigma^{\infty}_{T,+}}
&&   
\H_{\tau}(S)  \ar@/_1.2pc/[uu]_-{\Sigma^{\infty}_{T,+}}\ar[ll]_-{\epsilon_*}.
}
\end{equation}
\end{proposition}
\begin{proof} 
The adjunction of~\eqref{eq:adj-unstab} gives rise to the adjunction of~\eqref{eq:adj-stab} via the universal property of stabilization given in Theorem~\ref{thm:robalo-stab1}, 
which also leads to the diagram of left adjoints commuting. By uniqueness of adjoints, the right adjoints also commute.
\end{proof}

\begin{remark} 
\label{rem:susp-sheaf} 
From the commutation of the left adjoints in Proposition~\ref{prop:stab}, we get an equivalence of functors $\Sigma_T\epsilon^* \simeq \epsilon^*\Sigma_T$, 
and likewise for the right adjoints $\epsilon_*\Sigma_T^{-1} \simeq \Sigma_T^{-1}\epsilon_*$. 
The latter furnishes natural equivalences: $$\epsilon_* \simeq \epsilon_*\Sigma^{-1}_T \Sigma_T \simeq \Sigma^{-1}_T\epsilon_*\Sigma_T,$$ and thus:$$\Sigma_T\epsilon_*\simeq\epsilon_*\Sigma_T.$$
By the same reasoning:
$$\epsilon^* \simeq \epsilon^*\Sigma_T\Sigma^{-1}_T \simeq \Sigma^T \epsilon^* \Sigma_T^{-1},$$ and: $$\Sigma^{-1}_T \epsilon^* \simeq \epsilon^* \Sigma^{-1}_T.$$ 
Since $\epsilon^*$  and $\epsilon_*$ are exact and thus commute with $\Sigma^{1,0}$ and $\Sigma^{-1,0}$ we deduce the natural equivalence: $$\Sigma^{p,q}\epsilon_*\epsilon^* \simeq \epsilon_*\epsilon^*\Sigma^{p,q},$$ for all $p, q \in \ZZ$ in the $\PP^1$-stable motivic homotopy categories.
\end{remark}

\subsubsection{} We will call the functor:$$\epsilon^*: \SH_{\sigma}(S) \rightarrow \SH_{\tau}(S),$$ the \emph{$\tau$-localization} functor.
This terminology is justified by the next proposition. Let us define:
$$
W_{\tau}: = \{ \LL_{\mot,\sigma}(U \rightarrow X): X \in \Sm_S, U \rightarrow X\text{ is a $\tau$-hypercover} \},
$$ 
and:
$$
\Sigma^{\infty}_{T,+}W_{\tau} := \{ \Sigma^{p,q}\Sigma^{\infty}_{T,+}\LL_{\mot,\sigma}(U \rightarrow X):X \in \Sm_S, U \rightarrow X\text{ is a $\tau$-hypercover of $X$}, p,q \in \ZZ \}.
$$

\begin{proposition} 
\label{prop:loc} 
The functors:
$$
\epsilon^*: \H_{\sigma}(S) \rightarrow \H_{\tau}(S) 
\text{ and }
\epsilon^*: \SH_{\sigma}(S) \rightarrow \SH_{\tau}(S),
$$ 
are localizations at $W_{\tau}$ and $\Sigma^{\infty}_{T,+}W_{\tau}$, respectively. 
In particular, their right adjoints are fully faithful.
\end{proposition}
\begin{proof} 
For the unstable case, 
$\epsilon^*$ and localization at $W_{\tau}$ satisfy the same universal properties via Proposition~\ref{prop:univprop-unstab}. 
In slightly more detail, 
the composite of functors: 
$$
\Sm_S \rightarrow \H_{\sigma}(S) \rightarrow \H_{\sigma}(S)[W_{\tau}^{-1}],
$$ 
satisfies the universal property for $\H_{\tau}(S)$ as stated in Proposition~\ref{prop:univprop-unstab}. 
This follows similarly for the stabilized version via Theorem~\ref{thm:univprop-stab}.
\end{proof}

As a consequence of Proposition~\ref{prop:loc} we adopt the following terminology: 
If $X \in \H_{\sigma}(S)$ is in the full subcategory $\H_{\tau}(S)$ (equivalently, it satisfies $X \simeq \epsilon_*\epsilon^*X$), 
then $X$ is called a \emph{$\tau$-local object} in $\H_{\sigma}(S)$. 
Similarly, 
$\E \in \SH_{\sigma}(S)$ is a \emph{$\tau$-local object} in $\SH_{\sigma}(S)$ if it is in the full subcategory $\SH_{\tau}(S)$.

\subsubsection{} 
We give a characterization of $\tau$-local objects in terms of $S^1$-spectra, i.e., presheaves of spectra; we will discuss this $\infty$-category in more detail in~\S\ref{sect:s1-spectra}. 
Recall $\SH_{\sigma}(S)$ is enriched over the $\infty$-category of spectra.
We denote by $\maps(\E, \F)$ the \emph{mapping spectrum} corresponding to this enrichment. 
If $\E \in \SH_{\sigma}(S)$, 
define $\Omega^{\infty}_{\GG_m, \sigma}\E$ as the presheaf of spectra: 
$$
X \mapsto \maps(\Sigma^{\infty}_TX_+, \Sigma^{p,q}\E).
$$ 
This presheaf of spectra satisfies $\sigma$-hyperdescent and is $\AA^1$-invariant. 

\begin{lemma} 
\label{lem:concrete} 
An object $\E \in \SH_{\sigma}(S)$ is $\tau$-local if and only if $\Omega^{\infty}_{\GG_m, \sigma}\Sigma^{p,q}\E$ satisfies $\tau$-hyperdescent for all integers $p,q \in \ZZ$.
\end{lemma}
\begin{proof} 
This follows since for $X \in \Sm_S$ the homotopy groups of $\maps(\Sigma^{\infty}_+X,\Sigma^{p,q}\E)$ are computed as: 
$$
\epsilon_*\maps(\Sigma^{\infty}_TX_+, \Sigma^{p,q}\E) 
\simeq 
[\Sigma^{*,0}\Sigma^{\infty}_TX_+, \Sigma^{p,q}\E]_{\SH_{\sigma}(S)}.
$$
\end{proof}

\subsubsection{} \label{sect:doesnotdeserve}
The reason why $\epsilon^*$ \emph{does not deserve} to be called the $\tau$-sheafification functor is because the diagram:
\begin{equation}
\xymatrix{
\SH_{\sigma}(S) \ar[rr]^{\epsilon^*} \ar[dd]_{\Omega^{\infty}_{\GG_m, \sigma}}&&  \SH_{\tau}(S) \ar[dd]^{\Omega^{\infty}_{\GG_m, \tau}} \\\\
\SH^{S^{1}}_{\sigma}(S) \ar[rr]^-{\epsilon^*}  &&  \SH^{S^{1}}_{\tau}(S),
}
\end{equation}
\emph{does not} commute. 
Here $\SH^{S^{1}}_{\sigma}(S)$ is the $\infty$-category of homotopy invariant presheaves of spectra satisfying $\sigma$-hyperdescent. 
We offer an example to clarify the situation.

\begin{example} 
The spectrum $\Sigma^{0,q}\M\ZZ$ represents motivic cohomology in $\SH_{\Nis}(S)$ with Tate twist $q \in \ZZ$. 
By the cancellation theorem \cite{voevodsky-cancel} there is an equivalence: 
$$
\Omega^{\infty}_{\GG_m,\Nis}\Sigma^{0,q}\M\ZZ
\simeq 
\ZZ^{\tr}(q).
$$ 
Here, 
the chain complex of presheaves with transfer $\ZZ^{\tr}(q)$ is viewed as an object of $\SH_{\Nis}^{S^{1}}(S)$ via the Dold-Kan corespondence. 
If $\tau = \et$, then $\epsilon^*\M\ZZ$ computes \'etale motivic cohomology, see Corollary~\ref{cor:utrvspi}. 
Hence $\Omega^{\infty}_{\GG_m,\et}\Sigma^{0,q}\epsilon^*\M\ZZ\neq 0$ for $q<0$ since \'{e}tale motivic cohomology and \'{e}tale cohomology coincide for torsion coefficients, 
cf.~Theorem~\ref{thm:rigidity} . 
However, 
$\ZZ^{\tr}(q)$ and thus $\epsilon^*\ZZ^{\tr}(q)$ are $0$ for $q<0$ since negative weight motivic cohomology vanishes.
\end{example}

\subsection{Continuity and localization}\label{sec:cl}
Let $\Sch'_S$ be subcategory of the category of Noetherian schemes of finite dimension over $S$ which is, furthermore, \emph{adequate} in the sense of \cite{cisinski-deglise}*{2.0} and contains Henselizations of schemes. In this situation, we have access to the full six functors formalism as proved by Ayoub in \cite{ayoub} and extended in \cite{cisinski-deglise}, which we now rapidly review. We refer to \cite{cisinski-deglise}*{Theorem 2.4.50} for an exact form of this formalism that we need; see also \cite{etalemotives}*{Appendix A} for a succint summary. 

We have a functor
\begin{equation} \label{eq:premot}
\Mrm: (\Sch'_S)^{\op} \rightarrow \Cat_{\infty}
\end{equation}
which satisifies the conditions to be \emph{premotivic} as defined in \cite{etalemotives}*{Definition A.1.1} with $\mathcal{P} = \Sm$, i.e., the pullback functor $p^*$ admits a left adjoint $p_{\#}$ whenever $p$ is smooth. Although \cite{etalemotives} phrases their axioms in terms of triangulated categories, an extensive discussion of these axioms in the setting of $\infty$-categories, and the proof that $\SH$ is such an example, can be found in Khan's thesis \cite{adeel}*{Chapter 2, Section 3}. 

\subsubsection{} Two properties of the functor~\eqref{eq:premot} that we will later use are \emph{continuity} and \emph{localization}. Let us review the first property. Let $c = (c_i)_{i \in I}$ be a collection of Cartesian sections of $\Mrm$. We denote by $\Mrm_c(X) \subset \Mrm(X)$ the smallest thick subcategory \footnote{Recall that if $\C$ is a stable $\infty$-category, a full stable subcategory $\D \subset \C$ is called \emph{thick} if it is a stable subcategory and contains all retracts of objects.} of $\M(X)$ which contains $f_{\#}f^*c_{i,X}$ for any $f: Y \rightarrow X$ smooth. Following \cite{integral-mixed}*{Definition 2.3}, we call objects in $\M_c(X)$ \emph{c-constructible}. We say that $\M$ is \emph{c-generated} if for all $X \in \Sch'_S$, the $\infty$-category $\M(X)$ is generated by $\M_c(X)$ under all small colimits.

\begin{definition} \label{def:cont-loc} Suppose that $\Fscr \subset (\Sch_S')^{\Delta^1}$ is a collection of morphisms in $\Sch'.$ We say that $\M$ is \emph{continuous with respect to $\Fscr$} if for any diagram  $X: I \rightarrow \Sch_S'$ where $I$ is a small cofiltered category and the transition maps are in $\Fscr$ and the limit $X:=\lim X_i$ exists in $\Sch_S'$ the canonical map
\[
\M_c(X') \rightarrow \lim_I \M_c(X_i).
\]
is an equivalence.
We say that $M$ satisfies \emph{localization} if for any closed immersion $i: Z \hookrightarrow X$ and its open complement $j: U \rightarrow X$ we have a cofiber sequence of $\infty$-categories:
\[
\M(U) \stackrel{j_{\#}}{\rightarrow} \M(X) \stackrel{i^*}{\rightarrow} \M(Z).
\]


\end{definition}

%
%

%
%
%
%
%
%
%
%
\subsubsection{} We will make good use of the following lemma to reduce to the case of fields.

\begin{lemma} \label{lem:conserve} Suppose that we have a functor $\M: (\Sch'_S)^{\op} \rightarrow \Cat_{\infty}$ of the form~\eqref{eq:premot} which satisfies continuity and localization. Then for any $X \in \Sch'_S$, the functor
\[
\prod i_x^*: \M(X) \rightarrow \underset{i_x: \Spec\,k \rightarrow X}{\prod} \M(k),
\]
is conservative.
\end{lemma}

\begin{proof} This is paraphrase of \cite{bachmann}*{Corollary 14}, which is a consequence of the six functors formalism of \cite{ayoub} and \cite{cisinski-deglise}.

\end{proof}

\subsection{Motives} 
\label{sect:motives}
By a theory of ``motives" we mean the stable $\infty$-category of modules over the motivic cohomology spectrum constructed by the third author. Here is a summary of its basic properties"
\begin{theorem} [Spitzweck \cite{spitzweck-integral}] \label{thm:integral} Let $\Sch'$ denote the category of Noetherian separated schemes of finite dimension. For each $S \in \SH(S)$ denote by $\M\ZZ_S \in \SH(S)$ the motivic cohomology spectrum constructed in \cite{spitzweck-integral}. Then the following properties hold.
\begin{enumerate}
\item $\M\ZZ_S$ defines a Cartesian section of $\CAlg(\SH) \rightarrow \Sch'$, i.e., each $\M\ZZ_S$ is a canonically an $\mathcal{E}_{\infty}$-algebra and for every map $p: T \rightarrow S$ the canonical map $p^*\M\ZZ_S \rightarrow \M\ZZ_T$  is an equivalence of $\mathcal{E}_{\infty}$-algebras.
\item If $S$ is smooth over a field $k$, $\M\ZZ_S$ is canonically equivalent as $\mathcal{E}_{\infty}$-algebras to the motivic cohomology spectrum in the sense of Voevodsky (see, for example, \cite{cisinski-deglise}*{Definition 11.2.17}).
\item If $S$ is a Dedekind domain and $X \in \Sm_S$, then there is a canonical isomorphism
\[
\Hom_{\SH(S)}(\Sigma^{\infty}_TX_+, \M\ZZ(q)[p]) \iso H^{p,q}(X; \ZZ)
\]
where $H^{p,q}_{\mot}(-; \ZZ)$ denotes the motivic cohomology groups of \cite{levine-dedekind}. 

\item Suppose that $S = \Spec\,D$ is the spectrum Dedekind of a domain of mixed characteristics, $i: \Spec\,k \hookrightarrow S$ an inclusion of a closed point, $j: U:= S \setminus \{ x \} \hookrightarrow S$ its open complement and $f: \Spec\,K \rightarrow S$ the canonical map from the spectrum of the fraction field of $D$. Then there are canonical equivalences 
\begin{enumerate}
\item $i^! \M\ZZ_S \simeq \M\ZZ_k(-1)[-2]$ in $\SH(k)$, and
\item $f^* \M\ZZ_S \simeq \M\ZZ_K$ in $\SH(K).$
\end{enumerate}
\item In the notation of (3), there is a canonical equivalence in $\SH^{S^1}(S)$: $\Omega^{\infty}_{T}\M\ZZ_S(1) \simeq \Sigma^{-1}\Oscr^{\times}_S.$

\end{enumerate}
\end{theorem}

We denote by $\Mod_{\M\ZZ_S}$ the stable $\infty$-category of modules over the spectrum $\M\ZZ_S$. These assemble into a premotivic category:
\[
\Mod_{\M\ZZ}: \Sch'^{\op} \rightarrow \Cat_{\infty},
\]
which satisfies continuity and localization as discussed in \S\ref{sec:cl}; see \cite{spitzweck-integral}*{\S 10} for details.

\subsubsection{} Let $\tau$ be a topology finer than $\Nis$ and consider the adjunction~\eqref{eq:pi-adjunction-paper}. By the \emph{$\tau$-motivic cohomology spectrum} we mean the spectrum $\epsilon_*\epsilon^*\M\ZZ_S \in \SH(S)$. We will be most interested in $\tau=\et$, which we now make more concrete in some situations. Let $\D(S_{\et}, R)$ be the unbounded derived $\infty$-category of \'{e}tale sheaves on the small \'{e}tale site of $S$ with coefficients in $R$; 
its homotopy category is the unbounded derived triangulated category of \'{e}tale sheaves of $R$-modules on the small \'{e}tale site. We let $\DM^{\eff}_{\tau}(S,R)$ (resp. $\DM(S,R)$) the stable $\infty$-category of Voevodsky motives; for a construction in the language used in this paper see \cite{bachmann-hoyois}*{\S 14}.
We have the following ``relative rigidity theorem," first proved by Suslin over a field and later generalized by Ayoub and Cisinski-Deglise.
\begin{theorem} [Suslin, Ayoub, Cisinski-Deglise] 
\label{thm:rigidity} 
Let $S$ be a noetherian scheme and $R$ a coefficient ring of characteristic $\ell$. 
If $\ell$ is invertible in the structure sheaf $\mathscr{O}_{S}$, then the following hold.
\begin{enumerate}
\item The infinite suspension functor $\Sigma^{\infty}_{\tr}: \DM^{\eff}_{\et}(S, R) \rightarrow \DM_{\et}(S, R)$ is an equivalence.
\item There is a symmetric monoidal functor $\rho_!: \D(S_{\et}, R) \rightarrow \DM_{\et}(S, R)$ which is an equivalence of categories with inverse $\rho^*:\DM_{\et}(S, R) \rightarrow \D(S_{\et}, R)$ 
induced by restriction to the small \'etale site.
\end{enumerate}
\end{theorem}
\begin{proof} 
Part (1) is \cite{etalemotives}*{Corollary 4.1.2} and (2) is \cite{ayoub-etale}*{Theorem 4.1}, \cite{etalemotives}*{Theorem 4.5.2}.
\end{proof}
\subsubsection{} We obtain the following corollary about \'etale motivic cohomology. We shall set $\tau = \et$ and consider the adjunction~\eqref{eq:pi-adjunction}.

\begin{corollary} \label{cor:utrvspi} Let $k$ be a field and suppose that $m$ is an integer prime to the characteristic of $k$. Then for any essentially smooth scheme over a field $S$, the spectrum $\epsilon_*\epsilon^*\M \ZZ/m_S$ represents \'etale motivic cohomology, i.e. there is a canonical isomorphism
\[
H_{\et}^p(S, \mu_m^{\otimes q}) \cong [\sspt_S, \Sigma^{p,q}\epsilon_*\epsilon^*\M \ZZ/m]_{\SH(S)}.
\]
\end{corollary} 

\begin{proof} Recall that for any topology $\tau$ finer than $\Nis$, and any coefficient ring $R$, we have an adjunction $R_{\tr}:\SH_{\tau}(S) \rightleftarrows \DM_{\tau}(S; R):u_{\tr}$ (see for example the discussion in~\cite{hopkinsmorelhoyois}*{(4.1)} and the formula for motivic cohomology as in~\cite{MR2029171}). By Theorem~\ref{thm:integral}.2, $\M\ZZ_S$ represents Voevodsky's motivic cohomology spectrum, i.e., an equivalence $\M R \simeq u_{\tr}R_{\tr}(\sspt)$. The claim then follows immediately from Theorem~\ref{thm:rigidity}.2 above for $R = \ZZ/m$, which implies that $\epsilon_*\epsilon^*u_{\tr}R_{\tr}(\sspt)$ exactly represents the \'etale cohomology.
\end{proof}

\section{The slice comparison paradigm} 
\label{sect:slice}
Let $\E \in \SH(S)$ be a motivic spectrum. 
The counit map of the adjunction:
\begin{equation} 
\epsilon^*: \SH(S) \rightleftarrows \SH_{\et}(S): \epsilon_*,
\end{equation}
is always an equivalence since $\epsilon_*$ is fully faithful. 

We are asking how far away the unit map: 
\begin{equation}
\label{equation:unitmap}
\eta: 
\E \rightarrow \epsilon_*\epsilon^*\E,
\end{equation}
is from being an equivalence; \emph{a priori} we only know this when $E = 0$.

A target theorem involves the inversion of some element $\alpha$. 
The map $\eta$ factors through: 
$$
\E[\alpha^{-1}] \rightarrow \epsilon_*\epsilon^*\E,
$$ 
which under favorable conditions will turn out to be an equivalence. 
Our approach to this question employs the slice spectral sequence and its \'{e}tale localization.
The theory of slices in motivic homotopy theory has been developed in recent years since its introduction in Voevodsky's ``open problems" article \cite{voevodsky-open}. 
We refer to \cite{milnor-hermitian}*{\S2}, \cite{pi1}*{\S3}, and \cite{rso-solves}*{\S3} for basic properties of slices, some of which we review in a slightly more general setting. 
For a geometric incarnation in terms of the coniveau tower, 
see \cite{levine-chow}, and \cite{levine-coniveau}.

For the $q$-th effective cover $f_q\E$ of $\E$ we can consider the unit map $\eta_q: f_q\E \rightarrow \epsilon_*\epsilon^*f_q\E$.
Moreover, 
applying $\epsilon_*\epsilon^*$ to the natural map $f_q\E\to\E$ produces $\eta^q: \epsilon_*\epsilon^*f_q\E \rightarrow \epsilon_*\epsilon^*\E$.
By passing to the colimits of the diagrams afforded by $\eta_q$ and $\eta^q$ we obtain a factorization of \eqref{equation:unitmap}:
$$
\E
\simeq
\colim f_q\E 
\stackrel{\eta_{\infty}}{\longrightarrow} 
\colim \epsilon_*\epsilon^*f_q\E 
\stackrel{\eta^{\infty}}{\longrightarrow} 
\epsilon_*\epsilon^*\E.
$$
Hence we can break down the question about $\eta$ into two parts:
\begin{enumerate}
\item [$\eta_{\infty}$] 
The upshot is that the unit maps $\{\eta_q\}_{q \in \ZZ}$ induce a map from the slice filtration $\{f_q\E\}_{q \in \ZZ}$ of $\E$ to the ``\'{e}tale slice filtration" 
$\{\epsilon_*\epsilon^*f_q\E \}_{q \in \ZZ}$ of $\colim \epsilon_*\epsilon^*f_q\E$.  
We analyze the induced map between the associated spectral sequences. 
It will never be an isomorphism unless one inverts elements in the slice spectral sequence for $\E$. 
These are daring operations --- on the one hand, we localize a spectral sequence which will generally destroy convergence properties. 
On the other hand, 
we have ``motivically sheafified" (in the sense of~\S\ref{sect:mot-sheaf-slice}) the slice spectral sequence, 
and so we need to produce new convergence statements. 
\item [$\eta^{\infty}$]
If $\epsilon_*$ commutes past sequential colimits then $\eta^{\infty}$ is an equivalence.
This will typically not happen because $\epsilon_*$ is a right adjoint, 
but it will be the case under the assumption of finite cohomological dimension.
\end{enumerate}
We address the issues with $\eta_{\infty}$ in~\S\ref{sect:some-conn} and the issue with $\eta^{\infty}$ in~\S\ref{sect:commuting}.

\subsubsection{} 
Suppose that $\E \in \CAlg(\SH(S))$. 
We begin with a simple generalization of slices to $\E$-module slices.
Let $\SH^{\eff}(S)$ be the stable $\infty$-category generated by all suspension spectra $\Sigma^{\infty}_T X_{+}$ under retracts and colimits, 
i.e., 
the localizing subcategory generated by these objects.
For $q\in\ZZ$, 
let $\Sigma^q_T \SH^{\eff}(S)$ denote the full localizing subcategory of $\SH(S)$ generated by $\Sigma^q_T \E$, 
where $\E\in\SH^{\eff}(S)$. 
These categories form an exhaustive filtration: 
$$
\cdots \subset \Sigma^q_T\SH^{\eff}(S) \subset \Sigma^{q-1}_T \SH^{\eff}(S)\subset \cdots \subset \SH^{\eff}(S)\subset \Sigma^{-1}_T\SH^{\eff}(S)\subset \cdots \subset \SH(S).
$$
The fully faithful embedding:
$$
i_q: \Sigma^q_T \SH^{\eff}(S) \hookrightarrow \SH^{\eff}(S),
$$ 
admits a right adjoint $r_q$; 
see Proposition~\ref{prop:E-adjoints} below for a slightly more general result. 
Let $f_q$ denote the composite $i_qr_q$.

If $\E \in \CAlg(\SH(S))$ there is a presentably symmetric monoidal stable $\infty$-category of $\E$-modules $(\Mod_{\E},\wedge_{\E},1)$.
\begin{definition} 
Suppose that $\E \in \CAlg(\SH(S))$ and is furthermore effective\footnote{With this assumption, the forgetful functor to $\SH(S)$ detects effectivity. Without this assumption, the forgetful functor will not if $E$ is, for example, the spectrum $\KGL$.}, i.e., $\E \in \SH(S)^{\eff}$. The stable $\infty$-category of \emph{effective $\E$-modules} $\Mod_{\E}^{\eff}$ is the full stable subcategory of $\Mod_{\E}$ generated by $\E \wedge \Sigma^{\infty}_T X_{+}$ 
under retracts and colimits, 
with $X \in \Sm_S$, 
i.e., 
the full \emph{localizing} subcategory of $\Mod_{\E}$ generated by the said objects. 
In particular it is a stable $\infty$-category.
For $q\in\ZZ$, 
let 
\[
\Sigma^q_T \Mod_{\E}^{\eff}:= \{ \Sigma^q_T\F: \F \in \Mod_{\E}^{\eff} \}.
\] 
\end{definition} 

\begin{proposition} 
\label{prop:E-adjoints} 
Let $E$ be an effective motivic spectrum. The fully faithful embedding $i_q^{\E}: \Sigma^q_T \Mod_{\E}^{\eff} \hookrightarrow \Mod_{\E}$ admits a right adjoint $r_q^{\E}: \Mod_{\E} \rightarrow  \Sigma^q_T \Mod_{\E}^{\eff}$, 
which gives rise to the following commutative diagram of adjunctions:
\begin{equation}
\xymatrix{
\Sigma^q_T \Mod^{\eff}_{\E}  \ar@/^1.2pc/[rr]^-{i_q^{\E}}\ar[dd]^-{u_{\E}}
&& 
\Mod_{\E}\ar[dd]_-{u_{\E}}\ar[ll]^-{r_q^{\E}} \\\\
\Sigma^q_T  \SH^{\eff}(S) \ar@/_1.1pc/[rr]^-{i_q}\ar@/^1.2pc/[uu]^-{\E\wedge-}
&& 
\SH(S). \ar@/_1.2pc/[uu]_-{\E\wedge-}\ar[ll]_-{r_q}
}
\end{equation}
\end{proposition}
\begin{proof} 
$\Mod_{\E}$ and $\Sigma^q_T \Mod_{\E}$ are presentable stable $\infty$-categories and $i_q^{\E}$ preserves small coproducts and hence all small colimits by 
\cite{higheralgebra}*{Proposition 1.4.4.1.2}. 
By the $\infty$-categorical adjoint functor theorem \cite{htt}*{Corollary 5.5.2.9} $i_q^{\E}$ admits a right adjoint. 
It suffices to prove that the diagram of left adjoints commutes.
But this follows from the fact that $E$ is effective, the definition of $\Sigma^q_T \Mod^{\eff}_{\E}$ as the full subcategory generated by suspensions of free $\E$-modules together with the fact that left adjoints commute with colimits.
\end{proof}

Because of the constant appeal to Proposition~\ref{prop:E-adjoints}, we will assume $E$ to be effective for the rest of this section, unless stated otherwise.

\subsubsection{} 
The categories $\Sigma^q_T \Mod^{\eff}_{\E}$ assemble into a filtration:
$$
\cdots\subset \Sigma^q_T \Mod^{\eff}_{\E}\subset \Sigma^{q-1}_T \Mod^{\eff}_{\E} \subset \cdots \subset \Mod^{\eff}_{\E} \subset \Sigma^{-1}_T\Mod^{\eff}_{\E}\subset \cdots \subset \Mod_{\E}. 
$$
Setting $f^{\E}_q:= i_q^{\E} r_q^{\E}$ we obtain for every $\E$-module $\M$ the \emph{$\E$-module slice tower}:
$$
\cdots \rightarrow f_{q+1}^{\E}\M \rightarrow f_{q}^{\E}\M \rightarrow \cdots \rightarrow f_0^{\E}\M \rightarrow f_{-1}^{\E}\M \rightarrow \cdots \rightarrow\M.
$$
We refer to $f_q^{\E}\M$ as the \emph{$q$-th $\E$-effective cover of $\M$}, 
and the cofiber $s_q^{\E}\M$ of $f_{q}^{\E}\M \rightarrow f_{q+1}^{\E}\M$ as the \emph{$q$-th $\E$-slice of $\M$}.

By construction, we have:

\begin{proposition} 
\label{prop:univ} 
The counit $i^{\E}_qr^{\E}_q\M = f^{\E}_q\M \rightarrow \M$ witnesses $f^{\E}_q\M$ as the universal map from $\Sigma^q_T \Mod^{\eff}_{\E}$.
\end{proposition}

\subsubsection{} 

The following result compares $\E$-module slices with standard slices. 
\begin{proposition} 
\label{prop:oblv} 
For $q \in \ZZ$, there are natural equivalences: 
$$
u_{\E} \circ f_q^{\E} \stackrel{\simeq}{\rightarrow} f_q \circ u_{\E} \text{ and }u_{\E} \circ s_q^{\E} \simeq s_q u_{\E}.
$$
\end{proposition}
\begin{proof} 
To prove the first equivalence note that for $\M \in \Mod_{\E}$, $f_qu_{\E}(\M) \simeq i_q  r_q  u_{\E} \M \simeq i_q  u_{\E}  r_q^E M$ by Proposition~\ref{prop:E-adjoints}. 
We claim that $u_{\E}  i_q \simeq i_q  u_{\E}$. 
It holds by definition for $\E \wedge \Sigma^q_T \Sigma^{\infty}_TX_{+}$.
So we are done if $i_q$ and $u_{\E}$ preserve colimits. 
This holds for $i_q$ since it is localizing, 
and for $u_{\E}$ since the adjunction $\E \wedge -: \SH(S) \rightleftarrows \Mod_{\E}: u_{\E}$ is monadic. 
The second equivalence follows readily from the first.
\end{proof}

\subsubsection{} 
We review some basic properties of effective covers and slices for $\E$-modules.
\begin{proposition} 
For all $q \in \ZZ$, the functors $f^{\E}_q$ and $s^{\E}_q$ preserve all small colimits.
\end{proposition}
\begin{proof} 
For $\E$ the sphere spectrum this is in \cite{spitzweck-rel}*{Corollary 4.6} and the argument is similar in our setting. These are exact functors of stable $\infty$-categories by construction. 
For $f^E_q$ we need only check that it preserves all small sums. 
Since $f^{\E}_q$ is the composite of a left adjoint $i^{\E}_q$ and a right adjoint $r^{\E}_n$, 
it suffices to consider $r^{\E}_q$. 
We conclude by Lemma~\ref{lem:rightadjcolim} because $i^{\E}_q$ preserves compact objects. 
It follows that $s^{\E}_{q}$ preserves small colimits since it is the cofiber of $f^{\E}_{q+1}\to f^{\E}_{q}$.
\end{proof}

\begin{proposition} 
For any $\M\in\Mod_{\E}$ and $q\in\ZZ$, the $q$-th slice $s^{\E}_q\M$ belongs to $\Sigma^q_T \Mod^{\eff}_{\E}$ and is furthermore right orthogonal to $\Sigma^{q+1}_T\Mod^{\eff}_{\E}$.
\end{proposition}

\begin{proof} 
For the first claim, we note that both $f_{q+1}^{\E}\M$ and $f_q^{\E}M$ belong to the full subcategory $\Sigma^q_T \Mod^{\eff}_{\E}$ of $\Mod_{\E}$. Since $\Sigma^q_T \Mod^{\eff}_{\E}$ is closed under cofibers, it follows that $s_q^{\E}\M$ is in $\Sigma^q_T \Mod^{\eff}_{\E}$.
%

For the second claim, it suffices to prove that:
$$
\Maps_{\Mod_{\E}}(\Sigma^{q+1}_T{\E} \wedge \Sigma^{\infty}_{T}X_+, s^{\E}_q\M) 
\simeq 0.
$$ 
This follows from the equivalences:
\begin{eqnarray*}
\Maps_{\Mod_{\E}}(\Sigma^{q+1}_T\E \wedge \Sigma^{\infty}_{T}X_+, f^\E_q\M) 
& \simeq & \Maps_{\Mod_{\E}}(i^{\E}_{q+1}r^{\E}_{q+1}\Sigma^{q+1}_T\E \wedge \Sigma^{\infty}_{T}X_+, f^{\E}_q\M)\\
& \simeq & \Maps_{\Sigma^{q+1}_T\Mod_{\E}}(r^{\E}_{q+1}\Sigma^{q+1}_T\E \wedge \Sigma^{\infty}_{T}X_+, r_{q+1}^{\E}f^{\E}_q\M) \\
& = &  \Maps_{\Sigma^{q+1}_T\Mod_{\E}}(r^{\E}_{q+1}\Sigma^{q+1}_T\E \wedge \Sigma^{\infty}_{T}X_+, r_{q+1}^{\E}i^{\E}_qr^{\E}_q\M)\\
& \simeq & \Maps_{\Mod_{\E}}(i^{\E}_{q+1}r^{\E}_{q+1}\Sigma^{q+1}_T\E \wedge \Sigma^{\infty}_{T}X_+, i^{\E}_{q+1}r_{q+1}^{\E}i^{\E}_qr^{\E}_q\M)\\
& = & \Maps_{\Mod_{\E}}(\Sigma^{q+1}_T\E \wedge \Sigma^{\infty}_{T}X_+, f^{\E}_{q+1}f^{\E}_q\M) \\
& \simeq & \Maps_{\Mod_{\E}}(\Sigma^{q+1}_T\E \wedge \Sigma^{\infty}_{T}X_+, f^{\E}_{q+1}\M),
\end{eqnarray*}
where the first equivalence follows because, by definition, $\Sigma^{q+1}_T\E \wedge \Sigma^{\infty}_{T}X_+ \simeq i^{\E}_{q+1}r^{\E}_{q+1}\Sigma^{q+1}_T\E \wedge \Sigma^{\infty}_{T}X_+$. The second equivalence is by the adjunction between $i^{\E}_{q+1}$ and $r^{\E}_{q+1}$, the third is by fully faithfulness of $i_{q+1}$, and the fourth and final is by Lemma~\ref{lem:q<q'}.
\end{proof}
%

\begin{lemma}  
\label{lem:q<q'}
The natural transformation $f^{\E}_{q'}f^{\E}_q \rightarrow f^{\E}_{q'}$ is invertible for $q<q'$, 
and $f^{\E}_{q'}f^{\E}_q \rightarrow f^{\E}_q$ is invertible for $q' \leq q$.
Hence for $q \leq q'$ there is a natural equivalence $s^{\E}_{q'}f^{\E}_{q} \overset{\simeq}{\rightarrow} s^{\E}_{q'}$.
\end{lemma}
\begin{proof} 
These follow immediately from the universal property of $f^{\E}_q$ given in Proposition~\ref{prop:univ}, 
the defining cofiber sequences for the slices, 
and commutativity of the effective covers.
\end{proof}

\begin{lemma}  
The natural map $\colim_{q \rightarrow - \infty} f^{\E}_q \rightarrow id$ is an equivalence of endofunctors of $\Mod_{\E}$.
\end{lemma}
\begin{proof} 
First note that the slice filtration on categories is exhaustive in the sense that: 
$$
\colim_q \Sigma^q_T \Mod_{\E}^{\eff} \simeq \Mod_{\E}.
$$ 
Since the $\infty$-category $\Mod_{\E}$ is generated by $\E \wedge \Sigma^{2n,n}\Sigma^{\infty}_TX_{+}, n \in \ZZ$, 
where $X \in \Sm_S$, 
it suffices to prove that for any $\E$-module $\M$ and any $X \in \Sm_S$ we have: 
$$
\Maps_{\Mod_{\E}}( \E \wedge \Sigma^{2n,n}\Sigma^{\infty}_TX_{+}, \colim_{q \rightarrow - \infty} f^{\E}_qM) 
\simeq 
\Maps_{\Mod_{\E}}(\E \wedge \Sigma^{2n,n}\Sigma^{\infty}_TX_{+}, M).
$$ 
By adjunction we are checking that:
$$
\Maps_{\SH(S)}( \Sigma^{2n,n}\Sigma^{\infty}_TX_{+}, u_{\E} \colim_{q \rightarrow - \infty} f^{\E}_qM) 
\simeq 
\Maps_{\SH(S)}( \Sigma^{2n,n}\Sigma^{\infty}_TX_{+}, u_{\E} M).
$$ 
Since $u_{\E}$ commutes with colimits and $u_{\E} \circ f_q^{\E} \simeq f_q \circ u_{\E}$ according to Proposition~\ref{prop:oblv}, 
this equivalence follows from the analogous statement for $f_q$ shown in \cite{rso-solves}*{Lemma 3.1}.
\end{proof}

\subsubsection{} 
\label{slice-complete} 
The cofiber of $f^{\E}_n \rightarrow id$ is called the  \emph{$(n-1)$-th $\E$-effective cocover} of $\E$ and is denoted by $f_{\E}^{n-1}$ (see \cite{pi1}*{3.1} for a discussion in $\SH(S)$).
This yields, 
for any $\E$-module $\M$, 
a diagram where the horizontal rows are cofiber sequences (see also \cite{rso-solves}*{3.9}):

\begin{equation} \label{eq:cofibers}
\xymatrix{
f^{\E}_{q+1}\M \ar[r] \ar[d] & \M \ar[d] \ar[r] & f_{\E}^{q}\M   \ar[d]  \\
f^{\E}_{q}\M \ar[d] \ar[r] & \M \ar[r] \ar[d] & f_{\E}^{q-1}\M  \ar[d] \\
s_q^{\E}\M \ar[r] & 0 \ar[r] & s^{\E}_q\M[1], 
}
\end{equation} 
and the cofiber sequence (see also \cite{rso-solves}*{3.11}): 
$$
\lim_{q \rightarrow +\infty} f_q^{\E}\M \rightarrow \M \rightarrow \lim_{q \rightarrow +\infty} f^{q}_{\E}\M.
$$
The \emph{$\E$-slice completion of $\M$} is defined by setting $\sc_{\E}(\M) := \lim f^{q}_{\E}\M$. 
The $\E$-module $\M$ is \emph{slice complete} if $\M \rightarrow \sc_{\E}(\M)$ is an equivalence, 
or equivalently $\lim_{q \rightarrow +\infty}  f^q_{\E}\M \simeq 0$.

\subsubsection{Slice spectral sequence} 
\label{sss} 
If $\M$ is an $\E$-module, 
we write 
$$
\pi_{p,q}(\M)(X):= [\Sigma^{p,q}\E \wedge \Sigma^{\infty}_TX_+, \M]_{\Mod_{\E}}.
$$ 
If $X$ is the base scheme we just write $\pi_{p,q}(\M) :=[\Sigma^{p,q}\E, \M]_{\Mod_{\E}}$. 
So if $\E$ is the unit sphere $\sspt$, 
this is just the group $[\Sigma^{p,q}\sspt, \M]_{\SH(S)}$. 
When $S=\Spec\,k$ is a field, 
we recover the value of the homotopy sheaf on $\Spec\,k$: $\underline{\pi}_{p,q}^{\Nis}(\M)(k) \simeq \pi_{p,q}(\M)$.
The $\E$-module slice tower yields the trigraded $\E$-slice spectral sequence: 
$$
E^{1}_{p,q,w}(\M)(X)
= 
\pi_{p,w}(s_q^{\E}\M)(X) \Rightarrow \pi_{p,w}(\M)(X).
$$
The exact couple that gives rise to this spectral sequence (with our preference for its indexing) is written in, for example, \cite{milnor-hermitian}*{(11)}. 
The differentials go: 
$$
d_r(p,q,w): 
\pi_{p,w}(s_q^{\E}\M)(X) 
\rightarrow 
\pi_{p-1, w}(s_{q+r}^{\E}\M)(X).
$$ 
We view the trigraded slice spectral sequence as a family of bigraded spectral sequences indexed by the weight $w$. 
This spectral sequence is always conditionally convergent to the homotopy groups of the slice completion $\pi_{*,*}\sc_{\E}(\M)(X)$.   
If a conditionally convergent spectral sequence collapses at a finite stage, then it is strongly convergent \cite{boardman}*{Theorem 7.1}.
Convergence of the slice spectral sequence is discussed at length in \cite{levine-converge}, \cite{pi1}, and \cite{voevodsky-open}.

\subsubsection{$\et$-localization of slices}
\label{sect:mot-sheaf-slice}
Suppose $\tau$ is finer than the Nisnevich topology and there is an adjunction $\epsilon^*: \SH(S) \rightleftarrows \SH_{\tau}(S): \epsilon_*$. 
We define the \emph{$q$-th $\tau$-effective cover} of a motivic spectrum $\E$ to be $\epsilon_*\epsilon^*f_q\E$ and refer to the corresponding filtration as the \emph{$\tau$-slice filtration}. 
Using the cofiber sequence $f_{q+1}\E \rightarrow f_{q}\E \rightarrow s_q\E$ we obtain the commutative diagram:
\begin{equation}
\xymatrix{
f_{q+1}\E \ar[r] \ar[d] & f_{q}\E \ar[r] \ar[d] & s_q\E \ar[d] \\
\epsilon_*\epsilon^*f_{q+1}\E \ar[r] & \epsilon_*\epsilon^*f_{q}\E \ar[r] & \epsilon_*\epsilon^*s_q\E.
}
\end{equation}
We call $\epsilon_*\epsilon^*s_q\E$ the \emph{$q$-th $\tau$-slice} because the bottom row is also a cofiber sequence since $\epsilon_*\epsilon^*$ is an exact functor; 
the object $\epsilon_*\epsilon^*f^q\E$ is the \emph{$q$-th $\tau$-coeffective cover}. 
The \emph{$\tau$-slice tower} is given by: 
\begin{equation} \label{eqn:tau-slice-tower}
\cdots \rightarrow \epsilon_*\epsilon^*f_q\E \rightarrow \epsilon_*\epsilon^*f_{q-1}\E \rightarrow\cdots.
\end{equation}
We refer to the corresponding spectral sequence:
\begin{equation}
\label{equation:tausss} 
E^{1,\tau}_{p,q,w}(\E)
:= 
\pi_{p,n}(\epsilon_*\epsilon^*s_q\E) 
\Rightarrow 
\pi_{p,w}(\colim \epsilon_*\epsilon^*f_q\E),
\end{equation}
as the \emph{$\tau$-slice spectral sequence}. 
As above, this give a bigraded spectral sequence for every fixed weight $w$.
Some of the issues at stake in this paper are to determine the convergence properties of \eqref{equation:tausss} and identify its filtered target groups with $\pi_{p,w}(\epsilon_*\epsilon^*\E)$.

The above has an $\E$-module analogue. 
If $\E \in \CAlg(\SH(S))$ then $\epsilon^*\E \in \CAlg(\SH_{\tau}(S))$ since $\epsilon^*$ is a monoidal functor. 
Thus it makes sense to form the \emph{$q$-th $\E^{\tau}$-effective cover} $f_q^{\E^{\tau}}M := \epsilon_* \epsilon^*f_q^{\E}M$ and the slices $s_q^{\E^{\tau}}M := \epsilon_* \epsilon^*s_q^{\E}M$. 
\subsubsection{} 
\label{sect:wrong} 
There is another possibility for what one might call ``\'{e}tale slices." 
Let $\SH^{\eff}_{\et}(S)$ be the stable $\infty$-category generated by the suspension spectra $\Sigma^{\infty}_T X_{+}$ under retracts and colimits in $\SH_{\et}(S)$. 
We may define $\Sigma^q_T \SH^{\eff}_{\et}(S)$ as above and build a slice tower in the \'{e}tale setting. 
Denote the effective covers by $f^{\et}_q$ and the slices by $s^{\et}_q$. The following proposition essentially follows from \cite{voevodsky-open}*{Remark 2.1}, but we elaborate it here for the reader's convenience.
\begin{proposition} 
Suppose $k$ is a field containing a primitive $\ell$-th root of unity where $\ell$ is prime to the exponential characteristic of $k$. 
Then $s^{\et}_q\epsilon^*\M\ZZ/\ell \simeq 0$ for all $q \in \ZZ$.
\end{proposition}
\begin{proof} 
Since $H^{0}_{\et}(k;\mu_{\ell}) \simeq [\sspt, \Sigma^{0,1}\epsilon^*\M\ZZ_{\ell}]_{\SH_{\et}}$ there is an invertible map $ \tau^{\M\ZZ}_{\ell}: \epsilon^*\M\ZZ/\ell  \rightarrow \Sigma^{0,1}\epsilon^*\M\ZZ/\ell $ 
classifying a primitive $\ell$-th root of unity in $k$. 
By the proof of \cite{milnor-hermitian}*{Lemma 2.1}, 
which makes no use of the Nisnevich topology, 
there is an equivalence $f^{\et}_{q+1}\Sigma^{0,1}\epsilon^*\M\ZZ/\ell \overset{\simeq}{\to} \Sigma^{0,1}f^{\et}_{q}\epsilon^*\M\ZZ/\ell$ and a commutative diagram:
\begin{equation}
\xymatrix{
f^{\et}_{q+1}\epsilon^*\M\ZZ/\ell  \ar[d]_-{f^{\et}_{q+1}(\tau^{\M\ZZ}_{\ell})} \ar[r] & f^{\et}_q\epsilon^*\M\ZZ/\ell  \ar[d]^-{\tau^{\M\ZZ}_{\ell}} \\
f^{\et}_{q+1}\Sigma^{0,1}\epsilon^*\M\ZZ/\ell   \ar[r]^-{\simeq} & \Sigma^{0,1}f^{\et}_{q}\epsilon^*\M\ZZ/\ell .
}
\end{equation}
The vertical arrows are invertible because multiplication by the class of $\tau^{\M\ZZ}_{\ell}$ is invertible.
Thus the top arrow is invertible and its cofiber $s^{\et}_q\epsilon^*\M\ZZ/\ell\simeq 0$.
\end{proof}
In fact, one can run the argument in \cite{voevodsky-open}*{Remark 2.1} to show that $s^{\et}_0\epsilon^*\MGL/\ell \simeq s^{\et}_0\epsilon^*\M\ZZ/\ell\simeq 0$.  
Hence this approach should be abandoned.

\section{Some \'{e}tale connectivity results} 
\label{sect:some-conn}
We come to the technical heart of the paper which concerns the convergence of the \'et-slice spectral sequence. The main result that we will need is Corollary~\ref{cor:et-complete}. For the rest of the paper, we will only be considering the adjunction:
\begin{equation} \label{eq:pi-adjunction-paper} 
\epsilon^*:\SH(S) \rightleftarrows \SH_{\et}(S): \epsilon_*.
\end{equation}
We also employ the following conventions:
\begin{itemize}
\item $\E^{\et}:=\epsilon_*\epsilon^*\E$ for the \'{e}tale localization of $\E \in \SH(S)$,
\item $\H(S)$, $\SH^{S^1}(S)$, $\SH(S)$ for $\H_{\Nis}(S)$, $\SH^{S^1}_{\Nis}(S)$, $\SH_{\Nis}(S)$, respectively. Objects in $\SH(S)$ will be called motivic spectra.
\end{itemize}
We will still have occasion to use $\epsilon^*\E$ when we make arguments internal to $\SH_{\et}(S)$. We are interested in the spectral sequence \eqref{equation:tausss} for the \'{e}tale topology. 
For $X \in \Sm_S$ this takes the form: 
\begin{equation}
\label{equation:etsss}
E^{1,\et}_{p,q,w}(\E)(X)
:= 
\pi_{p,w}(s_q\E^{\et})(X)
\Rightarrow 
\pi_{p,w}(\colim f_q\E^{\et})(X).
\end{equation} 
We will refer to \eqref{equation:etsss} as the \emph{\'{e}tale slice spectral sequence}, 
but be warned that it might not be what the reader expects, 
as discussed in \S\ref{sect:wrong} and further in Example~\ref{exmp:doesnotdeserve2}. The goal of this section is to prove convergence results about this spectral sequence.

In order to access the convergence of the \'{e}tale slice spectral sequence, 
we will need a bound on the connectivity of $\E^{\et}$ and its \'{e}tale slices. 
This is in general a delicate issue, 
mainly due to the problems pointed out in~\S\ref{sect:doesnotdeserve} and the lack of a connectivity theorem for the $\SH_{\et}(S)$ in the sense of \cite{morel-conn}.
When $\E$ is a Landweber exact motivic spectrum, we can produce a bound on the connectivity, measured differently.
Producing this bound is achieved in Theorem~\ref{thm:et-conn-final}. Before we proceed, let us examine the example in~\S\ref{sect:doesnotdeserve} in greater detail.

\begin{example} 
\label{exmp:doesnotdeserve2} 
Suppose $k$ is a field with exponential characteristic coprime to $\ell$. 
We view the mod-$\ell$ motivic cohomology spectrum $\M\ZZ/\ell$ as a $\GG_m$-spectrum in $\SH^{S^1} (S)$, 
i.e., 
a sequence of $\AA^1$-invariant and Nisnevich-local presheaves of spectra $(\E_i)_{i \in \NN}$ equipped with bonding maps $\E_i \rightarrow \Omega_{\GG_m}\E_{i+1}$, 
each of which is an equivalence. 
From this point of view the constituent $S^1$-spectra (again thinking of chain complexes as spectra) of $\Sigma^{0,-1}\M\ZZ/\ell$ are: 
\begin{equation*}
\label{eq:s1-spectra}
(0, \ZZ/\ell[1], \ZZ/\ell(1)[2], \ZZ/\ell(2)[3], ...).
\end{equation*} 
For each $t \in \NN$ we have a bonding map: 
\begin{equation} 
\label{eq:bonding}
\ZZ/\ell(t-1)[t] \rightarrow \Omega_{\GG_m}\ZZ/\ell(t)[t+1],
\end{equation} 
which is an equivalence due to Voevodsky's cancellation theorem \cite{voevodsky-cancel}. 
Taking the \'{e}tale sheafification of each of these individual spectra gives us a sequence of $S^1$-spectra:
\begin{equation} \label{eqn:mz-example}
(0, \LL_{\et}\ZZ/\ell[1], \LL_{\et}\ZZ/\ell(1)[2], \LL_{\et\ZZ/\ell}(2)[3], ...),
\end{equation}
each of which is still $\AA^1$-invariant and in fact equivalent to $\mu_{\ell}^{\otimes t}[t+1]$ on account of the identification $\LL_{\et}\ZZ/\ell(1) \simeq \mu_{\ell}$. 
Furthermore,
and this is the crux point,  we have two cases:
\begin{enumerate}
\item each bonding map~\eqref{eq:bonding} induces an equivalence for $t>1$:
$$
\LL_{\et}\ZZ/\ell(t-1)[t] \rightarrow \Omega_{\GG_m}\LL_{\et}\ZZ/\ell(t)[t+1].
$$ 
\item For $t =1$ we have: 
\begin{equation} \label{eq:no-commute}
\Omega_{\GG_m}\LL_{\et}\ZZ/\ell[1] \simeq \mu_{\ell}^{\otimes -1} \not\simeq 0.
\end{equation}
\end{enumerate}
Indeed, 
non-commutativity of the diagram in~\S\ref{sect:doesnotdeserve} boils down to this phenomenon ---
even if the levelwise \'{e}tale sheafifcation preserves $\AA^1$-invariance, 
the individual spectrum may no longer be a $\GG_m$-spectrum but is, 
instead, 
a $\GG_m$-\emph{prespectrum} as we saw above --- a notion we will revisit in the $\infty$-categorical setting in~\S\ref{sect:gm-pre}. Phrased another way, the \'etale localization functor \emph{does not} commute with taking $\GG_m$-loops unstably and only does so after taking spectrification again in the sense elaborated in Remark~\ref{rem:susp-sheaf}.
\end{example}

\subsubsection{} 
As Example \ref{exmp:doesnotdeserve2} shows one cannot simply apply descent spectral sequence arguments to get connectivity bounds on $\E^{\et}$. 
To spell this out let us denote by: $$\Omega^{\infty}_{\GG_m,\et}: \SH_{\et}(S) \rightarrow \SH^{S^1}_{\et}(S),$$ the infinite loop functor in the etale topology. 
Each of the constituent $S^1$-spectra of $\E^{\et}$ is given by:
$$
(\Omega^{\infty}_{\GG_m,\et}\E^{\et}, \Omega^{\infty}_{\GG_m,\et}\Sigma^{1,1}\E^{\et}, \Omega^{\infty}_{\GG_m,\et}\Sigma^{2,2}\E^{\et}, \cdots).
$$
From this we have a descent spectral sequence: 
$$
H^p_{\et}(X, \underline{\pi}^{\et}_{-q}\Omega^{\infty}_{\GG_m,\et}\Sigma^{i,i}\E^{\et}) 
\Rightarrow 
[X,\GG_m^{\wedge i} \wedge \E^{\et}[p+q]].
$$ 
\emph{A priori} we do not know any relationship between $\underline{\pi}^{\et}_{-q}(\Omega^{\infty}_{\GG_m,\et}\Sigma^{i,i}\E^{\et})$ and its Nisnevich version 
$\underline{\pi}^{\Nis}_{-q}(\Omega^{\infty}_{\GG_m}\Sigma^{i,i}\E)$ since $\Omega^{\infty}_{\GG_m,\et}\Sigma^{i,i}\E^{\et}$ is not, 
in general, 
the \'{e}tale sheafification of the Nisnevich-local presheaf of spectra $\Omega^{\infty}_{\GG_m}\Sigma^{i,i}\E$.

\subsubsection{}  \label{sect:s1-spectra} We begin by examining more closely the category of $S^1$-spectra, $\SH_{\tau}^{S^1}(S)$, i.e., the $\infty$-category of $\tau$-local and $\AA^1$-invariant presheaves of spectra. Let us denote by: $$\P_{\Spt}(S):=\Fun(\Sm_S^{\op},\Spt),$$ the $\infty$-category of presheaves of spectra on $\Sm_S.$ The $\infty$-category $\SH^{S^1}_{\AA^1}(S)$ (resp. $\SH^{S^1}_{\tau}(S)$)  is the full subcategory of $\P_{\Spt}(S):=\Fun(\Sm_S^{\op},\Spt)$ spanned by those presheaves of spectra which are $\AA^1$-invariant as in~\S\ref{sec:pshv} (resp. $\tau$-local as in \S\ref{sect:hypersheaves}), with spectra instead of spaces. 

We denote by: 
$$\LL_{\AA^1}: \P_{\Spt}(S) \rightarrow \P_{\Spt,\AA^1}(S), \LL_{\tau}:\P_{\Spt}(S) \rightarrow \P_{\Spt, \tau}(S),$$ 
and: $$\LL_{\tau, \mot}:= \LL_{\AA^1}\LL_{\tau}: \P_{\Spt}(S) \rightarrow \SH^{S^1}_{\tau}(S),$$ the $\AA^1$-localization, $\tau$-localization, and motivic localization functors for presheaves of spectra, respectively. The notation is justified since these functors are just the $S^1$-stabilizations in the sense of~\cite{higheralgebra}*{Section 1.4} of their unstable versions described in~\S\ref{sec:pshv},~\S\ref{sect:hypersheaves},  and~\S\ref{sec:ltau}. 

In order to avoid confusion, 
we write 
\[\LL_{\et,\mot} = \LL_{\AA^1}\LL_{\et}: \SH^{S^1}(S) \rightarrow \SH^{S^1}_{\et}(S),
\] for the $\AA^1$-localization and $\et$-sheafification of presheaves of $\Nis$-local spectra, while we write 
\[
\epsilon^*:\SH(S) \rightarrow \SH_{\et}(S)
\] for \'{e}tale localization in the sense of Proposition~\ref{prop:loc}.

\subsubsection{} Before we proceed with motivic homotopy theory, we require some ingredients from the theory of presheaves of spectra. The first is a convergence statement for $\et$-local presheaves of spectra, see also \cite{jardine-spre}*{Lemma 3.4}, \cite{SAG}*{Corollary 1.3.3.11}.

\begin{proposition} \label{prop:jardine-improved} Let $k$ be a field such that $\cd_{\ell}(k) < \infty$, and let $$\cdots \rightarrow \E_{i+1} \rightarrow \E_i \rightarrow \E_{i-1} \rightarrow \cdots \rightarrow \E_0$$ be a tower in $\P_{\Spt}(k).$ Let $\fib_i(\E)$ denote the fiber of $\E_i \rightarrow \E_{i-1}$, and suppose that the following condition holds:
\begin{itemize} \label{eqn:tower-high-conn}
\item Let $\conn(i)$ denote the connectivity of $\fib_i(E)$ in $\P_{\Spt}(\Sm_k).$ Then for any $n \geq 0$, there exists an $i \gg 0$ such that for any $i' \geq i$, we have $\conn(i') \geq n$.
\end{itemize}
Then for any $\nu \geq 1$ there is a natural equivalence: 
$$
\LL_{\et}(\lim_{i} \E_i/\ell^{\nu}) \overset{\simeq}{\rightarrow} \lim_i \LL_{\et}(\E_i/\ell^{\nu}).
$$
\end{proposition}
\begin{proof} Denote by $\E/\ell^{\nu}$ the limit of the tower:
$$\cdots \rightarrow \E_{i+1}/\ell^{\nu}  \rightarrow \E_i/\ell^{\nu}  \rightarrow \E_{i-1}/\ell^{\nu}  \rightarrow \cdots \rightarrow \E_0/\ell^{\nu} ,$$
in $\P_{\Spt}(\Sm_k)$. We have canonical morphisms: $$\E/\ell^{\nu} \rightarrow \E_i/\ell^{\nu} \rightarrow \LL_{\et}\E_i/\ell^{\nu}.$$ The inclusion $\P_{\Spt,\et}(k) \hookrightarrow \P_{\Spt}(k)$  preserves limits since it is a right adjoint, and thus $\lim  \LL_{\et}\E_i/\ell^{\nu}$ is $\et$-local. Our goal is then to prove that the map: 
\begin{equation} \label{eqn:et-lim}
\E/\ell^{\nu} \rightarrow  \lim \LL_{\et}\E_i/\ell^{\nu},
\end{equation}
witnesses an $\et$-localization. We first claim that for any $n \geq 0$ the map on homotopy sheaves: $$\underline{\pi}_n^{\et}(\E/\ell^{\nu}) \rightarrow  \underline{\pi}_n^{\et}(\lim \LL_{\et}\E_i/\ell^{\nu}),$$ is an isomorphism. On homotopy presheaves, choose $i \gg 0$ such that: $$\cdots \cong \underline{\pi}_n(\E_{i+2}) \iso   \underline{\pi}_n(\E_{i+1}) \iso \underline{\pi}_n(\E_{i}),$$ which we can by the assumption on the connectivity of $\fib_i(\E).$ Therefore, by the Milnor $\lim$-$\lim^{1}$-sequence we have isomorphisms: 
\begin{equation} \label{eqn:lim-pi}
\underline{\pi}_n(\lim \E_i) \iso \lim \underline{\pi}_n(\E_i) \iso \underline{\pi}_n(\E_i), 
\end{equation}
for all $i \gg 0$. Having this, we obtain the isomorphisms:
\begin{eqnarray*}
\underline{\pi}_n^{\et}(\E/\ell^{\nu}) 
=  
\underline{\pi}_n^{\et}(\lim_i \E_i/\ell^{\nu}) 
\iso 
a_{\et}(\underline{\pi}_n \E_i/\ell^{\nu})
\iso 
\underline{\pi}_n^{\et}( \E_i/\ell^{\nu})
\iso 
\underline{\pi}_n^{\et}(\LL_{\et}\E_i/\ell^{\nu}). 
\end{eqnarray*}

Here, the first two isomorphisms are due to~\eqref{eqn:lim-pi}, 
and the third isomorphism is due to the fact that \'{e}tale sheafification does not change stalks \footnote{To see this: note that, by the finiteness assumption on cohomological dimension, we have that Postnikov and hypercompletion coincide \cite{clausen-mathew}*{Proposition 2.10}. As explained in \cite{clausen-mathew}*{Construction 2.25}, we can use the Godement resolution in order to compute Postnikov sheafification. Therefore we consider the map $\E/\ell^{\nu} \rightarrow L_{\et}\E/\ell^{\nu}$ and restrict it to a strictly henselian local ring. Since any strictly henselian local ring does not have any nontrivial \'etale cover, the Godement resolution is trivial and we the map $\E/\ell^{\nu} \rightarrow L_{\et}\E/\ell^{\nu}$ induces an isomorphism on homotopy groups since the Godement resolution commutes with homotopy groups \cite{aktec}*{(1.26)}.
}. 
Now the descent spectral sequence: 
$$H^p(X, \underline{\pi}^{\et}_{-q}(\E/\ell^{\nu})) \Rightarrow [X, \Sigma^{p+q}\LL_{\et}\E/\ell^{\nu}]_{\SH_{\et}(k)},$$ 
which is strongly convergent due to the cohomological dimension assumption, 
shows that the map $\E/\ell^{\nu} \rightarrow  \lim \LL_{\et}\E_i/\ell^{\nu}$ is an $\et$-localization, 
as desired.
\end{proof}

\subsubsection{} \label{sect:descent-ss-for-mot} 
Next, recall that we have the descent spectral sequence for $\et$-local motivic spectra, 
which boils down to a Bousfield-Kan spectral sequence as in \cite{aktec}*{Proposition 1.36}.

\begin{proposition} 
\label{prop:dss} For $\E \in \SH(S)$, there is a spectral sequence:
\begin{equation}
E^2_{p,q} = H^p_{\et}(S, \underline{\pi}^{\et}_{-q}(\Omega^{\infty}_{\GG_m}\Sigma^{t,t}\E^{\et})) \Rightarrow [\sspt_S, \Sigma^{p+q+t,t}\E^{\et}]_{\SH(S)}.
\end{equation}
This spectral sequence is strongly convergent if the homotopy sheaves $\underline{\pi}^{\et}_{-q}(\Omega^{\infty}_{\GG_m}\Sigma^{t,t}\E^{\et})$ are $\ell$-torsion 
and the residue characteristics of $S$ have finite $\ell$-cohomological dimension or if the homotopy sheaves $\underline{\pi}^{\et}_{-q}(\Omega^{\infty}_{\GG_m}\Sigma^{t,t}\E^{\et})$ 
vanish for $-q \ll 0.$
\end{proposition}
\begin{proof} 
For any $t \in \ZZ$, consider the presheaf of $S^{1}$-spectra $\Omega^{\infty}_{\GG_m} \Sigma^{t,t} \E^{\et}$. 
This presheaf satisfies $\et$-hyperdescent by Proposition~\ref{prop:loc}. 
The general format of the descent spectral sequence (see for example \cite{jardine}*{Section 6.1} or \cite{aktec}*{Proposition 1.36}) is: 
$$
H^p_{\et}(S, \underline{\pi}^{\et}_{-q}\F) \Rightarrow \pi_0\Maps_{\SH_{\et}^{S^{1}}(S)}(\sspt_S, \Sigma^{p+q}\F),
$$ 
for $\F$ a presheaf of spectra satisfying $\et$-hyperdescent. 
Plugging in $\Omega^{\infty}_{\GG_m} \Sigma^{t,t} \E^{\et}$ for $\F$ gets us the desired descent spectral sequence. 
The convergence statements follow from standard convergence criteria for the descent spectral sequence discussed in \cite{aktec}*{5.44-5.48}.
\end{proof}

\subsection{$\GG_m$-prespectra} 
We recall that the adjunction~\eqref{eqn:stab-adjunction} factors as the $S^1$-stabilization of $\H_{\tau}(S)$:
\begin{equation} 
\label{eqn:stab-adjunction2}
\Sigma^{\infty}_{S^1,+}: \H_{\tau}(S) \rightleftarrows \SH^{S^1}_{\tau}(S):\Omega^{\infty}_{S^1},
\end{equation}
and the process of $(\GG_m,1)$-inversion described in the sense discussed in Corollary~\ref{thm:robalo-stab2}:
\begin{equation} 
\label{eqn:stab-adjunction3}
\Sigma^{\infty}_{\GG_m,+}: \SH^{S^1}_{\tau}(S) \rightleftarrows \SH_{\tau}(S)= \SH^{S^1}_{\tau}(S)[ (\GG_m,1)^{-1}]:\Omega^{\infty}_{\GG_m}.
\end{equation}
By the formula given in~\eqref{eq:colim-stab}, the $\infty$-category $\SH_{\tau}(S)$ is obtained from $\SH^{S^1}_{\tau}(S)$ by the colimit:
\begin{equation} 
\label{eqn:colim-gm}
\SH^{S^1}_{\tau}(S) \stackrel{(\GG_m,1)\wedge -}{\rightarrow} \SH^{S^1}_{\tau}(S) \stackrel{(\GG_m,1)\wedge -}{\rightarrow}  \SH^{S^1}_{\tau}(S) \rightarrow \cdots,
\end{equation}
computed $\Mod_{\H_{\tau}(S)}(\Pr^L)$.
Since colimits in $\Mod_{\H_{\tau}(S)}(\Pr^L)$  can also be computed as a limit in $\Mod_{\H_{\tau}(S)}(\Pr^R)$, 
where the transition maps are the right adjoints \cite{htt}*{5.5.3}, 
and the forgetful functor $\Mod_{\H_{\tau}(S)}(\Pr^R) \rightarrow \Cat_{\infty}$ preserves limits, 
the underlying $\infty$-category of $\SH_{\tau}(S)$ is computed as the limit of the diagram:
\begin{equation} \label{eqn:lim-gm}
\SH^{S^1}_{\tau}(S) \stackrel{\Omega_{\GG_m}}{\leftarrow} \SH^{S^1}_{\tau}(S) \stackrel{\Omega_{\GG_m}}{\leftarrow}  \SH^{S^1}_{\tau}(S) \stackrel{\Omega_{\GG_m}}{\leftarrow} \cdots, 
\end{equation}
in the $\infty$-category of $\infty$-categories, $\Cat_{\infty}$.
This formulation lets us speak of $\GG_m$-prespectra as limits in $\Cat_{\infty}$ that can be described ``pointwise." 
The following discussion can be seen as a formulation of the usual notion of ``$\GG_m$-$S^1$" bispectra following \cite{hovey}, \cite{jardine-symspec}, and \cite{vro}

\subsubsection{} \label{sect:gm-pre} 
To elaborate on this point, 
let $L \subset \NN$ be the $1$-skeleton of the nerve of the poset defined by the natural numbers. 
The diagram~\eqref{eqn:lim-gm} determines a functor: 
$$
p:L^{\op} \rightarrow \Cat_{\infty}.
$$ 
The $\infty$-categorical Grothendieck construction \cite{htt}*{\S 3.2} applied to the functor $p$ furnishes a Cartesian fibration: 
$$
q:\int \SH^{S^1}_{\tau}(S) \rightarrow L,
$$
such that the $\infty$-category of Cartesian sections of $q$ is equivalent to the limit of \eqref{eqn:lim-gm}, 
i.e., 
the $\infty$-category $\SH_{\tau}(S)$ \cite{htt}*{Corollary 3.3.3.2}. 
Concretely, 
this means that to specify an object of $\SH_{\tau}(S)$ is equivalent to:
\begin{itemize}
\item specifying objects $\{ \E_i \in \SH^{S^1}_{\tau}(S)\}_{i \in \NN}$, 
\item and bonding maps $\epsilon_i: \E_i \rightarrow \Omega_{\GG_m}\E_{i+1}$, 
which are equivalences.
\end{itemize}
We shall call such a section a \emph{$\GG_m$-spectrum in $\SH^{S^1}_{\tau}(S)$}. 
The last condition is implied by the demand that the section be Cartesian.
However, 
we may also speak of the $\infty$-category of \emph{$\GG_m$-prespectra in $\SH^{S^1}_{\tau}(S)$}, which is defined as sections of $q$ which are not necessarily Cartesian. 
Concretely,  specifying a $\GG_m$-prespectrum is equivalent to:
\begin{itemize}
\item specifying objects $\{ \E_i \in \SH^{S^1}_{\tau}(S)\}_{i \in \NN}$,
\item and bonding maps $\epsilon_i: \E_i \rightarrow \Omega_{\GG_m}\E_{i+1}$, which need not be equivalences.
\end{itemize}
\subsubsection{} We denote by $\SH^{\pre}_{\tau}(S):= \Fun_{L}(L, \int \SH^{S^1}_{\tau}(S))$ the $\infty$-category of $\GG_m$-prespectra, 
which is defined as the space of sections of the Cartesian fibration $q$; 
\cite{htt}*{Corollary 3.3.3.2} implies that there is a fully faithful embedding $u:\SH_{\tau}(S) \hookrightarrow \SH^{\pre}_{\tau}(S)$ as the Cartesian sections. 
According to \cite{hoyois-cdh}*{Page 7} this defines a localization, 
i.e., 
$u$ has a left adjoint:
\begin{equation} 
\label{eq:spectrification}
Q: \SH^{\pre}_{\tau}(S) \rightleftarrows \SH_{\tau}(S): u,
\end{equation}
which we call \emph{spectrification}.
It is computed by the formula~\eqref{eqn:model-tau-spectrify} which we discuss in~\S\ref{sect:transfinite}.

For the rest of the paper we will specify $\GG_m$-prespectra in $\SH^{S^1}_{\tau}(S)$ by writing a sequence of objects in $\SH^{S^1}_{\tau}(S)$:
$$
(\E_0, \E_1, \cdots ), \E_i \in \SH^{S^1}_{\tau}(S).
$$
The bonding maps will always be clear from the context.


\subsubsection{} 
\label{sect:application} 
Let $\E \in \SH(S)$ be a motivic spectrum. 
By the discussion of~\S\ref{sect:gm-pre}, 
we obtain a $\GG_m$-spectrum in $\SH^{S^1} (S)$, 
specified by the sequence of objects in $\SH^{S^1} (S)$: 
$$
(\Omega^{\infty}_{\GG_m}\E, \Omega^{\infty}_{\GG_m}\Sigma^{1,1}\E, \Omega^{\infty}_{\GG_m}\Sigma^{2,2}\E,\cdots),
$$ 
such that the bonding maps: 
$$
\Omega^{\infty}_{\GG_m}\Sigma^{i,i}\E \stackrel{\simeq}{\rightarrow} 
\Omega_{\GG_m}\Omega^{\infty}_{\GG_m}\Sigma^{i+1,i+1}\E,
$$ 
are equivalences for all $i \in \NN$. 
Now, we have an inclusion $\pi_{*,\pre}:\SH^{\pre}_{\et}(S) \hookrightarrow  \SH^{\pre}(S)$ over $L$, 
where $\SH^{\pre}_{\et}(S)$ is identified with the full subcategory of $\GG_m$-prespectra in $\SH^{S^1}(S)$ spanned by those prespectra such that $(\Omega^{\infty}_{\GG_m}\E)$ is $\et$-local, 
i.e., 
sections of $\int \SH^{S^1}(S) \rightarrow L$ which lands in $\SH^{S^1}_{\et}(S) \hookrightarrow \SH(S)$. 
Since the $\infty$-category of $\GG_m$-prespectra in $\SH^{S^1}(S)$ is described as a functor category (being a category of sections) and colimits are computed pointwise in such categories, 
$\pi_{*,\pre}$ preserves limits and thus admits a left adjoint: 
$$
\pi^{*,\pre}: \SH^{\pre}(S) \rightarrow \SH^{\pre}_{\et}(S).
$$

\subsubsection{} 
The functor $\pi^{*,\pre}$ is computed by applying: 
$$
 \LL_{\et,\mot}: \SH^{S^1}(S) \rightarrow \SH^{S^1}_{\et}(S),
$$ 
to each individual presheaves of spectra to obtain a $\GG_m$-prespectrum in $\SH^{S^1}_{\et}(S)$ in the sense of~\S\ref{sect:gm-pre} just discussed:
\begin{equation} \label{eqn:pt-wise}
\pi^{*,\pre}\E:= (\LL_{\et,\mot}\Omega^{\infty}_{\GG_m}\E, \LL_{\et,\mot}\Omega^{\infty}_{\GG_m}\Sigma^{1,1}\E, \LL_{\et,\mot}\Omega^{\infty}_{\GG_m}\Sigma^{2,2}\E, \cdots ).
\end{equation} 
Here, 
for each $i \in \NN$, 
the map: 
$$
\Omega^{\infty}_{\GG_m}\Sigma^{i,i}\E \simeq \Omega^{\infty}_{\GG_m}\Omega_{\GG_m}\Sigma^{i+1,i+1}\E  
\stackrel{\simeq}{\rightarrow} 
\Omega_{\GG_m}\Omega^{\infty}_{\GG_m}\Sigma^{i+1,i+1}\E 
\rightarrow 
\Omega_{\GG_m}\LL_{\et,\mot}\Omega^{\infty}_{\GG_m}\Sigma^{i+1,i+1}\E,
$$ 
induces the bonding map: 
$$
\LL_{\et,\mot}\Omega^{\infty}_{\GG_m}\Sigma^{i,i}\E 
\rightarrow 
\Omega_{\GG_m}\LL_{\et,\mot}\Omega^{\infty}_{\GG_m}\Sigma^{i+1,i+1}\E,
$$ 
since the target is $\AA^1$-invariant and $\et$-local. 
Indeed, 
a morphism in $\SH^{\pre}(S)$ is just a natural transformation $f: \E \rightarrow \F$ over $L$ between functors $\E, \F: L \rightarrow \int \SH^{S^1}(S)$. 
Hence if each term of $\F$ has \'{e}tale descent then $f$ must factor through $\pi^{*,\pre}\E$. Note that the bonding maps in $\pi^{*,\pre}\E$ are not necessarily equivalences and so $\pi^{*,\pre}\E$ \emph{does not} necessarily descend to an object of $\SH_{\et}(S)$.

\subsubsection{}  
The following lemma identifies $Q(\epsilon^*\E^{\et})$ as the \'{e}t-localization of $\E$:
\begin{lemma} 
\label{lem:spect} 
The functor $\epsilon^*: \SH(S) \rightarrow \SH_{\et}(S)$ is calculated as the composite functor 
$$
\SH(S) \stackrel{\pi^{*,\pre}}{\rightarrow} \SH^{\pre}_{\et}(S) \stackrel{Q}{\rightarrow} \SH_{\et}(S)\footnote{In fact, this claim is true for any topology $\tau$.}.
$$ 
\end{lemma}

\begin{proof} Using \cite{htt}*{Proposition 5.2.7.4} we need  to check that the essential image of $Q\circ\pi^{*,\pre}$ is the essential image of $\epsilon^*$, i.e., $\et$-local spectra in the sense of Proposition~\ref{prop:loc} and that $Q\circ\pi^{*,\pre}$ is left adjoint to $Q\circ\pi^{*,\pre}(\SH(S)) \subset \SH(S).$ The first statement follows from the characterization of objects in $\SH_{\et}(S)$ in Lemma~\ref{lem:concrete}, the second statement follows since the analogous statements are also true for the localization functors $Q$ and $\pi^{*,\pre}.$
\end{proof}

\subsubsection{} 
We will now define a condition on a motivic spectrum $\E \in \SH(S)$ such that its $\et$-localization can be controlled. 
To motivate the next definition, 
let us examine Example~\ref{exmp:doesnotdeserve2} in the language we just set up.

\begin{remark}  
In the notation of this section, 
the $\GG_m$-prespectrum in $\SH^{S^1}_{\et}(S)$ specified in~\eqref{eqn:mz-example} is exactly $\pi^{*,\pre}\Sigma^{0,-1}\M\ZZ/\ell$.  
Note that~\eqref{eq:no-commute} tells us that $\pi^{*,\pre}\Sigma^{0,-1}\M\ZZ/\ell$ is \emph{not} equivalent to $Q(\pi^{*,\pre}\Sigma^{0,-1}\M\ZZ/\ell) \simeq \Sigma^{0,-1}\epsilon^*\M\ZZ/\ell$, 
i.e., 
it is not a $\GG_m$-spectrum in $\SH^{S^1}_{\et}(S)$. 
However, 
we note that each bonding map $\ZZ/\ell_{\et}(t-1)[t] \rightarrow \Omega_{\GG_m}\ZZ/\ell_{\et}(t)[t+1]$ is an equivalence for $t \geq 1$. 
This motivates the next definition.
\end{remark}

\begin{definition} 
\label{defn:et-a1-naive} 
Let $\E \in \SH (S)$ and consider the $\GG_m$-prespectrum $\pi^{*,\pre}\E$ which is specified by $\{\LL_{\et,\mot}\Omega^{\infty}_{\GG_m}\Sigma^{i,i}\E\}_{i \in \NN}$.  
We say that \emph{$\E$ is \'{e}tale-$\AA^1$ naive above degree $r$} if for all $i \geq r$:
\begin{itemize}
\item 
the natural map 
\[\LL_{\et}\Omega^{\infty}_{\GG_m}\Sigma^{i,i}\E \rightarrow  \LL_{\AA^1}\LL_{\et}\Omega^{\infty}_{\GG_m}\Sigma^{i,i}\E = \LL_{\et,\mot}\Omega^{\infty}_{\GG_m}\Sigma^{i,i}\E
\] 
is an equivalence in $\P_{\et,\Spt}(k)$, 
i.e., 
the $\et$-local presheaf of spectra $\LL_{\et}\Omega^{\infty}_{\GG_m}\Sigma^{i,i}\E$ is $\AA^1$-invariant and,
\item 
the induced map 
\[\LL_{\et}\Omega^{\infty}_{\GG_m}\Sigma^{i,i}\E \rightarrow \Omega_{\GG_m}\LL_{\et}\Omega^{\infty}_{\GG_m}\Sigma^{i+1,i+1}\E
\] is an equivalence in $\P_{\et,\Spt}(k).$
\end{itemize}
\end{definition}
With this definition, 
Example~\ref{exmp:doesnotdeserve2} states that  $\Sigma^{0,-1}\M\ZZ/\ell$ is \'{e}tale $\AA^1$-naive above degree $1$. 
The following lemma justifies the above definition
\begin{lemma} 
\label{lem:naive-is-good} 
Suppose that $\E \in \SH (S)$ is \'{e}tale-$\AA^1$-naive above degree $r$. 
Then for any $i \geq r$, 
we have an equivalence in $\P_{\et,\Spt}(S)$: 
$$
\LL_{\et}\Omega^{\infty}_{\GG_m}\Sigma^{i,i}\E \simeq \Omega^{\infty}_{\GG_m}\Sigma^{i,i}\epsilon^*\E.
$$
\end{lemma}
\begin{proof} 
Both presheaves of spectra are \'{e}t-local and $\AA^1$-invariant and so define objects in $\SH^{S^1}_{\et}(S)$; 
for the left hand side this is because of the first condition in Definition~\ref{defn:et-a1-naive}, 
which also tells us that for $n \geq 0$ there is an equivalence in $\P_{\et,\Spt}(S)$: 
$$
\LL_{\et}\Omega^{\infty}_{\GG_m}\Sigma^{i+n,i+n}\E \simeq \LL_{\AA^1}\LL_{\et}\Omega^{\infty}_{\GG_m}\Sigma^{i+n,i+n}\E.
$$
By Lemma~\ref{lem:spect} we have an equivalence $Q(\pi^{*,\pre}\E) \simeq \epsilon^*\E$. 
In other words, 
for any $i \in \NN$, 
the presheaf of spectra $\Omega_{\GG_m}^{\infty}\Sigma^{i,i}\epsilon^*\E$ is the $i$-th term of the $\GG_m$-spectrum in $\SH^{S^1}_{\et}(S)$, 
$Q(\pi^{*,\pre}\E)$.
According to~\eqref{eqn:model-tau-spectrify}, for each $i \in \NN$, $Q(\pi^{*,\pre}\E)_i$ is computed by a transfinite colimit whose transition maps are the bonding maps $\LL_{\et}\Omega^{\infty}_{\GG_m}\Sigma^{i,i}\E \rightarrow \Omega_{\GG_m}\LL_{\et}\Omega^{\infty}_{\GG_m}\Sigma^{i+1,i+1}\E$ for $i \geq r$. Since $\E$ is $\AA^1$-naive above degree $r$ the hypothesis that $i \geq r$ informs us that the transition maps in the colimit are all equivalences.
Therefore, we obtain the desired equivalence:
$$ 
\Omega^{\infty}_{\GG_m}\Sigma^{i,i}\epsilon^*\E \simeq \LL_{\et}\Omega^{\infty}_{\GG_m}\Sigma^{i,i}\E.
$$ 
\end{proof}

\subsubsection{} 
The point of Definition~\ref{defn:et-a1-naive} is that we can estimate the connectivity of $\E^{\et}.$ 
In general, 
if $\E \in \SH(S)$, then we have a descent spectral sequence by the discussion in \S\ref{sect:descent-ss-for-mot}:
$$
H^p_{\et}(X, \underline{\pi}^{\et}_{-q}\Omega^{\infty}_{\GG_m}\epsilon^*\Sigma^{i,i}\E) 
\Rightarrow 
[\Sigma^{\infty}_TX_+, \epsilon^*\Sigma^{i,i}\E[p+q]]_{\SH_{\et}(S)} \simeq [\Sigma^{\infty}_TX_+, \Sigma^{i,i}\E^{\et}[p+q]]_{\SH(S)}
$$ 
Suppose that $\E$ is \'{e}tale-$\AA^1$ naive above degree $r$, then for all $X \in \Sm_S,$ we have the following equivalences of mapping spaces whenever $i \geq r$:
\begin{eqnarray*} \label{naive-maps}
\Maps_{\SH(S)}(\Sigma^{q,0}\Sigma^{\infty}_T X_+, \Sigma^{i,i}\E^{\et}) &  \simeq  & \Maps_{\SH^{S^1}_{\et}(S)}(\Sigma^{q,0}\Sigma^{\infty}_{S^1} X_+, \Omega^{\infty}_{\GG_m}\Sigma^{i,i}\epsilon^*\E) \\
& \simeq & \Maps_{\SH^{S^1}_{\et}(S)}(\Sigma^{q  ,0}\Sigma^{\infty}_{S^1} X_+, \LL_{\et}\Omega^{\infty}_{\GG_m}\Sigma^{i,i}\E).
\end{eqnarray*}
Here, 
the last equivalence is exactly Lemma~\ref{lem:naive-is-good}. 
Therefore, 
in this range, 
the \'{e}tale homotopy sheaves are computed as:
\begin{equation} \label{eqn:et-hpty-sheaves}
\underline{\pi}^{\et}_{q,0}\Sigma^{i,i}\E 
\iso 
a_{\et}(\underline{\pi}^{\pre}_q\Omega_{\GG_m}^{\infty}\Sigma^{i,i}\epsilon^*\E) 
\iso 
a_{\et}(\underline{\pi}^{\pre}_q\LL_{\et}\Omega_{\GG_m}^{\infty}\Sigma^{i,i}\E) 
\iso 
a_{\et}(\underline{\pi}^{\pre}_q\Omega_{\GG_m}^{\infty}\Sigma^{i,i}\E),
\end{equation}
where the second isomorphism is due to the condition of being \'{e}tale-$\AA^1$ naive, 
and the last isomorphism follows because $\LL_{\et}$ does not change stalks. We record this as a lemma
\begin{lemma} \label{lem:the-dss-naive} Suppose that $\E \in \SH (S)$ is \'{e}tale-$\AA^1$ naive above degree $r$. 
Then for any $i \geq r$, we have a spectral sequence:
\begin{equation} 
\label{eq:the-dss-naive}
H^p_{\et}(X, \underline{\pi}^{\et}_{-q}\Omega^{\infty}_{\GG_m}\E_i)
\Rightarrow 
 [\Sigma^{\infty}_TX_+, \Sigma^{i,i}\E^{\et}[p+q]]_{\SH (S)}.
\end{equation}
\end{lemma}

\subsection{\'Etale-$\AA^1$ naive properties of Landweber exact motivic spectra} 
Now we apply the above machinery to investigate the \'{e}tale-$\AA^1$ naive properties of certain $\MGL$ modules whose slices are computed in \cite{spitzweck2}, namely the Landweber exact motivic spectra and its slice covers. 
For the rest of this section, we work over a field $k$ and we also assume that all our Landweber exact spectra are effective.
We prove a sequence of lemmas towards establishing some \'etale-$\AA^1$ naive properties of Landweber exact motivic spectra 

\subsubsection{} \label{landweber-rapid}
We give a rapid review of motivic Landweber exact spectra, 
the basic reference being \cite{landweber}. 
If $N_*$ is a Landweber exact $\MU_*[\frac{1}{c}]$-module as in \cite{spitzweck2}*{Section 5}, 
then the functor from $\SH(S)$ to the category of Adams-graded graded Abelian groups --- see \cite{landweber}*{Section 3} --- given by 
$$
X \mapsto \MGL_{*,*}(X) \otimes_{\MU_*[\frac{1}{c}]} N_*
$$
is representable by a (motivic) Landweber exact spectrum $\E_N$. 
By construction $\E_N$ is a Cartesian section of $\SH(-)$ \cite{landweber}*{Proposition 8.5}. 
In fact, \cite{landweber} proves more: every Landweber exact spectrum is an object of the $\infty$-category $\Mod_{\MGL}$ \cite{landweber}*{Proposition 7.9} and is furthermore 
cellular \cite{landweber}*{Proposition 8.4}.
The slices of Landweber exact spectra are computed in \cite{spitzweck2}*{Theorem 6.1}. 
For $\E_{M}$ there is an equivalence $s_q(\E_{N}) \simeq \Sigma^{2q,q} \M(N_{2q})$ which is compatible with the map $N_{*} \rightarrow E_{N, *,*}$; 
see \cite{hopkinsmorelhoyois}*{Theorem 8.5}.

%
%
%

\subsubsection{} The form of the slices of Landweber exact spectra results in the following

\begin{lemma} 
\label{lem:lim-conv} 
Let $k$ be a field of exponential characteristic coprime to $\ell$ and assume $\cd_{\ell}(k) < \infty$. 
Let $\E_N \in \SH(k)$ be a Landweber exact motivic spectrum associated to a Landweber exact $\MU_*$-module $N_*$. 
Then for any pair of integers $(r,q)$ we have an equivalence in $\P_{\Spt,\et}(\Sm_k)$:
\begin{equation} 
\label{eq:lim-formula}
\LL_{\et}\Omega^{\infty}_{\GG_m}\Sigma^{r,r}f_q\E_N/\ell^{\nu} \stackrel{\simeq}{\rightarrow} \lim_n \LL_{\et}\Omega^{\infty}_{\GG_m}\Sigma^{r,r}f^{n+1}f_q\E_N/\ell^{\nu}.
\end{equation}
\end{lemma}
\begin{proof} 
For a fixed integer $q$ we have the following tower in $\SH(S)$: 
\begin{equation} 
\label{tower-1}
\cdots  \rightarrow f^nf_q\E/\ell^{\nu} 
\rightarrow \cdots 
\rightarrow f^{q+3}f_q\E_N/\ell^{\nu}   
\rightarrow   f^{q+2}f_q\E_N/\ell^{\nu}   
\rightarrow   f^{q+1}f_q\E_N/\ell^{\nu} 
\simeq s_q\E_N/\ell^{\nu},
\end{equation}
and thus for any $r \in \ZZ$, we obtain a tower in $\SH^{S^1} (S)$ after applying $\Omega^{\infty}_{\GG_m}\Sigma^{r,r}$:
\begin{equation} 
\label{tower-2}
\cdots \rightarrow 
\Omega_{\GG_m}^{\infty}\Sigma^{r,r}f^{q+3}f_q\E_N/\ell^{\nu}   
\rightarrow   \Omega_{\GG_m}^{\infty}\Sigma^{r,r}f^{q+2}f_q\E_N/\ell^{\nu}   
\rightarrow   \Omega_{\GG_m}^{\infty}\Sigma^{r,r}f^{q+1}f_q\E_N/\ell^{\nu} 
\simeq \Omega_{\GG_m}^{\infty}\Sigma^{r,r}s_q(\E_N/\ell^{\nu}).
\end{equation}
Now, 
for any $n > q$, 
we have a cofiber sequence in $\SH (S)$ (see~\eqref{eq:cofibers}): 
$$
f^{n+1}f_q\E_N/\ell^{\nu} \simeq f^{n+1}\E_N/\ell^{\nu} 
\rightarrow f^{n}f_q\E_N/\ell^{\nu} \simeq f^{n}\E_N/\ell^{\nu} 
\rightarrow \Sigma^{1,0}s_{n+1}\E_N/\ell^{\nu} 
\simeq \Sigma^{1,0}\Sigma^{2n+2,n+1}\M N_*/\ell^{\nu},
$$ where the last equivalence is due to the computation of the slices of Landweber exact spectra as in \cite{spitzweck2}. 
Therefore, the fiber of: 
$$
\Omega_{\GG_m}^{\infty}\Sigma^{r,r}f^{n+1}f_q\E_N/\ell^{\nu}   
\rightarrow   
\Omega_{\GG_m}^{\infty}\Sigma^{r,r}f^{n}f_q\E_N/\ell^{\nu},
$$ 
is equivalent to:
\begin{eqnarray*}
\Omega_{\GG_m}^{\infty}\Sigma^{r,r}s_{n+1}\E_{N}/\ell^{\nu} &  \simeq & \Omega_{\GG_m}^{\infty}\Sigma^{r,r}\Sigma^{2n+2,n+1}\M N_*/\ell^{\nu}  \\
& \simeq & \ZZ/\ell^{\nu}(2n+2+r)[n+1+r] \otimes N_{2n+2}.
\end{eqnarray*}
Since the connectivity of $\ZZ/\ell^{\nu}(2n+2+r)[n+1+r] \otimes N_{2n+2} \in \SH^{S^1} (S)$ tends to $\infty$ as $n\rightarrow \infty$, 
Proposition~\ref{prop:jardine-improved} gives us an equivalence in $\P_{\Spt,\et}(\Sm_k)$:
\begin{equation} \label{eq:lim-formula}
\LL_{\et}\Omega^{\infty}_{\GG_m}\Sigma^{r,r}f_q\E_N/\ell^{\nu} \stackrel{\simeq}{\rightarrow} \lim_n \LL_{\et}\Omega^{\infty}_{\GG_m}\Sigma^{r,r}f^{n+1}f_q\E_N/\ell^{\nu}.
\end{equation}
\end{proof}

\begin{lemma} 
\label{lem:a1-landweber}  
Let $k$ be a field of exponential characteristic coprime to $\ell$ and assume $\cd_{\ell}(k) < \infty$. 
Let $\E_N \in \SH(k)$ be a Landweber exact motivic spectrum.
For integers $(r,q)$ the $\et$-local presheaf of spectra $\LL_{\et}\Omega^{\infty}_{\GG_m}\Sigma^{r,r}f_q\E_N/\ell^{\nu}$ is $\AA^1$-invariant. 
Hence there is a canonical equivalence in $\P_{\et,\Spt}(k)$: 
$$
\LL_{\et}\Omega^{\infty}_{\GG_m}\Sigma^{r,r}f_q\E_N/\ell^{\nu} 
\stackrel{\simeq}{\rightarrow} 
\LL_{\AA^1}\LL_{\et}\Omega^{\infty}_{\GG_m}\Sigma^{r,r}f_q\E_N/\ell^{\nu},
$$ 
and $\LL_{\et}\Omega^{\infty}_{\GG_m}\Sigma^{r,r}f_q\E_N/\ell^{\nu} \in \SH^{S^1}_{\et}(k)$.
\end{lemma}
\begin{proof} 
Note that $\AA^1$-invariant presheaves in $\P_{\et,\Spt}(\Sm_k)$ are closed under limits. 
Indeed, 
if $F: I \rightarrow \P_{\et,\Spt}(\Sm_k)$ is a diagram of $\AA^1$-invariant \'{e}t-local presheaf of spectra, 
then for any $X \in \Sm_k$ we have natural equivalences:
\begin{eqnarray*}
\Maps(\Sigma^{\infty}_{S^1}X_+, \lim F_i) & \simeq & \lim_i \Maps(\Sigma^{\infty}_{S^1}X_+,  F_i) \\
& \simeq & \lim_i \Maps(\Sigma^{\infty}_{S^1}(X \times \AA^1)_+,  F_i)\\
& \simeq & \Maps(\Sigma^{\infty}_{S^1}(X \times \AA^1)_+, \lim_i F_i).
\end{eqnarray*}
Hence, 
after Lemma~\ref{lem:lim-conv}, it suffices to prove that for any $n \geq q+1$,  
each of the $\et$-local presheaf of spectra $\LL_{\et}\Omega^{\infty}_{\GG_m}\Sigma^{r,r}f^{n+1}f_q\E_N/\ell^{\nu}$ is $\AA^1$-invariant. 
For $n = q+1$, we note that $\LL_{\et}\Omega_{\GG_m}^{\infty}\Sigma^{r,r}s_{n+1}\E_N/\ell^{\nu} \simeq \LL_{\et}\ZZ/\ell^{\nu}(2n+2+r)[n+1+r] \otimes N_{2n+2}$ is $\AA^1$-invariant. 
Indeed, 
by \cite{mvw}*{Theorem 2.3}, $\LL_{\et}\ZZ/\ell^{\nu}(q)[p] \simeq \mu^{\otimes q}[p]$ for $q \geq 0$, 
which are indeed $\AA^1$-invariant \'{e}t-local presheaves \cite{SGA4}*{XV, Corollary 2.2}.

Suppose the claim holds for all $n' < n$. 
The property of being $\AA^1$-invariant is clearly closed under extensions, 
and so we conclude using the cofiber sequence in $\P_{\et,\Spt}(k)$:
$$
\LL_{\et} \Omega_{\GG_m}^{\infty}\Sigma^{r,r}s_{n}\E_N/\ell^{\nu} 
\rightarrow 
\LL_{\et}\Omega^{\infty}_{\GG_m}\Sigma^{r,r}f^{n}f_q\E_N/\ell^{\nu} \rightarrow \LL_{\et} \Omega_{\GG_m}\Omega^{\infty}_{\GG_m}\Sigma^{r,r}f^{n-1}f_q\E_N/\ell^{\nu}.
$$
\end{proof}

\begin{lemma} 
\label{lem:landweber} 
Let $\E_N \in \SH(k)$ be a Landweber exact motivic spectrum.
Then for all integers $(r,q)$ such that $r > -q$ there is a canonical equivalence: 
$$
\LL_{\et}\Omega^{\infty}_{\GG_m}\Sigma^{r-1,r-1}f_q\E/\ell^{\nu}
\overset{\simeq}{\rightarrow} 
\Omega_{\GG_m}\LL_{\et}\Omega^{\infty}_{\GG_m}\Sigma^{r,r}f_q\E_N/\ell^{\nu}.
$$ 
\end{lemma}
\begin{proof} First we note that for $q \geq 1$, we have an equivalence for all $p \in \ZZ$:
\begin{equation} \label{eq:loop-etale-formula}
\Omega_{\GG_m}\LL_{\et}\ZZ/\ell^{\nu}(q)[p] \simeq \LL_{\et}\Omega_{\GG_m}\ZZ/\ell^{\nu}(q)[p].
\end{equation}
Indeed, both sides are equivalent to $\mu_{\ell^{\nu}}^{\otimes q-1}[p-1]$.
With this, we claim that the natural map:
\begin{equation} 
\label{eq:loop-formula}
\LL_{\et} \Omega_{\GG_m}\Omega^{\infty}_{\GG_m}\Sigma^{r,r}f^{n}f_q\E_N/\ell^{\nu} 
\rightarrow 
\Omega_{\GG_m} \LL_{\et} \Omega^{\infty}_{\GG_m}\Sigma^{r,r}f^{n}f_q\E_N/\ell^{\nu},
\end{equation}
is an equivalence; in other words, we claim that $\Omega_{\GG_m}$ and $\LL_{\et}$ commute.
This is proved by induction on $n$. 
For the base case $n=q+1$ we have equivalences in $\DM^{\eff}_{\et}(k, \ZZ/\ell)$:
\begin{eqnarray*}
\Omega_{\GG_m} \LL_{\et} \Omega^{\infty}_{\GG_m}\Sigma^{r,r}f^{q+1}f_q\E_N/\ell^{\nu} &= & \Omega_{\GG_m} \LL_{\et} \Omega^{\infty}_{\GG_m} \Sigma^{r,r}s_q(\E_N/\ell^{\nu})\\
& \simeq &  \Omega_{\GG_m} \LL_{\et} \ZZ/\ell^{\nu}[2q+r](q+r) \otimes N_{2n+2}\\
& \simeq &  \LL_{\et} \Omega_{\GG_m}\ZZ/\ell^{\nu}[2q+r](q+r)\otimes N_{2n+2}.
\end{eqnarray*}
The first equivalence follows from the computation of the slices of Landweber exact spectra \cite{spitzweck2}, 
and the second follows from~\eqref{eq:loop-formula} using the hypothesis that $q+r >0$.

Now assume that~\eqref{eq:loop-formula} has been verified for $n-1$, and consider the commutative diagram:
\begin{equation}
\xymatrix{
\LL_{\et} \Omega_{\GG_m}  \Omega_{\GG_m}^{\infty}\Sigma^{r,r}s_{n}\E_N/\ell^{\nu} \ar[r] \ar[d]  
&  \LL_{\et} \Omega_{\GG_m}\Omega^{\infty}_{\GG_m}\Sigma^{r,r}f^{n}f_q\E_N/\ell^{\nu} \ar[r] \ar[d] 
&  \LL_{\et} \Omega_{\GG_m}\Omega^{\infty}_{\GG_m}\Sigma^{r,r}f^{n-1}f_q\E_N/\ell^{\nu}\ar[d] \\
\Omega_{\GG_m}  \LL_{\et} \Omega_{\GG_m}^{\infty}\Sigma^{r,r}s_{n}\E_N/\ell^{\nu} \ar[r] 
&  \Omega_{\GG_m} \LL_{\et} \Omega^{\infty}_{\GG_m}\Sigma^{r,r}f^{n}f_q\E_N/\ell^{\nu} \ar[r]  
&  \Omega_{\GG_m} \LL_{\et}\Omega^{\infty}_{\GG_m}\Sigma^{r,r}f^{n-1}f_q\E_N/\ell^{\nu}
}
\end{equation}
Since $\LL_{\et}$ and $\Omega_{\GG_m}$ are exact functors, the horizontal rows are cofiber sequences. 
The leftmost vertical map is an equivalence by the same argument as in the base case since $\Omega_{\GG_m}^{\infty}\Sigma^{r,r}s_{n}\E_N/\ell^{\nu} \simeq  \ZZ/\ell^{\nu}[2n+r](n+r) \otimes N_{2n+2}$, 
while the rightmost vertical map is an equivalence by the inductive hypothesis. 
Thus the middle vertical map is an equivalence. 

To conclude, we have the following string of equivalences, for $r > -q$:
\begin{eqnarray*}
\Omega_{\GG_m} \LL_{\et}\Omega^{\infty}\Sigma^{r,r}f_q\E_N/\ell^{\nu}& \simeq & \Omega_{\GG_m}  \lim_n  \LL_{\et}\Omega^{\infty}_{\GG_m}\Sigma^{r,r}f^{n+1}f_q\E_N/\ell^{\nu} \\
& \simeq & \lim_n \Omega_{\GG_m} \LL_{\et}\Omega^{\infty}_{\GG_m}\Sigma^{r,r}f^{n+1}f_q\E_N/\ell^{\nu}\\
& \simeq & \lim_n  \LL_{\et}\Omega_{\GG_m}\Omega^{\infty}_{\GG_m}\Sigma^{r,r}f^{n+1}f_q\E_N/\ell^{\nu}\\
& \simeq &  \lim_n  \LL_{\et}\Omega^{\infty}_{\GG_m}\Sigma^{r-1,r-1}f^{n+1}f_q\E_N/\ell^{\nu}\\
& \simeq &  \LL_{\et} \lim_n \Omega^{\infty}_{\GG_m}\Sigma^{r-1,r-1}f^{n+1}f_q\E_N/\ell^{\nu}\\
& \simeq &  \LL_{\et} \Omega^{\infty}_{\GG_m}\Sigma^{r-1,r-1}f_q\E_N/\ell^{\nu}.
\end{eqnarray*}
The first equivalence is due to~\eqref{eq:lim-formula}, 
the second is because $\Omega_{\GG_m}$ is a right adjoint and thus preserves limits, 
the third is due to~\eqref{eq:loop-formula}, the fourth is because of the canonical equivalence $\Omega_{\GG_m}\Omega^{\infty}_{\GG_m}\simeq\Omega^{\infty}_{\GG_m}\Omega_{\GG_m}$, 
the fifth is another application of~\eqref{eq:lim-formula}, 
while the last equivalence is again because $\Omega_{\GG_m}$ preserves limits.
\end{proof}

As a result we have proved:

\begin{proposition} 
\label{prop:landweber-naive} 
Let $\E_N \in \SH(k)$ be a Landweber exact motivic spectrum.
Then $f_q\E_N/\ell^{\nu}$ is \'{e}tale-$\AA^1$-naive above degree $-q$.
\end{proposition}
\begin{proof} 
The first point of Definition~\ref{defn:et-a1-naive} is verified by Lemma~\ref{lem:a1-landweber}, while the second point at the desired range is verified by Proposition~\ref{lem:landweber}.
\end{proof}

\subsection{Notions of connectivity in motivic homotopy} 
We now use Proposition~\ref{prop:landweber-naive} to prove convergence results about the \'{e}tale slice spectral sequence. We say that a motivic spectrum $\E \in \SH(S)$ is: 
\begin{itemize}
\item \emph{$t$-connected} \cite{pi1}*{Definition 3.16} if for any triple $(p,q,d)$ of integers for which $p-q+d< t$ and every $d$-dimensional $X \in \Sm_S$ the group $[\Sigma^{p,q} \Sigma^{\infty}_{T}X_+, \E]$ is zero, 
\item \emph{affine $t$-connected} if for any triple $(p,q,d)$ of integers for which $p-q+d< t$ and every $d$-dimensional $X \in \Sm_S$ which is affine, the group $[\Sigma^{p,q} \Sigma^{\infty}_{T}X_+, \E]$ is zero,
\item \emph{$t$-connective} \cite{hopkinsmorelhoyois}*{\S 2.1} if it is contained in the localizing subcategory generated by $\{ \Sigma^{p,q} \Sigma^{\infty}_{T}X_+ \}_{p-q \geq t, X \in \Sm_S},$
\item \emph{locally $t$-connective} if the homotopy sheaves $\underline{\pi}^{\Nis}_{p,q}\E = 0$ for any $p, q \in \ZZ$ such that $p-q< t.$
\end{itemize}
If $S$ is a field, then we can say more about the relationships of the above notions.
\begin{itemize}
\item Being $t$-connected is stronger than being affine $t$-connected. These notions are, in general, different because the \'etale cohomological dimension of a scheme $X$ over an algebraically closed field is bounded above by $2\dim(X)$ in general and by $\dim(X)$ if $X$ is affine \cite{milne}*{Corollary 1.4, Remark 1.5(a)}.
The latter notion is meant to accommodate certain phenomena in the \'{e}tale topology, 
which we will consider later, starting from \S\ref{main-conn}.
\item $\E$ is $t$-connective if and only if it is locally $t$-connective by \cite{hopkinsmorelhoyois}*{Theorem 2.3}. 
\item The locally $0$-connective motivic spectra form the nonnegative part of the \emph{homotopy $t$-structure} \cite{morel-trieste}*{Section 5.2}. 
Thus $0$-connective motivic spectra and $0$-locally connective motivic spectra are nonnegative parts of isomorphic $t$-structures on $\SH(S)$.
\item According to \cite{pi1}*{Lemma 3.17}, $\E$ being $t$-connective implies that $\E$ is $t$-connected. 
\end{itemize}

\subsubsection{} 
Next we turn to convergence properties of the slice spectral sequence. 

\begin{definition} 
A tower of motivic spectra: 
$$
\{E_q\}_{q \in \ZZ} = \cdots \rightarrow  \E_{q+1} \rightarrow \E_{q} \rightarrow \cdots,
$$ 
is called \emph{left bounded with respect to a pair $(X,w)$ where $X \in \Sm_S$ and $w \in \ZZ$} if for every $s \in \ZZ$, 
the group $[\Sigma^{s,w}_T\Sigma^{\infty}X_+, \E_q] =0$ for $q \gg 0$,
and is it called  \emph{left bounded} if it is left bounded with respect to all $(X, w)$. 
\end{definition}

For the effective covers in the slice filtration $\{f_q\E\}$ the notion of left boundedness is stronger than that of slice completeness; 
see the discussion in \cite{hopkinsmorelhoyois}*{\S 8.5}. 
In this case we say that the motivic spectrum \emph{$\E$ is left bounded}. 
Consequently, 
for $X \in \Sm_S$ and $w \in \ZZ$ fixed, 
the slice spectral sequence \S\ref{sss}: 
$$
E^1_{p,q,w}(\E)(X) 
\Rightarrow 
[\Sigma^{p,w}\Sigma^{\infty}_TX_+,\E],
$$
converges conditionally whenever we have left boundedness with respect to $(X, w)$ due to the slice completeness of $\E$. 
The following illustrates how connectedness implies left boundedness:
\begin{lemma} 
\label{lem:leftbounded} 
Let $\E \in \SH(S)$ and suppose that for $q \gg 0$ there exists an integer $n \in \ZZ$ such that the effective cover $f_q\E$ is $(q+n)$-connected.
Then the tower $\{f_q\E\}$ is left bounded, i.e., the motivic spectrum $\E$ is left bounded.
\end{lemma}
\begin{proof} 
Let us fix $(X, w)$, $s \in \ZZ$.
By assumption the group $[\Sigma^{s,w}\Sigma^{\infty}_TX_+, f_qE]$ is trivial when $s-w+\dim(X)-n<q$. 
We conclude by letting $q\rightarrow\infty$. 
\end{proof}

\subsubsection{} 
For convergence of the \'{e}tale slice spectral sequence, it turns out that the usual notion of connectedness will not be useful because of the following example.
\begin{example} 
\label{exmp:no-go} 
One might guess that applying the \'{e}tale localization decreases the connectivity of $\E \in \SH(S)$ by the cohomological dimension of the base scheme. 
However, 
we show that $\M\ZZ^{\et}/\ell$ is \emph{not} $0$-connected when $k$ is an algebraically closed field.  
Assuming that $\M\ZZ^{\et}/\ell$ is $0$-connected then for $p-q< 0$ the group:
$$
[\Sigma^{p,q}\sspt, \M\ZZ^{\et}/\ell] \iso H^{-p}_{\et}(k, \mu_{\ell}^{\otimes-q})=0.
$$ 
But for $p = 0$ and $q>0$ the above group $H^0_{\et}(k, \mu_{\ell}^{\otimes -q}) \iso H_{\et}^0(k, \mu_{\ell}) \iso  \mu_{\ell}(k)\neq 0$. 

However, 
$\M\ZZ^{\et}/\ell$ will be \emph{weightlessly affine $0$-connected} (see Definition~\ref{defn:eff-conn-better}). 
This notion tests vanishing only with respect to the ``$S^1$-variable" and affine schemes. 
It captures the following fact: 
let $(n,d)$ be a pair of integers such that $n+d<0$.  
Then for any affine scheme $X$ of dimension $d$ over an algebraically closed field, the group $[\Sigma^{n, 0}\Sigma^{\infty}_TX_+, \M\ZZ^{\et}/\ell] \iso H_{\et}^{-n}(X; \mu_{\ell})=0$ \cite{SGA4}*{XIV, Th\'eor\`eme 3.1}.
\end{example}

\subsubsection{} 
In the \'{e}tale-local setting the most relevant notion of connectivity for our purposes is the following one.

\begin{definition} 
\label{defn:eff-conn-better} 
A motivic spectrum $\E \in \SH(S)$ is \emph{weightlessly affine $t$-connected} if for any pair of integers $(n,d)$ such that $n+d <t$ and any $X \in \Sm_S$ which is an affine scheme of dimension $d$, 
the group $[\Sigma^{n,0}\Sigma^{\infty}_TX_+,\E]=0.$
\end{definition}

The above notion is so defined because we are only testing connectivity against generators of $\SH(S)^{\eff}$, 
namely the collection $\{\Sigma^{n,0}\Sigma^{\infty}_TX_+\}$ where $n \in \ZZ$ and $X\in \Sm_S$ for any $X \in \Sm_S$ which is an affine scheme. 
In spite of the appearance of ``effective" in Definition~\ref{defn:eff-conn-better}, 
$\E$ does not have to be an effective motivic spectrum. 
Indeed, for our purposes, 
we will be calculating the effective connectivity of motivic spectra which are of the form $\E^{\et}$ and there is no reason why $\E^{\et}$ should be effective since $\epsilon_*$ 
need not preserve effective objects.

\subsection{Main connectivity results}
\label{main-conn} 
The following is our main connectivity result. 

\begin{theorem} 
\label{thm:et-conn-final} 
Let $k$ be a field of exponential characteristic coprime to $\ell$ and assume $\cd_{\ell}(k) < \infty$. 
Let $\E_N \in \SH(k)$ be a Landweber exact motivic spectrum.
Then for all integers $(q,w)$ such that $q \geq w$, 
the spectrum $\Sigma^{-w,-w}(f_q\E_N/\ell^{\nu})^{\et}$ is weightlessly affine $q-\cd_{\ell}(k)$-connected.
\end{theorem}
\begin{proof} 
Proposition~\ref{prop:landweber-naive} implies $f_q\E_N/\ell^{\nu}$ is \'{e}tale-$\AA^1$-naive above degree $-q$. 
Therefore since $-w \geq -q$, for any $X \in \Sm_k$, 
the descent spectral sequence~\eqref{eq:the-dss-naive} takes the form: 
$$
H^s_{\et}(X, \underline{\pi}^{\et}_{-t}(\LL_{\et}\Omega^{\infty}_{\GG_m}\Sigma^{-w,-w}f_q\E_N/\ell^{\nu})) 
\Rightarrow 
[\Sigma^{\infty}_TX_+, \Sigma^{-w,-w}(f_q\E_N)^{\et}/\ell^{\nu}[s+t]]_{\SH(k)}.
$$
Suppose $n+d < q-\cd_{\ell}(k)$ and $X$ is affine of dimension $d$. 
We claim the vanishing: 
$$
[\Sigma^{n,0}\Sigma^{\infty}_TX_+, \Sigma^{-w,-w}f_q\E_N^{\et}/\ell^{\nu}]_{\SH (k)} 
\iso 
[\Sigma^{n,0}\epsilon^*\Sigma^{\infty}_TX_+, \Sigma^{-w,-w}\epsilon^*f_q\E_N/\ell^{\nu}]_{\SH_{\et}(k)}  
=
0.
$$  
Here, we applied the isomorphism given by the adjunciton $(\epsilon^*,\epsilon_*)$ and ~\S\ref{rem:susp-sheaf} to commute $\epsilon_*$ past the suspension.
Examining the descent spectral sequence, we see that the only contributions come from the terms: 
$$
H^s_{\et}(X, \underline{\pi}^{\et}_{-t}(\LL_{\et}\Omega^{\infty}_{\GG_m}\Sigma^{-w,-w}f_q\E_N/\ell^{\nu})),
$$ 
for $s+t=-n$. 
Whenever $s > d + \cd_{\ell}(k)$, 
the latter
is trivial for cohomological dimension reasons by \cite{SGA4}*{XIV, Th\'eor\`eme 3.1} since the \'{e}tale homotopy sheaves are $\ell$-torsion sheaves.

If $s \leq d + \cd_{\ell}(k)$, 
then since we assume $n+d<q-\cd_{\ell}(k)$, 
we have that  $-t< q$. 
Thus, the homotopy sheaves 
$\underline{\pi}^{\et}_{-t}(\LL_{\et}\Omega^{\infty}_{\GG_m}\Sigma^{-w,-w}f_q\E_N/\ell^{\nu}) \iso a_{\et}\underline{\pi}^{\Nis}_{-t}(\Omega^{\infty}_{\GG_m}\Sigma^{-w,-w}f_q\E_N/\ell^{\nu})$  
vanishes by the connectivity of Landweber exact motivic spectra \cite{hopkinsmorelhoyois}*{Lemma 8.11}.
\end{proof}

\begin{corollary} 
\label{corollary:MGLetleftbounded}
Under the assumptions in Theorem \ref{thm:et-conn-final} the \'{e}tale slice tower $\{(f_q\MGL)^{\et}/\ell^{\nu} \}_{q\in \ZZ}$ is left bounded.
\end{corollary}

\begin{proof}  
Fix $s, w\in \ZZ$, and an affine $X \in \Sm_k$ of dimension $d$. 
For $q \gg 0$ we want to show that
$$
[\Sigma^{s,w} \Sigma^{\infty}_TX_+,(f_q\MGL^{\et})/\ell^{\nu}] 
\iso 
[\Sigma^{s-w,0}\Sigma^{\infty}_TX_+, \Sigma^{-w,-w}(f_q\MGL)^{\et}/\ell^{\nu}],
$$  
vanishes.

When $q \geq w$, $\Sigma^{-w,-w}(f_q\MGL/\ell^{\nu})^{\et}$ is  weightlessly affine $q-\cd_{\ell}(k)$-connected by Theorem~\ref{thm:et-conn-final}. 
Hence the desired vanishing holds for all  $q > \max\{d+s-w+\cd_{\ell}(k),w\}$.
\end{proof}

\begin{corollary} 
\label{cor:et-complete}
Under the assumptions in Theorem \ref{thm:et-conn-final} we have $\lim_q (f_q\MGL)^{\et}/\ell^{\nu} \simeq 0$.
Thus the $\et$-slice spectral sequence \eqref{equation:tausss} for $\MGL/\ell^{\nu}$ is conditionally convergent. 
\end{corollary}
\begin{proof}
By Proposition~\ref{prop:gen-stab}
it suffices to prove that for all $s,w \in\ZZ$, and affine $X\in \Sm_k$, 
$$
[\Sigma^{s,w}X_+, \lim_q (f_q\MGL)^{\et}/\ell^{\nu}]
=
0.
$$  
To that end we employ the Milnor $\lim$-$\lim^{1}$ exact sequence:
\begin{equation}
\label{equation:milnorlimlim1}
0
\to
\underset{q}\limone
[\Sigma^{s+1,w}X_+,(f_q\MGL)^{\et}/\ell^{\nu}]
\to
[\Sigma^{s,w}X_+, \lim_q (f_q\MGL)^{\et}/\ell^{\nu}]
\to
\lim_{q} [\Sigma^{s,w}X_+,(f_q\MGL)^{\et}/\ell^{\nu}]
\to
0.
\end{equation}
By left boundedness of $\{(f_q\MGL)^{\et}/\ell^{\nu} \}_{q\in \ZZ}$ in Corollary \ref{corollary:MGLetleftbounded} the outer terms in \eqref{equation:milnorlimlim1} are trivial for $q\gg 0$,
and we are done.
\end{proof}

\section{The forgetful functor $\epsilon_*$ and colimits} 
\label{sect:commuting}
Next we address the other half of the comparison paradigm for the change of topology adjunction of exact functors:
\begin{equation} \label{eq:main-adjunction}
\epsilon^*:\SH(S) 
\rightleftarrows 
\SH_{\et}(S):
\epsilon_*.
\end{equation}

\subsubsection{} \label{subsec:unif-bdd} Suppose that $S$ is a scheme and let $\ell$ be a prime. We say that $S$ has \emph{uniformly bounded $\ell$-cohomological dimension} if for all residue fields $k(s)$ of $S$, $\cd_{\ell}(k(s)) < C_{\ell}$ for some constant $C_{\ell}.$ The goal of this section is to prove the following theorem

\begin{theorem} \label{thm:pi-pres-colimits} Suppose that $S$ is a Noetherian base scheme. Let $\ell$ be a prime such that $S$ has uniformly bounded $\ell$-cohomological dimension. Then the functor:
$$
\epsilon_*: \SH_{\et}(S)_{(\ell)}  
\rightarrow \SH(S)_{(\ell)},
$$ 
preserves colimits.
\end{theorem}

The statement that we will need to proceed is the following corollary.

\begin{corollary}   \label{cor:comp2} Suppose that $S$ is a Noetherian base scheme. Let $P$ be a set of positive integers which contains primes $p$ for which $S$ \emph{does not} have uniformly bounded $p$-cohomological dimension. Let $L$ be a product of all primes in $P$ then the functor 
$$
\epsilon_*: \SH_{\et}(k)[\frac{1}{L}]  
\rightarrow \SH(k)[\frac{1}{L}],
$$ 
preserves colimits.
\end{corollary}

\begin{proof} Let $X: I \rightarrow \SH_{\et}(k)[\frac{1}{L}], i \mapsto X(i)$ be a small diagram. We need to verify that the natural map $\colim \epsilon_*(X(i)) \rightarrow  \epsilon_*\colim(X(i))$ is an equivalence in $\SH(S)[\frac{1}{L}]$. Using the conservative family of localizations:
$$
\{\SH(S)[\frac{1}{L}] \rightarrow \SH(k)_{(\ell)}\}^{\ell\text{ prime}}_{\ell\nmid L},
$$ 
the claim follows from Theorem~\ref{thm:pi-pres-colimits}.
\end{proof}

\subsubsection{} To begin proving Theorem~\ref{thm:pi-pres-colimits} we note that all compact objects are essentially finite colimits and retracts of $T$-desuspensions of suspension spectra of smooth affine $S$-schemes.

\begin{lemma} \label{cpct-retracts} Let $S$ be a base scheme, then $\SH(S)^{\omega}$ is the smallest stable subcategory of $\SH(S)$ closed under finite colimits and retracts of $\Sigma^{-2n,-n}\Sigma^{\infty}_TX_+$ where $X$ is a smooth affine $S$-scheme and $n \in \ZZ$.
\end{lemma}

\begin{proof}Let $\C$ be the smallest subcategory of $\SH(S)$ closed under finite colimits and retracts of $\Sigma^{-2n,-n}\Sigma^{\infty}_TX_+$  where $X$ is a smooth affine $S$-scheme and $n \in \ZZ$. Using Proposition~\ref{prop:gen-stab}.3, we get that all the $\Sigma^{-2n,-n}\Sigma^{\infty}_TX_+$ are compact, and thus $\C \subset \SH(S)^{\omega}.$ Now compact objects are stable under finite colimits and retracts \cite{htt}*{Lemma 5.1.6.4.}. Since Proposition~\ref{prop:gen-stab}.2 also tells us that $\SH(S)$ is generated by objects in $\C$, we are done.
\end{proof}

\subsubsection{} Now, since the functors in~\eqref{eq:main-adjunction} are exact it suffices to know when $\epsilon_*$ preserves infinite coproducts \cite{higheralgebra}*{Proposition 1.4.4.1.3}.  We recall the following easy but important lemma:
\begin{lemma} 
\label{lem:rightadjcolim} 
Suppose $F:\C\rightleftarrows \D:G$ is an adjunction of stable $\infty$-categories where $F,G$ are exact, $\C$ and $\D$ admit small coproducts, and $\C$ is compactly generated. 
Then $G$ preserves all small coproducts if and only $F$ preserves compact objects.
\end{lemma}
\begin{proof} 
If $G$ preserves small coproducts, then \cite{higheralgebra}*{Proposition 1.4.4.1.2} implies $G$ preserves all small colimits, 
and therefore \cite{htt}*{Proposition 5.5.7.2.1} tells us that $F$ preserves compact objects. 
Conversely, 
since $\C$ is compactly generated it is accessible, and so \cite{htt}*{Proposition 5.5.7.2.2} applies to tell us that $G$ preserves all small colimits. 
\end{proof}

\subsubsection{} We can now proceed to the
\begin{proof} [Proof of Theorem~\ref{thm:pi-pres-colimits}] After Lemma~\ref{lem:rightadjcolim} it suffices to prove that $\epsilon^*$ preserves compact objects. Since the \'etale topology is subcanonical the discussion of~\S\ref{subcanonical} tells us that $\epsilon^*\Sigma^{\infty}_{T}X_+ \simeq \Sigma^{\infty}_TX_+$ for any $X \in \Sm_S.$ After Lemma~\ref{cpct-retracts} we thus need only verify that $\Sigma^{\infty}_TX_+$ is compact in $\SH_{\et}(S)_{(\ell)}$ for any $X \in \Sm_S$ which is affine. To do so, we pick a small collection of objects $(\E_j)_{j \in J}$ in $\SH_{\et}(S)_{(\ell)}$.  We compare the descent spectral sequence:
\begin{equation} \label{dss1}
H^p_{\et}(X, \underline{\pi}^{\et}_{t-q, t}(\oplus_j \E_j)) 
\Rightarrow 
\Hom_{\SH_{\et}(k)}(\GG_m^{\wedge t} \wedge \Sigma^{\infty}_TX_+, \oplus_j \E[p+q]),
\end{equation}
with the sum of descent spectral sequences: 
\begin{equation} \label{dss2}
\oplus_j H^p_{\et}(X, \underline{\pi}^{\et}_{t-q, t}(\E_j)) 
\Rightarrow  
\oplus_j  \Hom_{\SH_{\et}(k)}(\GG_m^{\wedge t} \wedge \Sigma^{\infty}_TX_+,\E_j[p+q]).
\end{equation}
Using the cohomological dimension assumptions on the residue fields, the natural map of spectral sequences $\eqref{dss1} \rightarrow \eqref{dss2}$ is an isomorphism (see, for example, \cite{etalemotives}*{Lemma 1.1.7}).  From the same cohomological dimension assumptions on the residue fields we also deduce strong convergence of the spectral sequences (see the discussion in \cite{aktec}*{5.44-5.48}), 
so that the isomorphism on $E^{2}$-pages implies an isomorphism on the abutments.

\end{proof}

\section{Proof of Main Theorems} \label{proof-main}

We now prove the main theorems of this paper in the following order:
\begin{enumerate}
\item Construction of Bott elements in motivic cohomology (\S\ref{sec:bott-mz}).
\item Proof of \'Etale Descent for Bott-inverted motivic cohomology for essentially smooth schemes over a field (Theorem~\ref{thm:motcohcase3}).
\item Construction of Bott elements in algebraic cobordism (\S\ref{sec:bott-algc}).
\item Proof of \'Etale Descent for Bott-inverted algebraic cobordism for essentially smooth schemes over a field (Theorem~\ref{thm:mgl}).
\item Proof of \'Etale Descent for Bott-inverted algebraic cobordism for Noetherian schemes (Theorem~\ref{thm:S-noeth-new}).
\item Proof of \'Etale Descent for Bott-inverted $\MGL$-modules for Noetherian schemes (Theorem~\ref{thm:S-noeth-new-module}).
\item Proof of an integral statement (Theorem~\ref{thm:integralresult}).
\end{enumerate}

\subsection{Bott elements in motivic cohomology} \label{sec:bott-mz}

Our goal in this section is to produce \emph{Bott elements} in motivic cohomology over $\ZZ[\frac{1}{\ell}]$ and, more generally, schemes on which $\ell$ is invertible. To describe this element, we will use the following notation:
For a prime number $\ell$ and $\nu \geq 1$ we define:
\[
\e(\ell^{\nu})=
\begin{cases}
(\ell-1)\ell^{\nu-1} &  \ell\,\text{odd}\\
2^{\nu-2} &  \ell =2, \nu \geq 3\\
2 & \ell =2, \nu =2 \\
1 & \ell =2, \nu =1.
\end{cases}
\]
Here $\e(\ell^{\nu})$ is the exponent of the multiplicative group of units of the cyclic group $\ZZ/\ell^{\nu}$.  The properties demanded of these elements are summarized in the following proposition:
\begin{proposition} \label{prop:bott-ddk} There exists a collection of elements:
\begin{equation} \label{eq:bott-sys}
\{ \tau^{\M\ZZ}_{\ell^{\nu}} \in H_{\mot}^{0, \e(\ell^{\nu})}( \ZZ[\frac{1}{\ell}]; \ZZ/\ell^{\nu}) \}_{\nu \geq 1},
\end{equation}
such that:
\begin{enumerate}
\item for any $\ell$ odd and $\nu > 1$ and $\ell =2$ and $\nu \geq 2$, under the reduction map:
\[
H_{\mot}^{0,\e(\ell^{\nu})}( \ZZ[\frac{1}{\ell}]; \ZZ/\ell^{\nu}) \stackrel{/\ell}{\rightarrow} H_{\mot}^{0,\e(\ell^{\nu})}( \ZZ[\frac{1}{\ell}]; \ZZ/\ell^{\nu-1}),
\] 
the element $\tau^{\M\ZZ}_{\ell^{\nu}}$ maps to $(\tau^{\M\ZZ}_{\ell^{\nu-1}})^{\ell}$.
\item For any field $k$ with $\frac{1}{\ell} \in k$, 
%
let $q: \Spec\,k \rightarrow \Spec\,\ZZ[\frac{1}{\ell}]$ be the canonical map. 
Then the element:
\[
q^*\tau^{\M\ZZ}_{\ell^{\nu}} \in H^{0,\e(\ell^{\nu})}_{\mot}(k; \ZZ/\ell^{\nu}) \cong H^0_{\et}(k,\mu_{\ell^{\nu}}^{\e(\ell^{\nu})}),
\]
is a periodicity operator in \'etale cohomology.
\item Let  $\{ \tau^{\M\ZZ'}_{\ell^{\nu}} \}_{\nu \geq 1}$ be  another choice of a collection as in~\eqref{eq:bott-sys} which satisfies (1) and (2) above. Then the spectra $\M\ZZ/\ell^{\nu}[(\tau^{\M\ZZ'}_{\ell^{\nu}})^{-1}]$ and $\M\ZZ/\ell^{\nu}[(\tau^{\M\ZZ}_{\ell^{\nu}})^{-1}]$ are equivalent.
\end{enumerate}
\end{proposition}
We call any collection as in~\eqref{eq:bott-sys} a \emph{system of $\M\ZZ$ $\ell$-adic Bott elements} while for a fixed $\nu \geq 1$ we call $\tau^{\M\ZZ}_{\ell^{\nu}}$ an \emph{$\M\ZZ$ mod-$\ell^{\nu}$ Bott element}. Pulling back along $f:S \rightarrow \Spec\,\ZZ[\frac{1}{\ell}]$ yields a collection
\begin{equation} \label{eq:tau-x}
\{ (\tau^{\M\ZZ}_{\ell^{\nu}})_S:=f^*\tau^{\M\ZZ}_{\ell^{\nu}} \in H_{\mot}^{0, \e(\ell^{\nu})}( S; \ZZ/\ell^{\nu}) \}_{\nu \geq 1}.
\end{equation}

\subsubsection{} We proceed with the construction of the classes in~\eqref{prop:bott-ddk}. Let $D$ be a Dedekind domain. Then, according to part (5) of Theorem~\ref{thm:integral}, we have an isomorphism:
\[
H^{1,1}_{\mot}(\Spec\,D; \ZZ) \cong D^{\times}.
\]
For any integer $n$, by the universal coefficients theorem, we have an isomorphism:
\begin{equation} \label{units}
H^{0,1}_{\mot}(\Spec\,D; \ZZ/n) \cong \mu_n(D).
\end{equation}
Therefore, for any prime $\ell$ and $\nu \geq 1$, we have an $\M\ZZ$ mod-$\ell$ Bott element over the ring $\ZZ[\frac{1}{\ell}, \zeta_{\ell}]$ where $\zeta_{\ell}$ is a primtive $\ell$-th root of unity. Namely, it is the class of:
 \begin{equation} \label{eq:root1}
 \zeta_{\ell} \in \mu_{\ell}(D) \iso H^{0,1}_{\mot}(\Spec\,\ZZ[\frac{1}{\ell}, \zeta_{\ell}]; \ZZ/\ell),
 \end{equation} which we denote by $\widetilde{\tau}^{\M\ZZ}_{\ell}$. The extension of rings $\ZZ[\frac{1}{\ell}] \rightarrow \ZZ[\frac{1}{\ell}, \zeta_{\ell}]$ induces a Galois extension of Dedekind schemes $\pi_{\ell}: \Spec\,\ZZ[\frac{1}{\ell}, \zeta_{\ell}] \rightarrow \Spec\,\ZZ[\frac{1}{\ell}]$. The automorphism group of this extension is isomorphic to the group of units $(\ZZ/\ell)^{\times}$; we denote the former group by $G_{\ell}$. Crucially, this group is of order prime to $\ell$. The group $G_{\ell}$ acts via the cyclotomic character on the subgroup of units in $\ZZ[\frac{1}{\ell}, \zeta_{\ell}]$ generated by $\zeta_{\ell}$:
 \[
 \chi_{\ell}: G_{\ell} \rightarrow \Aut(\langle \zeta_{\ell} \rangle).
 \]

\begin{lemma} \label{lem:ell-n-power} Let $\ell$ be an odd prime and suppose that $\nu \geq 1$ then 
\begin{enumerate}
\item the element $(\widetilde{\tau}^{\M\ZZ}_{\ell})^{\ell -1} \in H_{\mot}^{0, \ell-1}(\Spec\,\ZZ[\frac{1}{\ell}, \zeta_{\ell}]; \ZZ/\ell)$ is invariant under the action of $G_{\ell}$.
\item the element $(\widetilde{\tau}^{\M\ZZ}_{\ell})^{\e(\ell^{\nu})} \in H_{\mot}^{0, \e(\ell^{\nu})}(\Spec\,\ZZ[\frac{1}{\ell}, \zeta_{\ell}]; \ZZ/\ell)$ is the reduction mod $\ell$ of an element 
\begin{equation} \label{eq:pre-bott}
\widetilde{\tau}^{\M\ZZ}_{\ell^{\nu}} \in H_{\mot}^{0, \e(\ell^{\nu})}(\Spec\,\ZZ[\frac{1}{\ell}, \zeta_{\ell}]; \ZZ/\ell^{\nu}),
\end{equation} which is also $G_{\ell}$-invariant.
\end{enumerate}
\end{lemma}


\begin{proof}  

By the discussion in the previous paragraph, $G_{\ell}$ acts on the subgroup of $H^{0,1}_{\mot}(\Spec\,\ZZ[\frac{1}{\ell}, \zeta_{\ell}]; \ZZ/\ell)$ generated by $\widetilde{\tau}^{\M\ZZ}_{\ell^{\nu}}$ via $\chi_{\ell^{\nu}}$. Hence, the $\ell-1$-st cup power of $\widetilde{\tau}^{\M\ZZ}_{\ell^{\nu}}$ is $G_{\ell}$-invariant.  Now the $\nu$-th Bockstein map, $\beta_{\nu}$, fits into the following exact sequence:
\[
H_{\mot}^{0,*}(\Spec\,\ZZ[\frac{1}{\ell}, \zeta_{\ell}], \ZZ/\ell^{\nu})  \stackrel{/\ell}{\rightarrow} H_{\mot}^{0,*}(\Spec\,\ZZ[\frac{1}{\ell}, \zeta_{\ell}], \ZZ/\ell) \stackrel{\beta_{\nu}}{\rightarrow} H_{\mot}^{1,*}(\Spec\,\ZZ[\frac{1}{\ell}, \zeta_{\ell}], \ZZ/\ell^{\nu-1}),
\]
and is a derivation. Therefore, we have that:
\[
\beta_{\nu}((\tau^{\M\ZZ}_{\ell})^{\ell^{\nu}}) = \ell^{\nu} \beta_{\nu}(\tau^{\M\ZZ}_{\ell}) = 0.
\] 
Hence there is an element $\widetilde{\tau}^{\M\ZZ'}_{\ell^{\nu}} \in H_{\mot}^{0, (\ell-1)(\ell^{\nu})}(\Spec\,\ZZ[\frac{1}{\ell}, \zeta_{\ell}], \ZZ/\ell^{\nu})  $ whose reduction mod $\ell$ is $(\widetilde{\tau}^{\M\ZZ}_{\ell})^{\ell-1}$. By naturality of the Bockstein maps and part (1), we obtain the $G_{\ell}$-invariance statement.

\end{proof} 

The same argument gives the even case with slightly different numerics.

\begin{lemma}  \label{lem:2-n-power} Suppose that $\nu \geq 3$ then 
\begin{enumerate}
\item the element $(\widetilde{\tau}^{\M\ZZ}_{\ell})^{2} \in H_{\mot}^{0, 2}(\Spec\,\ZZ[\frac{1}{\ell}, i]; \ZZ/\ell)$ is invariant under the action of $G_{\ell}$.
\item the element $(\widetilde{\tau}^{\M\ZZ}_{\ell})^{\e(2^{\nu})} \in H_{\mot}^{0, \e(2^{\nu})}(\Spec\,\ZZ[\frac{1}{\ell}, i]; \ZZ/\ell)$ is the reduction mod $\ell$ of an element:
\begin{equation} \label{eq:pre-bott}
\widetilde{\tau}^{\M\ZZ}_{2^{\nu}} \in H_{\mot}^{0, \e(2^{\nu})}(\Spec\,\ZZ[\frac{1}{\ell}, i]; \ZZ/\ell^{\nu}),
\end{equation} which is also $G_{\ell}$-invariant.
\end{enumerate}
\end{lemma}

\subsubsection{} We use transfers in motivic cohomology (which we discuss in the generality of $\MGL$-modules in Appendix~\S\ref{mgl-trsfs}) to descend the Bott element down to $H_{\mot}^{*,*}(\Spec\,\ZZ[\frac{1}{\ell}]; \ZZ/\ell^{\nu})$.
\begin{lemma} \label{lem:bott-trsfs} Let $\ell$ be a prime and $\nu \geq 1$. The natural map:
\[
\epsilon^*_{\ell}:H_{\mot}^{*,*}(\Spec\,\ZZ[\frac{1}{\ell}]; \ZZ/\ell^{\nu}) \rightarrow H_{\mot}^{*,*}(\Spec\,\ZZ[\frac{1}{\ell}, \zeta_{\ell}]; \ZZ/\ell^{\nu}),
\]
factors through an isomorphism:
\begin{equation} \label{eq:fixed}
\epsilon^*_{\ell}:H_{\mot}^{*,*}(\Spec\,\ZZ[\frac{1}{\ell}]; \ZZ/\ell^{\nu}) \stackrel{\cong}{\rightarrow} H_{\mot}^{*,*}(\Spec\,\ZZ[\frac{1}{\ell}, \zeta_{\ell}]; \ZZ/\ell^{\nu})^{G_{\ell}}
\end{equation}
\end{lemma}

\begin{proof} Since $G_{\ell}$ acts trivially on the group $H_{\mot}^{*,*}(\Spec\,\ZZ[\frac{1}{\ell}]; \ZZ/\ell^{\nu})$, the map $\epsilon^*_{\ell}$ factors through the $G_{\ell}$-invariants of $H_{\mot}^{*,*}(\Spec\,\ZZ[\frac{1}{\ell}, \zeta_{\ell}]; \ZZ/\ell^{\nu})$. The claim now follows from a standard transfer argument and the fact that $\ell-1$ is coprime to $\ell^{\nu}$.  Indeed, since $\M\ZZ/\ell^{\nu}$ is an $\MGL$-module, we have a transfer map by Proposition~\ref{prop:trsf-mgl} :
\[
\pi_{\ell*}: H_{\mot}^{*,*}(\ZZ[\frac{1}{\ell}, \zeta_{\ell}]; \ZZ/\ell^{\nu}) \rightarrow H_{\mot}^{*,*}(\ZZ[\frac{1}{\ell}]; \ZZ/\ell^{\nu}).
\]
such that
\begin{enumerate}
\item $\pi_{\ell*}\epsilon^*_{\ell} = (|G_{\ell}|) \cdot = (\ell-1) \cdot$,
\item $\epsilon^*_{\ell}\pi_{\ell*} = \underset{g \in G_{\ell}}{\Sigma g_*}$, where $g_*$ is the action of $g \in G_{\ell}$ on the cohomology group.
\end{enumerate}
The first property implies that we have an injection:
\[
\epsilon^*_{\ell}:H_{\mot}^{*,*}(\Spec\,\ZZ[\frac{1}{\ell}]; \ZZ/\ell^{\nu}) \hookrightarrow H_{\mot}^{*,*}(\Spec\,\ZZ[\frac{1}{\ell}, \zeta_{\ell}]; \ZZ/\ell^{\nu})^{G_{\ell}}.
\]
The second property tells us if $x$ is an $G_{\ell}$-invariant element, then: 
\[
\epsilon^*_{\ell}\pi_{\ell*}x = \underset{g \in G_{\ell}}{\Sigma g_*}x = |G_{\ell}|x.
\] 
Whence the
map:
\[
\epsilon_*/|G_{\ell}|: H_{\mot}^{*,*}(\Spec\,\ZZ[\frac{1}{\ell}, \zeta_{\ell}]; \ZZ/\ell^{\nu})^{G_{\ell}} \hookrightarrow  H_{\mot}^{*,*}(\Spec\,\ZZ[\frac{1}{\ell}, \zeta_{\ell}]; \ZZ/\ell^{\nu}) \rightarrow H_{\mot}^{*,*}(\ZZ[\frac{1}{\ell}]; \ZZ/\ell^{\nu}),
\]
is a canonical inverse to $i^*$.
\end{proof}

\subsubsection{} Now, we give a 

\begin{proof} [Proof of Proposition~\ref{prop:bott-ddk}] For $\ell$ odd and $\nu \geq 1$ and $\ell =2$ and $\nu \geq 3$, Lemmas~\ref{lem:ell-n-power} and~\ref{lem:2-n-power}, furnishes us with an element: 
\[
(\widetilde{\tau}^{\M\ZZ}_{\ell^{\nu}})^{\e(\ell^{\nu})} \in H_{\mot}^{0, \e(\ell^{\nu})}(\Spec\,\ZZ[\frac{1}{\ell}, \zeta_{\ell^{\nu}}]; \ZZ/\ell^{\nu})^{G_{\ell}},
\] depending only on the choice of a primitive $\ell$-th root of unity in $\ZZ[\frac{1}{\ell}, \zeta_{\ell}]$.  The isomorphism in Lemma~\ref{lem:bott-trsfs}, then gives us a \emph{unique} element:
\[
\tau^{\M\ZZ}_{\ell} \in H_{\mot}^{0,\ell-1}(\Spec\,\ZZ[\frac{1}{\ell}]; \ZZ/\ell^{\nu}).
\]
When $\nu > 1$,  take $\widetilde{\tau}^{\M\ZZ}_{\ell^{\nu}} \in H_{\mot}^{0, \e(\ell^{\nu})}(\Spec\,\ZZ[\frac{1}{\ell}, \zeta_{\ell}]; \ZZ/\ell^{\nu})$ as in~\eqref{eq:pre-bott} which is $G_{\ell}$-invariant by Lemma~\ref{lem:ell-n-power}.2 and hence gives us a unique element $\tau^{\M\ZZ}_{\ell^{\nu}} \in H_{\mot}^{0, \e(\ell^{\nu})}(\Spec\,\ZZ[\frac{1}{\ell}]; \ZZ/\ell^{\nu})$ by Lemma~\ref{lem:bott-trsfs}. Now, point (1) of Proposition~\ref{prop:bott-ddk} follows by the commutativity of the diagram:
\begin{equation}
\xymatrix{
H_{\mot}^{0, \e(\ell^{\nu})}(\Spec\,\ZZ[\frac{1}{\ell}]; \ZZ/\ell^{\nu}) \ar[d] \ar[r]^{/\ell}& H_{\mot}^{0, \e(\ell^{\nu})}(\Spec\,\ZZ[\frac{1}{\ell}]; \ZZ/\ell^{\nu-1}) \ar[d]\\
H_{\mot}^{0, \e(\ell^{\nu})}(\Spec\,\ZZ[\frac{1}{\ell}, \zeta_{\ell}]; \ZZ/\ell^{\nu}) \ar[r]^{/\ell}& H_{\mot}^{0, \e(\ell^{\nu})}(\Spec\,\ZZ[\frac{1}{\ell}, \zeta_{\ell}]; \ZZ/\ell^{\nu-1}),
}
\end{equation}
the naturality of the Bockstein sequences, and the construction of the Bott elements. The fact that they are periodicity operators for \'etale cohomology follows from the construction using cyclotomic characters and the fact that the $\e(\ell^{\nu})$'s are the exponents of the multiplicative group of units of the cyclic group $\ZZ/\ell^{\nu}$. 

To address the remaining cases, we note that when $\ell=2$ and $\nu=1$ there is nothing to show. When $\ell =2$ and $\nu=2$, we note that if we take $[i] \in H^{0,1}_{\mot}(\ZZ[\frac{1}{2}, i], \ZZ/4)$, then $[i]^2 = [-1]$ is in $H_{\mot}^{0,2}(\ZZ[\frac{1}{2}], \ZZ/4)$ and we are done as before. The last statement is proved in Lemma~\ref{lem:noambiguity1} below.

\end{proof}

\subsection{Bott inverted motivic cohomology} Let $S$ be a $\ZZ[\frac{1}{\ell}]$-scheme and suppose that we have a collection $\{\tau_{\ell^{\nu}}^{\M\ZZ}\}$ as in~\eqref{eq:tau-x}, then we may employ the formalism of~\S\ref{per} to construct the \emph{Bott-inverted motivic cohomology} spectra $\{\M\ZZ_S/\ell^{\nu}[(\tau^{\M\ZZ}_{\ell^{\nu}})_S^{-1}]\}$. Each $\M\ZZ_S/\ell^{\nu}[(\tau^{\M\ZZ}_{\ell^{\nu}})_S^{-1}]$ is a $\mathcal{E}_{\infty}$-ring spectrum by Proposition~\ref{prop:invert-long}.

\begin{lemma} 
\label{lem:noambiguity1} 
Suppose $\ell$ is coprime to the exponential characteristic of $k$. 
Then any choice of $\tau_{\ell^{\nu}}^{\M\ZZ}$ which satisfies the first two statements of Proposition~\ref{prop:bott-ddk} gives equivalent Bott inverted motivic cohomology $\mathcal{E}_{\infty}$-rings.
\end{lemma}
\begin{proof}
Under the cyclotomic character $\kappa\colon G_{k}\to\ZZ_{\ell}^{\times}$ the Galois module $\mu_{\ell^{\nu}}^{\otimes i}$ corresponds to the $i$th Tate twist of the $\ZZ_{\ell}$-module $\ZZ/\ell^{\nu}$ on which 
$\ZZ_{\ell}^{\times}$ acts by a homomorphism $\ZZ_{\ell}^{\times}\to\Aut(\ZZ/\ell^{\nu})$.
Thus different choices of $\M\ZZ$-theoretic mod-$\ell^{\nu}$ Bott elements coincide up to an automorphism, 
resulting in an equivalence between the corresponding Bott inverted $\mathcal{E}_{\infty}$-rings.
\end{proof}

\subsubsection{} \label{sect:pull-mz-et}
Using Theorem~\ref{thm:integral}.1 and the properties of the Bott element as in Proposition~\ref{eq:bott-sys} we immediately deduce that for any morphism $f: T \rightarrow S$ we have a canonical equivalence of Bott-inverted motivic cohomology:
\begin{equation} \label{eq:mz-pullsback}
f^*\M\ZZ_S/\ell^{\nu}[(\tau^{\M\ZZ}_{\ell^{\nu}})^{-1}] \simeq f^*\M\ZZ_T/\ell^{\nu}[(\tau^{\M\ZZ}_{\ell^{\nu}})^{-1}].
\end{equation}
The equivalence~\eqref{eq:mz-pullsback} shows that Bott inverted motivic cohomology pulls back, we will also need to know how \'{e}tale motivic cohomology pulls back, at least along essentially smooth morphisms.
\begin{lemma}
\label{lem:pull-mz-et} 
For essentially smooth schemes $S$ and $T$ over a field $k$ and $f: T \rightarrow S$ a smooth morphism over $\Spec\,k$, there is an equivalence:
\begin{equation}\label{eq:mz-et-pulls}
f^* \epsilon_*\epsilon^*\M\ZZ_S/\ell^{\nu} \simeq  \epsilon_*\epsilon^*\M\ZZ_T/\ell^{\nu}.
\end{equation}
\end{lemma}
\begin{proof} 
Since $f^*$ commutes with $\epsilon^*$ already on the level of sheaves, we need only prove that $f^*$ commutes with $\epsilon_*$. 
To see this, 
note that these functors have left adjoints $f_{\#}$ and $\epsilon^*$, respectively, 
and there is an equivalence $f_{\#}\epsilon^* \simeq \epsilon^*f_{\#}$ since both sides agree on representables and preserve colimits, being left adjoints. 
\end{proof}

\subsubsection{} \label{sect:bott-motcoh}
We now turn to the analog of our main results for motivic cohomology, which is essentially a consequence of Theorem \ref{thm:IntroMotcoh}. We first give the details of the proof of that result. For the reader who would rather skip this proof, one can always rely on Voevodsky's proof of the Bloch-Kato conjecture to yield the weaker Theorem \ref{thm:IntroMotcoh}, as mentioned in Remark \ref{rem:BK}.

\begin{proof}[Proof of \hbox{Theorem \ref{thm:IntroMotcoh}}] As mentioned in the introduction, we need only handle the case $\Char k=0$,  $\ell=2$ and $\sqrt{-1}\not\in k$. The argument is an addendum to the proof of  \cite[Theorem 4.5]{levine-bott}. Following that proof, we may assume that $k$ is finitely generated over $\QQ$. Let $k_0$ be the algebraic closure of $\QQ$ in $k$ and let $\bar{k}$ be the algebraic closure of $k$. Following the arguments of {\it loc. cit.} it suffices to construct a tower of fields
\[
k=L_0\subset L_1\subset\ldots\subset L_N=\bar{k}
\]
such that $L_i$ is Galois over $L_{i-1}$ and each Galois group  $\Gal(L_i/L_{i-1})$ has 2-cohomological dimension $\le 1$. We let $k_1=k_0(\mu_{2^\infty})$ and $k_2=\bar{\QQ}$, $L_1=L_0k_1$, $L_2=L_0k_2$. As in {\it loc. cit.}, one constructs the  $L_i$ for $i>2$ by using a trancendence basis of $L_2$ over $\bar{\QQ}$ and   $\Gal(L_i/L_{i-1})$ has 2-cohomological dimension $\le 1$ for $i>2$, so the only question is for the layers $L_1/L_0$ and $L_2/L_1$. 

As a quotient of $\Gal(\bar{k}/k)$, $\Gal(L_1/L_0)$ has finite 2-cohomological dimension. Since $\Gal(L_1/L_0)$ is a open subgroup of $\Gal(\QQ(\mu_{2^\infty})/\QQ)=\ZZ_2\times \ZZ/2$ and as an  open subgroup of $\ZZ_2\times \ZZ/2$ with finite 2-cohomological dimension is isomorphic to $\ZZ_2$,  we have $\Gal(L_1/L_0)\cong \ZZ_2$ and thus $\Gal(L_1/L_0)$ has  2-cohomological dimension 1.  It follows from \cite[II, Proposition 9]{serre} that $\Gal(k_2/k_1)$ has  2-cohomological dimension 1 and as an open subgroup of $\Gal(k_2/k_1)$, $\Gal(L_2/L_1)$ has  2-cohomological dimension 1 as well.
\end{proof}


\begin{theorem} 
\label{thm:motcohcase3} 
Let $k$ be a field with exponential characteristic prime to $\ell$ and let $f: S\rightarrow \Spec\,k$ be an essentially smooth scheme over $k$. Suppose that $\cd_{\ell}(k) < \infty$.   For all $\nu \geq 1$, the natural map:
$$
\M\ZZ_S/\ell^{\nu} \rightarrow \M\ZZ_S^{\et}/\ell^{\nu},
$$ 
induces an equivalence in $\SH(S)$: 
$$
\M\ZZ_S/\ell^{\nu}[(\tau_{\ell^{\nu}}^{\M\ZZ})_S^{-1}] 
\overset{\simeq}{\rightarrow}
\M\ZZ_S^{\et}/\ell^{\nu}.
$$
Consequently, we have an equivalence on $\ell$-completions: 
$$
\M\ZZ_S[(\tau^{\M\ZZ})\compl]  \simeq \M\ZZ^{\et}_S\compl.
$$
\end{theorem}

\begin{proof} After Corollary~\ref{cor:utrvspi} and essentially smooth base change (to reduce the essentially smooth case to the smooth case) \cite{hopkinsmorelhoyois}*{Appendix A}, it suffices to prove that for all smooth $k$-schemes $S$, the canonical map:
\[
H^{*,*}_{\mot}(S, \ZZ/\ell^{\nu})[(\tau_{\ell^{\nu}}^{\M\ZZ})_S^{-1}] \rightarrow H^{*}_{\et}(S; \mu_{\ell}^{\otimes *}),
\]
is a isomorphism. This is Theorem \ref{thm:IntroMotcoh}.
\end{proof}

\subsection{Bott elements in algebraic cobordism} \label{sec:bott-algc}
We now proceed to prove \'etale descent results for Bott-inverted algebraic cobordism. We first construct Bott elements in $\MGL$ over a Dedekind domain. To do so we will need further input from \cite{spitzweck-mixed}, where the Hopkins-Morel isomorphism is suitably generalized to Dedekind domains of mixed characteristics. Recall that if $S$ is a base scheme, we have the \emph{Hopkins-Morel map} \cite{spitzweck-integral}*{\S 11.1}:
\begin{equation} \label{eq:hm-map}
\Phi_S: \MGL_S/(x_1, \cdots x_n) \rightarrow \M\ZZ_S.
\end{equation}
Let $R$ be a ring, and let $\sspt_{(R)}$ be the motivic Moore spectrum corresponding to $R$. If $R$ is a localization of $\ZZ$, then smashing with $\sspt_{(R)}$ computes the $R$-localization of $\SH(S)$. The state-of-the-art of the Hopkins-Morel-Hoyois isomorphism is given by \cite{spitzweck-integral}*{Theorem 11.3}. 
\begin{theorem} \label{thm:hmh-iso} Let $S$ be a base scheme and let $R$ be a ring such that for any positive residue characteristic of $p$ of $S$, then the prime $p$ is invertible in $R$. Then, the Hopkins-Morel map~\eqref{eq:hm-map}:
\begin{equation} \label{eq:hm-map-r}
\Phi_S \wedge \sspt_{(R)}: \MGL_S/(x_1, \cdots x_n) \wedge \sspt_{(R)} \rightarrow \M R_S,
\end{equation}
is an equivalence.
\end{theorem}

\subsubsection{} 
\label{subsection:slicesofMGL} Furthermore, we have the following results by the third author on the slices of $\MGL$.

\begin{theorem} \label{prop:mgl-ddk} Let $S$ be an essentially smooth scheme over Dedekind domain $D$. Let $R$ be a localization of $\ZZ$ such that $\Phi_S \wedge \sspt_{(R)}$ is an equivalence. Then
\begin{enumerate}
\item $\MGL_S \wedge \sspt_{(R)}$ is slice complete in the sense of~\S\ref{slice-complete}.
\item For all $k \in \ZZ$, the $k$-th effective cover $f_k\MGL_D \wedge \sspt_{(R)}$ is $k$-connective in the homotopy $t$-structure, i.e., it is in $\SH(D)_{\geq k}$.
\item We have an equivalence of graded $\mathcal{E}_{\infty}$-algebra objects in $\SH(S)$: 
\begin{equation}
\label{equation:slicesMGL}
s_*\MGL_S \wedge \sspt_{(R)} \simeq \M R_S [x_1, ..., x_q, ...],
\end{equation}
where each $x_q$ is assigned degree $q \in \NN$ and corresponds to a copy of $\M R$ suspended $(2q,q)$-times.
\end{enumerate}

\end{theorem}

\begin{proof} The first statement is \cite{spitzweck-mixed}*{Corollary 5.9}, the second is \cite{spitzweck-mixed}*{Proposition 7.1}, the third statement follows from \cite{spitzweck-mixed}*{Theorem 3.1}, and the highly coherent multiplicativity of the slice filtration \cite{operadsslices}, while the second statement comes from Theorem~\ref{thm:hmh-iso}.
\end{proof}

%
%
%
%

In other words, 
$(s_q \MGL_S)\wedge \sspt_{(R)} $ is a sum of $\Sigma^{2q,q}\M R$'s indexed by monomials $x_i$ of total degree $q$. 
For low dimensional examples:
\begin{itemize}
\item $s_0\MGL_S \wedge \sspt_{(R)} \simeq \M R_S$
\item $s_1\MGL_S \wedge \sspt_{(R)} \simeq \Sigma^{2,1} \M R_S \{x_1 \}$
\item $s_2\MGL_S \wedge \sspt_{(R)} \simeq \Sigma^{4,2} \M R_S \{x_1^2 \} \vee \Sigma^{4,2} \M R_S \{x_2 \}$
\item $s_3\MGL_S \wedge \sspt_{(R)} \simeq \Sigma^{6,3} \M R_S \{x_1^3 \} \vee \Sigma^{6,3} \M R_S \{x_1x_2 \} \vee \Sigma^{6,3}\M R_S \{x_3 \}. $
\end{itemize}

\subsubsection{} 
We note that for any base scheme $S$ $s_*(\MGL_S)$ is a graded $\mathcal{E}_{\infty}$ algebra over the $\mathcal{E}_{\infty}$-ring spectrum 
$s_0\MGL_S$; see \cite[\S6 (iv),(v)]{operadsslices}. Hence, if the Hopkins-Morel map~\eqref{eq:hm-map} is an equivalence $s_*(\MGL_S) \wedge \sspt_{(R)}$ is a graded $\mathcal{E}_{\infty}$-algebra over $\M R_S$ and the equivalence in~\eqref{equation:slicesMGL} indicates the $\M R_S$-algebra structure of  $s_*(\MGL_S) \wedge \sspt_{(R)}$. Now since $\epsilon_*\epsilon^*$ is a lax monoidal functor, 
applying it to this $\mathcal{E}_{\infty}$-graded algebra object gives the $\mathcal{E}_{\infty}$-graded algebra $\epsilon_*\epsilon^*s_*(\MGL_S).$

\begin{lemma} 
\label{lem:identification} 
Let $S$ be any base scheme. There is an equivalence of graded algebras over $\epsilon_*\epsilon^*s_0\MGL \simeq \M\ZZ^{\et}$: 
$$
\epsilon_*\epsilon^*s_*(\MGL_S) \wedge \sspt_{(R)} \simeq \M R_S^{\et}[x_1, ...,x_q. ...].
$$ 
\end{lemma}
\begin{proof} 
The terms of the graded spectrum $\epsilon_*\epsilon^*s_*(\MGL_S)$ are just shifts of $\epsilon_*\epsilon^*\M\ZZ_S \simeq \epsilon_*\M\ZZ_S^{\et}$, 
and so we have an equivalence of graded objects. 
The equivalence of algebras follows from the fact that $\epsilon_*\epsilon^*$ is a lax monoidal functor. 
In more detail, we first identify $s_0\MGL_S \wedge \sspt_{(R)}$ with $\M R_S$. 
Then, a graded $\M R$-algebra is just a commutative algebra object in the $\infty$-category $\Fun(\NN^{\delta}, \Mod_{\M R})$ 
with respect to the Day convolution monoidal structure. 
Here $\NN^{\delta}$ is the constant simplicial set on the the set of natural numbers with the monoidal structure obtained from addition. 
The adjunction: 
$$
\epsilon^*: \Mod_{\M R_S} \rightleftarrows \Mod_{\M R_S^{\et}}: \epsilon_*,
$$ 
can be promoted to an adjunction where $\epsilon^*$ is monoidal and $\epsilon_*$ is lax monoidal: 
$$
\epsilon^*: \Fun(\NN^{\delta},  \Mod_{\M R_S}) \rightleftarrows \Fun(\NN^{\delta}, \Mod_{\M R_S^{\et}}):\epsilon_*.
$$ 
The composite $\epsilon_*\epsilon^*$ is lax monoidal, and thus it preserves algebras.
\end{proof}

The unit map induces a map of graded $\mathcal{E}_{\infty}$-$\M\ZZ$-algebras: 
\begin{equation} \label{eqn:graded-map}
s_*(\MGL_S) \wedge \sspt_{(R)} \simeq \M R_S[x_1, ..., x_q, ...] 
\rightarrow 
\epsilon_*\epsilon^*s_*(\MGL_S) \wedge \sspt_{(R)} \simeq \M R_S^{\et}[x_1, ...,x_q. ...].
\end{equation}
Now, suppose that $\ell$ is invertible in $S$ and $\nu \geq 1$, 
we shall consider the map of graded $\mathcal{E}_{\infty}$-$\M\ZZ/\ell$-algebras: 
\begin{equation}\label{eqn:graded-map-ell}
s_*(\MGL_S /\ell^{\nu}) \simeq \M\ZZ_S/\ell^{\nu}[x_1, ..., x_q, ...] 
\rightarrow 
\epsilon_*\epsilon^*s_*(\MGL_S /\ell^{\nu})
\simeq 
(\M\ZZ_S^{\et}/\ell^{\nu})[x_1, ...,x_q. ...].
\end{equation}

\subsubsection{Bott elements in $\MGL$} We now choose Bott elements for $\MGL$. Our method extracts Bott elements for $\MGL$ from its slice spectral sequence, i.e., the Bott elements originate from the Bott elements in motivic cohomology in the sense of Proposition~\ref{prop:bott-ddk}. While there could be other methods to produce the Bott elements for $\MGL$, we go through the trouble of tracing their origins in the slice spectral sequence precisely because the source of \'etale descent for Bott-inverted $\MGL$ comes from \'etale descent for Bott inverted motivic cohomology. In other words, we amplify \'etale descent on the $E_2$-page to the target, after overcoming some difficult convergence issues. Let us begin by recording a vanishing range for motivic cohomology over Dedekind domains.

\begin{lemma} 
\label{lem:mot-vanish}
Suppose that $\Oscr$ is the ring of integers in a number field $F$ such that $\ell$ is invertible in $\Oscr$, then $H^{p,q}_{\mot}(\Spec\,\Oscr; \ZZ/\ell)$ vanishes whenever $p > q$ except when $(p,q) = (2,1)$.
\end{lemma}

\begin{proof} By \cite{spitzweck-integral}*{Corollary 7.12}, we have a cofiber sequence:
\[
\oplus_{\kappa(\mathfrak{p})} i_{\kappa(\mathfrak{p})*}\M\ZZ/\ell_{\kappa(\mathfrak{p})}(-1)[-2] \rightarrow \M\ZZ/\ell_{\Oscr} \rightarrow j_*\M\ZZ/\ell_{F}.
\]
Therefore, for all $p, q \in \ZZ$, we have an exact sequence:
\[
\oplus_{\kappa(\mathfrak{p})} H_{\mot}^{p-2,q-1}(\Spec\,{\kappa(\mathfrak{p})}; \ZZ/\ell) \rightarrow H_{\mot}^{p,q}(\Spec\,\Oscr; \ZZ/\ell) \rightarrow H_{\mot}^{p,q}(\Spec\,F; \ZZ/\ell).
\]
By vanishing of motivic cohomology of fields (see, for example, \cite{hopkinsmorelhoyois}*{Corollary 4.26}), the only range we need to check is when $p=q+1$. In this case the right most term is zero. Now we have an exact sequence:
\[
H_{\mot}^{q,q}(\Spec\,F; \ZZ/\ell) \rightarrow \oplus_{\kappa(\mathfrak{p})} H_{\mot}^{q-1,q-1}(\Spec\,\kappa(\mathfrak{p}); \ZZ/\ell)  \rightarrow H_{\mot}^{q+1,q}(\Spec\,\Oscr; \ZZ/\ell),  
\]
Whenever $q > 1$, the first map identifies with the map boundary map in Milnor $K$-theory:
\[
K_q^M(F) \stackrel{\partial}{\rightarrow} \oplus_{\kappa(\mathfrak{p})} K_{q-1}^M(\kappa(\mathfrak{p})). 
\]
which is always a surjection by \cite{milnor}*{Theorem 2.3} under the identification of motivic cohomology with Milnor $K$-theory of Nesterenko-Suslin and Totaro; see for example \cite{mvw}*{Lecture 5}.
\end{proof}

\subsubsection{} The mod-$\ell$ $\MGL$ Bott elements are produced from the following. 
\begin{lemma} 
\label{lem:bottperm2} 
Any $\M\ZZ$ mod-$\ell$ Bott element $\tau_{\ell}^{\M\ZZ} \in h^{0,\ell-1}(\ZZ[\frac{1}{\ell}])$ defines a permanent cycle in the slice spectral sequence for $\MGL/\ell_{\ZZ[\frac{1}{\ell}]}$.
Consequently, there exists an \emph{$\MGL$ mod-$\ell$ Bott element} $\tau_{\ell}^{\MGL} \in \MGL/\ell_{0,1-\ell}(\ZZ[\frac{1}{\ell}])$ that is detected under the edge map by the element $\tau_{\ell}^{\M\ZZ}$ of order $\ell$. 
\end{lemma}
\begin{proof} 
Consider the slice spectral sequence for $\MGL/\ell_{\ZZ[\frac{1}{\ell},\zeta_{\ell}]}$ in weight $w=-1$: 
\begin{center}
\begin{tikzpicture}[scale=1.0,font=\scriptsize,line width=1pt]
\draw[help lines] (-2.5,0) grid (6.5,4.5);
\foreach \i in {-2,...,6} {\node[label=below:$\i$] at (\i,-.2) {};}
{\draw[fill]
(-2,0) circle (1pt) node[above right=-1pt] {$h^{2,1}$}
(-1,0) circle (1pt) node[above right=-1pt] {$h^{1,1}$}

(0,0) circle (1pt) node[above right=-1pt] {$h^{0,1}$}
(0,1) circle (1pt) node[above right=-1pt] {$h^{2,2}$}

(1,1) circle (1pt) node[above right=-1pt] {$h^{1,2}$}
(1,2.2) circle (1pt)
(1,2) circle (1pt) node[above right=-1pt] {$h^{3,3}$}

(2,1) circle (1pt) node[above right=-1pt] {$h^{0,2}$}
(2,2.2) circle (1pt)
(2,2) circle (1pt) node[above right=-1pt] {$h^{2,3}$}
(2,3.4) circle (1pt)
(2,3.2) circle (1pt)
(2,3) circle (1pt) node[above right=-1pt] {$h^{4,4}$}

(3,2.2) circle (1pt)
(3,2) circle (1pt) node[above right=-1pt] {$h^{1,3}$}
(3,3.4) circle (1pt)
(3,3.2) circle (1pt)
(3,3) circle (1pt) node[above right=-1pt] {$h^{3,4}$}

(4,2.2) circle (1pt)
(4,2) circle (1pt) node[above right=-1pt] {$h^{0,3}$}
(4,3.4) circle (1pt)
(4,3.2) circle (1pt)
(4,3) circle (1pt) node[above right=-1pt] {$h^{2,4}$}

(5,3.4) circle (1pt)
(5,3.2) circle (1pt)
(5,3) circle (1pt) node[above right=-1pt] {$h^{1,4}$}

(6,3.4) circle (1pt)
(6,3.2) circle (1pt)
(6,3) circle (1pt) node[above right=-1pt] {$h^{0,4}$}
;}
\end{tikzpicture}
\end{center}
Here, we have written:
\[
h^{p,q} := H_{\mot}^{p,q}(\ZZ[\frac{1}{\ell},\zeta_{\ell}]; \ZZ/\ell),
\]
and we stick with this notation for the remainder of the proof. The vanishing range in the display above is furnished by Lemma~\ref{lem:mot-vanish}. 

Now, since no differentials enter or exit the $0$-th stem, we obtain the short exact sequence:
\begin{equation} \label{eq:ses-bott}
0 \rightarrow h^{2,2}(\ZZ[\frac{1}{\ell},\zeta_{\ell}]) \rightarrow \MGL/\ell_{0,-1}(\ZZ[\frac{1}{\ell},\zeta_{\ell}]) \rightarrow h^{0,1}(\ZZ[\frac{1}{\ell},\zeta_{\ell}]) \rightarrow 0.
\end{equation}
Therefore any choice of $\widetilde{\tau}^{\M\ZZ}_{\ell}$ corresponding to a generator of $h^{0,1}(\ZZ[\frac{1}{\ell},\zeta_{\ell}]) \iso \mu_{\ell}(\ZZ[\frac{1}{\ell},\zeta_{\ell}])$ (as in~\eqref{eq:root1})
determines an element $\tau_{\ell}^{\MGL} \in \MGL/\ell_{0,-1}(\ZZ[\frac{1}{\ell},,\zeta_{\ell}])$. 
This disposes of the case $\ell=2$.

Suppose first that $\ell \geq 5$. Consider the Galois extension $\ZZ[\frac{1}{\ell}] \subset \ZZ[\frac{1}{\ell},\zeta_{\ell}]$ with Galois group $G_{\ell}$. Let $\tau \in \MGL/\ell_{0,-1}(\ZZ[\frac{1}{\ell},\mu_{\ell}])$ be picked as in the procedure of the preceding paragraph. 
Since $\MGL/\ell$ is a homotopy commutative and associative ring spectrum for $\ell \geq 5$, 
the element $\tau^{\ell-1} \in \MGL/\ell_{0,1-\ell}(\ZZ[\frac{1}{\ell},\zeta_{\ell}])$ is Galois invariant. 
By the transfer constructed in Proposition~\ref{prop:trsf-mgl}, 
the isomorphism \eqref{eq:fixed} gives a unique element in $\tau^{\MGL}_{\ell} \in \MGL/\ell_{0,1-\ell}(\ZZ[\frac{1}{\ell}])$ such that $\tau^{\MGL}_{\ell}$ maps to $\tau^{\ell-1}$ under the natural map:
\[
\MGL/\ell_{0,1-\ell}(\ZZ[\frac{1}{\ell}]) \rightarrow \MGL/\ell_{0,1-\ell}(\ZZ[\frac{1}{\ell},\zeta_{\ell}]).
\] 
Since the element $\tau^{\M\ZZ}_{\ell}$ is constructed in the same way in motivic cohomology, by first producing it over $\ZZ[\frac{1}{\ell},\zeta_{\ell}])$ and then utilizing the transfer (see Proposition~\ref{prop:bott-ddk}),
we may arrange $\tau^{\MGL}_{\ell}$ to be detected by a $\M\ZZ$-theoretic mod-$\ell$ Bott element $\tau_{\ell}^{\M\ZZ} \in h^{0,\ell-1} \in E^1_{0,0,1-\ell}(\MGL/\ell).$

Now suppose that $\ell =3$. In this case, 
$\MGL/3$ \emph{does not} have an associative or commutative multiplication (although it does have a unital multiplication) since this is the case for the mod $3$ Moore spectrum \cite{Oka}. 
We examine the slice spectral sequence of weight $1-\ell = -2$:
\begin{center}
\begin{tikzpicture}[scale=1.0,font=\scriptsize,line width=1pt]
\draw[help lines] (-2.5,0) grid (6.5,4.5);
\foreach \i in {-2,...,6} {\node[label=below:$\i$] at (\i,-.2) {};}
{\draw[fill]
(-2,0) circle (1pt) node[above right=-1pt] {$h^{2,2}$}

(-1,1) circle (1pt) node[above right=-1pt] {$h^{3,3}$}
(-1,0) circle (1pt) node[above right=-1pt] {$h^{1,2}$}

(0,0) circle (1pt) node[above right=-1pt] {$h^{0,2}$}
(0,1) circle (1pt) node[above right=-1pt] {$h^{2,3}$}
(0,2.2) circle (1pt)
(0,2) circle (1pt) node[above right=-1pt] {$h^{4,4}$}

(1,1) circle (1pt) node[above right=-1pt] {$h^{1,3}$}
(1,2.2) circle (1pt)
(1,2) circle (1pt) node[above right=-1pt] {$h^{3,4}$}
(1,3.4) circle (1pt)
(1,3.2) circle (1pt)
(1,3) circle (1pt) node[above right=-1pt] {$h^{5,5}$}

(2,1) circle (1pt) node[above right=-1pt] {$h^{0,3}$}
(2,2.2) circle (1pt)
(2,2) circle (1pt) node[above right=-1pt] {$h^{2,4}$}
(2,3.4) circle (1pt)
(2,3.2) circle (1pt)
(2,3) circle (1pt) node[above right=-1pt] {$h^{4,5}$}

(3,2.2) circle (1pt)
(3,2) circle (1pt) node[above right=-1pt] {$h^{1,4}$}
(3,3.4) circle (1pt)
(3,3.2) circle (1pt)
(3,3) circle (1pt) node[above right=-1pt] {$h^{3,5}$}

(4,2.2) circle (1pt)
(4,2) circle (1pt) node[above right=-1pt] {$h^{0,4}$}
(4,3.4) circle (1pt)
(4,3.2) circle (1pt)
(4,3) circle (1pt) node[above right=-1pt] {$h^{2,5}$}

(5,3.4) circle (1pt)
(5,3.2) circle (1pt)
(5,3) circle (1pt) node[above right=-1pt] {$h^{1,5}$}

(6,3.4) circle (1pt)
(6,3.2) circle (1pt)
(6,3) circle (1pt) node[above right=-1pt] {$h^{0,5}$};

\draw[black,->]
(0,0) -- (-1,1)
;}
\end{tikzpicture}
\end{center}
To show $\tau_{3}^{\M\ZZ}$ is an permanent cycle we need to examine the single differential displayed above. 
We recall that (over any base) for \emph{all primes} $\ell$, $s_{q}\MGL/\ell$ is an $\M\ZZ/\ell$-module \cite[\S6 (iv),(v)]{operadsslices} since $s_{0}\MGL/\ell \simeq s_{0}\unit/\ell\simeq\M\ZZ/\ell$ 
by \cite{levine-coniveau}, \cite{voevodsky-open}.
The first differential in the slice spectral sequence for $\MGL/\ell$ is induced by the composite map:
\[ 
s_{q}\MGL/\ell
\to 
\Sigma^{1,0}\smash f_{q+1}\MGL/\ell
\to 
\Sigma^{1,0}\smash s_{q+1}\MGL/\ell. 
\] 
Hence the differential of interest is an operation in the Steenrod algebra $\mathcal{A}^{\ast,\ast}_{\ell}\iso\M\ZZ/\ell^{\ast,\ast}\M\ZZ/\ell$ of \cite{hoyois-kelly-ostvaer}, \cite{voevodsky-power}.
As an algebra $\mathcal{A}^{\ast,\ast}_{\ell}$ is generated by the reduced power operations: 
$$
\mathcal{P}^{i}\in\mathcal{A}^{2i(\ell-1),i(\ell-1)}_{\ell},
$$ 
for $i\geq 1$, 
the Bockstein operation:
$$
\beta\in\mathcal{A}^{1,0}_{\ell}, 
$$ 
and the operations given by multiplication by mod-$\ell$ motivic cohomology classes in $h^{\ast,\ast}$.
As noted in \cite[\S8]{voevodsky-power} the Bockstein operation satisfies: 
\begin{equation}
\label{equation:bockstein-deriv}
\beta^{2}=0, 
\beta(uv)=\beta(u)v+(-1)^{r}u\beta(v); u\in h^{r,\ast}, v\in h^{\ast,\ast}.
\end{equation}
Now for all odd primes $\ell$, 
by inspection of degrees we see the reduced power operation $\mathcal{P}^{i}$ acts trivially on $E^{1}_{p,q,w}(\MGL/\ell)$ for all $i\geq 1$ (and all weights $w$). 
Thus all the possibly nontrivial differentials are obtained from powers of $\beta$ and mod-$\ell$ cohomology classes.

To check triviality of the $d_{1}$-differential exiting bidegree $(0,0)$ we note that all $(3,1)$-operations $h^{0,\ell-1}\to h^{3,\ell}$ arise from powers of $\beta$ and 
a class $u$ in either $h^{0,1}$ or $h^{1,1}$.
If $u\in h^{0,1}$ then $\beta^{2}=0$ ensures triviality.
If $u\in h^{1,1}$ then we claim that $\beta(u)=0$. But now since $\ell =3$, we note that  $H^{2,1}(\Spec\,\ZZ[\frac{1}{\ell}, \zeta_{\ell}], \ZZ/\ell)=0$ since the Picard group of $\Spec\,\ZZ[\frac{1}{\ell}, \zeta_{\ell}]$ is zero\footnote{This argument does not work for a general prime $\ell$ --- particularly those such that $\QQ(\sqrt{\ell})$ has class number larger than one.}. Indeed, $\ZZ[\frac{1}{\ell}, \zeta_{\ell}]$ is unique factorization domain, since it is the localization of the ring of integers of $\QQ(\sqrt{-3})$ which has class number one \cite[Page 37]{neukirch}. Hence, we conclude from \eqref{equation:bockstein-deriv}. 

For $\ell=3$, there are no more differentials to check as we see from the weight $-2$ spectral sequence displayed above and, 
therefore, 
we can conclude that $\tau_{3}^{\M\ZZ}$ is indeed a permanent cycle.
\end{proof}

\subsubsection{} 
\label{bott-choice} 
Up to~\S\ref{sec:gradediso}, we work over the base scheme $\Spec\,\ZZ[\frac{1}{\ell}]$. We now seek $\MGL$-theoretic Bott elements for prime powers. This is easily done in the presence of a unital multiplication on $\MGL/\ell^{\nu}$ using Bockstein arguments. 
However, due to the notorious lack of a multiplicative structure on $\MGL/2$\cite{Oka} one cannot readily make sense of inverting $\tau_{2}^{\MGL}$ in general.
Instead, 
we will form the Bott inversion of $\MGL/2$ via the action of a mod-$4$ Bott element $\tau_{4}^{\MGL}$, 
and similarly for $\MGL/2^{\nu}$ and mod-$2^{\nu}$ Bott elements $\tau_{2^{\nu}}^{\MGL}$ for all $v\geq 2$.
As in \cite{aktec},  
at odd primes we argue via a Bockstein spectral sequence.
Again, the Bott element $\tau_{\ell^{\nu}}^{\MGL}$ is detected in the slice spectral sequence by the Bott element $\tau_{\ell^{\nu}}^{\M\ZZ}$ for $\M\ZZ$. For the rest of this section, we write:
\[h^{p,q} := H^{p,q}(\Spec\,\ZZ[\frac{1}{\ell}]; \ZZ/\ell^{\nu}),
\] for some prime $\ell$, some $\nu \geq 1$. The context will always make it clear.

\subsubsection{} We begin with $\ell=2$ and the motivic version:
\begin{equation}
\label{equation:okapairing1}
\unit/4 \wedge \unit/2 \to \unit/2, 
\end{equation}
of Oka's module action of the mod-$4$ by the mod-$2$ Moore spectrum \cite[\S6]{Oka}.
If $\E$ is a motivic ring spectrum, 
i.e., 
a monoid in $\Ho(\SH(S))$, 
then \eqref{equation:okapairing1} induces a pairing:
\begin{equation}
\label{equation:okapairing2}
\E/4 \wedge \E/2 \to \E/2.
\end{equation}
According to \cite[Proposition 2.24]{pi1}, 
\eqref{equation:okapairing2} induces a pairing of slice spectral sequences:
\begin{equation}
\label{eq:R-pairing}
E^r_{p,q,w}(\E/4) \otimes E^r_{p',q',w'}(\E/2) \to E^r_{p+p',q+q',w+w'}(\E/2),
\end{equation}
satisfying the Leibniz rule: 
\begin{equation}
\label{equation:leibniz}
d^{r}(a\cdot b)
=
d^{r}(a)\cdot b + (-1)^{p}a\cdot d^{r}(b) \text{ for } a\in E^r_{p,q,w}(\E/4), b\in E^r_{p',q',w'}(\E/2).
\end{equation}
When $\nu\geq 2$, 
we note that $\E/2^{\nu}$ admits a unital multiplication and its slice spectral sequence is one of (unital, but not necessarily associative) algebras satisfying the same type of Leibniz rule 
for the differentials as \eqref{equation:leibniz}.

\begin{lemma}  
\label{lemma:mgl-bott-2powers}  
For $\nu\geq 2$ any $\M\ZZ$-theoretic mod-$2^{\nu}$ Bott element $\tau_{2^{\nu}}^{\M\ZZ} \in h^{0,\e(2^{\nu})}$ defines a permanent cycle of order $2^{\nu}$ in the slice spectral sequence 
for $\MGL/2^{\nu}$:
$$
(\tau_{2^{\nu}}^{\M\ZZ})^{2^{\nu}} \in E^{1}_{0,0,-2^{\nu}\e(2^{\nu})}(\MGL/2^{\nu}) = h^{0,2^{\nu}\e(2^{\nu})}.
$$ 
Hence, there exists an \emph{$\MGL$-theoretic mod-$2^{\nu}$ Bott element} $\tau_{2^{\nu}}^{\MGL} \in \MGL/2^{\nu}_{0,-2^{\nu}\e(2^{\nu})}(\ZZ[\frac{1}{\ell}])$ that is detected by $(\tau_{2^{\nu}}^{\M\ZZ})^{2^{\nu}}$ 
under the edge map in the slice spectral sequence for $\MGL/2^{\nu}$.
\end{lemma}
\begin{proof}
From the degree zero part of $s_*\MGL$ \eqref{equation:slicesMGL} it follows that $\tau_{2^{\nu}}^{\M\ZZ} \in E^{1}_{0,0,-\e(2^{\nu})}(\MGL/2^{\nu})$.
Since $\nu\geq 2$ the unital multiplicative structure on $\MGL/2^{\nu}$ yields a pairing:
\begin{equation}
\label{eq:MGL2power-pairing}
E^r_{p,q,w}(\MGL/2^{\nu}) \otimes E^r_{p',q',w'}(\MGL/2^{\nu}) \to E^r_{p+p',q+q',w+w'}(\MGL/2^{\nu}).
\end{equation}
By the Leibniz rule associated to \eqref{eq:MGL2power-pairing} we deduce the desired vanishing for $r\geq 1$:
\begin{equation}
\label{eq:MGL2power-vanishing}
d_{r}((\tau_{2^{\nu}}^{\M\ZZ})^{2^{\nu}})
=
2^{\nu}(\tau_{2^{\nu}}^{\M\ZZ})^{2^{\nu}-1}d_{r}(\tau_{2^{\nu}}^{\M\ZZ})
=
0.
\end{equation}
The existence of $\tau_{2^{\nu}}^{\MGL}$ follows by strong convergence.
\end{proof}
\begin{remark}
We note that $\tau_{2}^{\M\ZZ}$ need not be a permanent cycle in the slice spectral sequence for $\MGL/2$.
The issue is basically that $d_{1}:h^{0,2}\to h^{3,3}$ maps $(\tau_{2}^{\M\ZZ})^{2}$ to $\rho^{3}$ via the third Steenrod operation $\Sq^{3}$, 
where $\rho\in h^{1,1}$ is the square class of $-1$.
\end{remark}

\subsubsection{} 
For odd primes $\ell$, the analogous lemma holds.
\begin{lemma} 
\label{lem:mgl-bott-oddpowers} 
Let $\ell$ be an odd prime number. 
Any choice of an $\MGL$-theoretic mod-$\ell$ Bott element $\tau_{\ell}^{\MGL}$ yields an $\MGL$-theoretic mod-$\ell^{\nu}$ Bott element 
$\tau_{\ell^{\nu}}^{\MGL} \in \MGL/\ell^{\nu}_{0,-\e(\ell^{\nu})}(\ZZ[\frac{1}{\ell}])$ that is detected by an $\M\ZZ$-theoretic mod-$\ell^{\nu}$ Bott element in the slice spectral sequence for $\MGL/\ell^{\nu}$.
Furthermore, the mod-$\ell$ reduction of $\tau_{\ell^{\nu}}^{\MGL}$ equals $(\tau_{\ell}^{\MGL})^{\ell^{\nu-1}}$.
\end{lemma}
\begin{proof} 
For $\nu\geq 2$ there is the Bockstein exact sequence:
\begin{equation} 
\label{equation:bockstein}
\MGL/\ell^{\nu}_{0,-\e(\ell^{\nu})}(\ZZ[\frac{1}{\ell}])
\to
\MGL/\ell_{0,-\e(\ell^{\nu})}(\ZZ[\frac{1}{\ell}])
\overset{\beta}{\to}
\MGL/\ell^{\nu-1}_{-1,-\e(\ell^{\nu})}(\ZZ[\frac{1}{\ell}]).
\end{equation}
Now choose $\tau_{\ell}^{\MGL}\in\MGL/\ell_{0,1-\ell}(\ZZ[\frac{1}{\ell}])$ as in Lemma~\ref{lem:bottperm2}.
Then $(\tau_{\ell}^{\MGL})^{\ell^{\nu-1}}\in\MGL/\ell_{0,-\e(\ell^{\nu})}(\ZZ[\frac{1}{\ell}])$ is detected by $(\tau_{\ell}^{\M\ZZ})^{\ell^{\nu-1}}$ in the slice spectral sequence for $\MGL/\ell_{(\ZZ[\frac{1}{\ell}])}$.
Since $\beta$ is a derivation, 
we obtain: 
$$
\beta((\tau_{\ell}^{\MGL})^{\ell^{\nu-1}}) 
= 
\ell^{\nu-1}(\tau_{\ell}^{\MGL})^{\ell^{\nu-1}-1}\beta(\tau_{\ell}^{\MGL})
=
0
\in
\MGL/\ell^{\nu-1}_{-1,-\e(\ell^{\nu})}(\ZZ[\frac{1}{\ell}]). 
$$ 
By exactness of \eqref{equation:bockstein} there exists an element $\tau_{\ell^{\nu}}^{\MGL}\in\MGL/\ell^{\nu}_{0,-\e(\ell^{\nu})}(\ZZ[\frac{1}{\ell}])$ mapping to $(\tau_{\ell}^{\MGL})^{\ell^{\nu-1}}(\ZZ[\frac{1}{\ell}])$.
Using the analogous properties for $\tau^{\M\ZZ}_{\ell^{\nu}} \in \M\ZZ/\ell^{\nu}_{0,-\e(\ell^{\nu})}(\ZZ[\frac{1}{\ell}])$ and naturality of the Bockstein sequences, 
we may arrange so that $\tau_{\ell^{\nu}}^{\MGL}$ is detected by $\tau^{\M\ZZ}_{\ell^{\nu}}$.
\end{proof}

With these choices of Bott elements in mind and the notation introduced in \S\ref{proof-main}, we define:
$$
\e_{\MGL}(\ell^{\nu})=
\begin{cases}
\e(\ell^{\nu}) &  \ell\,\text{odd}\\
2^{\nu}\e(2^{\nu}) & \ell =2.\\
\end{cases}
$$

\begin{definition} \label{def:s-bott} Let $\ell$ be a prime, $\nu \geq 1$ and suppose that $S$ is a $\ZZ[\frac{1}{\ell}]$-scheme with structure map $f:S \rightarrow \Spec\,\ZZ[\frac{1}{\ell}]$. Then an \emph{$\MGL$ mod-$\ell^{\nu}$ Bott element} on $S$ is the the class: 
\[(\tau_{\ell^{\nu}}^{\MGL})_S:=f^*\tau_{\ell^{\nu}}^{\MGL} \in \MGL/\ell^{\nu}_{0,-\e_{\MGL}(\ell^{\nu})}(S).
\]
where $\tau_{\ell^{\nu}}^{\MGL} \in \MGL/\ell^{\nu}_{0,-\e_{\MGL}(\ell^{\nu})}(\ZZ[\frac{1}{\ell}])$ obtained in Lemmas~\ref{lem:mgl-bott-oddpowers} and~\ref{lemma:mgl-bott-2powers}. We call the collection of element $\{(\tau_{\ell^{\nu}}^{\MGL})_S \}_{\nu \geq 1}$ a \emph{system of $\ell$-adic $\MGL$ Bott elements on $S$}.
\end{definition}

\subsection{Bott inverted algebraic cobordism} \label{sec:no-ambiguity}  
We made choices when defining Bott elements for $\MGL$. 
As for the case of $\M\ZZ$ in Lemma~\ref{lem:noambiguity1} we note that the corresponding Bott inverted algebraic cobordism spectrum is independent of the choice of a Bott element. We invert the Bott elements again using the discussion in \S\ref{sect:incoh}, i.e., in the homotopy category instead of the $\infty$-category. This is in contrast with the situation we previously described with motivic cohomology and stems from the lack of coherent multiplication in Moore spectra. 
\begin{lemma} 
\label{lem:noambiguity2} 
Suppose that $\tau^{\MGL}_{\ell^{\nu}},\tau'^{\MGL}_{\ell^{\nu}}$ are two choices of $\MGL$-theoretic mod-$\ell^{\nu}$ Bott elements as in Lemma~\ref{lemma:mgl-bott-2powers} for $\ell=2$, 
and Lemma~\ref{lem:mgl-bott-oddpowers} for $\ell$ odd. 
Then the resulting Bott inverted spectra $\MGL/{\ell^{\nu}}[(\tau^{\MGL}_{\ell^{\nu}})^{-1}]$ and $\MGL/{\ell^{\nu}}[(\tau'^{\MGL}_{\ell^{\nu}})^{-1}]$ are equivalent.
\end{lemma}
\begin{proof} 
First, suppose that $\ell$ is an odd prime. 
When $\nu =1$, the Bott elements are constructed from the corresponding Bott elements for $\M\ZZ/\ell$; 
see the proof of Lemma~\ref{lem:bottperm2}. 
Since $\MGL/\ell$ has a unital multiplication, 
we may form the Bott inverted spectrum $\MGL/\ell[(\tau^{\MGL}_{\ell})^{-1}]$ via the formula in~\eqref{incoh-mult}. 
In this case Lemma~\ref{lem:noambiguity1} and the convergence of the Bott inverted slice spectral sequence (Proposition~\ref{prop:conv-invert} below) gives us the desired result.

For $\nu \geq 2$, 
the elements are obtained via Bockstein spectral sequences. 
Therefore, 
any two choices of $\tau_{\ell^{\nu}}^{\MGL}$ differ by an $\ell$-divisible element, 
and so the difference acts nilpotently on the spectrum $\MGL/\ell^{\nu}$.  
Hence, 
the results of inverting both elements, 
using the formula in~\eqref{incoh-mult}, are again equivalent (compare with the case of algebraic $K$-theory discussed in \cite{eventuallysurjects}).

Now let $\ell=2$. 
If $\nu=1$, then the Bott inversion is constructed using the pairing~\eqref{equation:okapairing2}. 
Indeed, 
by Lemma~\ref{lemma:mgl-bott-2powers}, 
we may pick a mod-$4$ Bott element $\tau^{\MGL}_4 \in \pi_{0,-8}\MGL/4$. 
Under the paring~\eqref{equation:okapairing2}, 
$\MGL/2$ is a left $\MGL/4$-module, 
and $\MGL/{\ell^{\nu}}[(\tau^{\MGL}_{2})^{-1}]$ is then obtained by inverting the map: 
$$
\tau^{\MGL}_4 \cdot: \MGL/2 \rightarrow \Sigma^{0,-8}\MGL/2,
$$ 
in the sense discussed in~\S\ref{sect:incoh}, again using~\eqref{incoh-mult}.
If $\nu \geq 2$, then the spectrum $\MGL/2^{\nu}$ has an unital multiplication, 
and so we can form the Bott inverted spectrum $\MGL/{2^{\nu}}[(\tau^{\MGL}_{2^{\nu}})^{-1}]$ as in the case of odd primes. 

The independence of choices for both cases follows by the same arguments as in the odd case using the strong convergence of the Bott inverted slice spectral sequence, 
and the fact that any two choices of the Bott elements involved again differ by a $2$-divisible element.
\end{proof}
After Lemma~\ref{lem:noambiguity2} we define:

\begin{definition} \label{def:bott-invertedmgl} 
For each prime $\ell$ and any $\nu \geq 1$, 
the \emph{Bott inverted algebraic cobordism spectrum} $\MGL/\ell^{\nu}[ (\tau^{\MGL}_{\ell^{\nu}})^{-1}]$ is the spectrum obtained from inverting the Bott elements as described 
in Lemma~\ref{lem:noambiguity2}.
\end{definition}

\subsubsection{} 
\label{sec:gradediso} 
In order to access the homotopy groups of Bott inverted algebraic cobordism, 
we first study the Bott inverted slice spectral sequence. 
For an odd prime $\ell$ and $\nu \geq1$, 
we invert the action of $\tau^{\M\ZZ}_{\ell^{\nu}} \in \pi_{0,-\e_{\MGL}(\ell^{\nu})}(s_0\MGL/\ell^{\nu})$ on the graded $\M\ZZ/\ell$-module $s_*\MGL/\ell^{\nu}$. 
Explicitly, 
by the procedure described in \S\ref{sect:incoh}, 
we start with the map in $\SH(k)$: 
\begin{equation} \label{eq:map-tau-1}
\tau^{\M\ZZ}_{\ell^{\nu}} \cdot : 
\Sigma^{0,-\e_{\MGL}(\ell^{\nu})}\sspt \wedge s_*\MGL/\ell^{\nu} 
\stackrel{\tau^{\M\ZZ}_{\ell^{\nu}} \wedge id}{\rightarrow} 
s_0\MGL/\ell^{\nu} \wedge s_*\MGL/\ell^{\nu} 
\rightarrow 
s_*\MGL/\ell^{\nu},
\end{equation} \label{eq:map-tau-2}
and apply $\Sigma^{0,\e_{\MGL}(\ell^{\nu})}$ to obtain the map:
\begin{equation}
\Sigma^{0,\e_{\MGL}(\ell^{\nu})} \tau^{\M\ZZ}_{\ell^{\nu}} \cdot: 
\ s_*\MGL/\ell^{\nu} 
\rightarrow  
\Sigma^{0,\e_{\MGL}(\ell^{\nu})} s_*\MGL/\ell^{\nu}.
\end{equation}
The graded motivic spectrum $s_*\MGL/\ell^{\nu} [(\tau^{\M\ZZ}_{\ell^{\nu}}){^{-1}}]$  is then calculated by the colimit of graded $s_0\MGL/\ell^{\nu}$-modules:
\begin{equation} \label{eq:ss-invert}
s_*\MGL/\ell^{\nu} \stackrel{\Sigma^{0,\e_{\MGL}(\ell^{\nu})} \tau^{\M\ZZ}_{\ell^{\nu}} \cdot}{\rightarrow}  \Sigma^{0,\e_{\MGL}(\ell^{\nu})} s_*\MGL/\ell^{\nu} \stackrel{\Sigma^{0,2\e_{\MGL}(\ell^{\nu})} \tau^{\M\ZZ}_{\ell^{\nu}}\cdot}{\rightarrow} \Sigma^{0,2\e_{\MGL}(\ell^{\nu})} s_*\MGL/\ell^{\nu} \cdots.
\end{equation}
Since the element  $\tau^{\MGL}_{\ell^{\nu}}$ is detected by $\tau^{\M\ZZ}_{\ell^{\nu}}$ by Lemma~\ref{lem:mgl-bott-oddpowers}, we obtain a spectral sequence:
\begin{equation}\label{ss-inverted}
E^{*}_{p,q,w}(\MGL/\ell^{\nu})[(\tau^{\M\ZZ}_{\ell^{\nu}})^{-1}] \Rightarrow   \MGL/\ell^{\nu}[ (\tau^{\MGL}_{\ell^{\nu}})^{-1}].
\end{equation}
We will give more details on the construction of this spectral sequence in~\S\ref{sec:systemofss} below.

Let us now deal with the prime $2$ and $\nu=1.$ Lemma~\ref{lemma:mgl-bott-2powers} picks out an element $(\tau^{\M\ZZ}_4)^4 \in  E^{1}_{0,0,-8}(\MGL/4)$. 
We then define: 
\begin{equation}\label{eq:map-tau-3}
\tau^{\M\ZZ}_{4} \cdot : \Sigma^{0,-8}\sspt \wedge s_*\MGL/2 \stackrel{(\tau^{\M\ZZ}_4)^4 \wedge id}{\rightarrow} s_0\MGL/4 \wedge s_*\MGL/2 \rightarrow s_*\MGL/2,
\end{equation}
where the last map is defined using the pairing~\eqref{equation:okapairing2}. 
As in~\eqref{eq:ss-invert}, 
we define the graded motivic spectrum $s_*\MGL/2 [(\tau^{\M\ZZ}_2)^{-1}]$ as the colimit of graded spectra:
\begin{equation} \label{eq:ss-invert2}
s_*\MGL/2 \stackrel{\Sigma^{0,8}\tau^{\M\ZZ}_{4} \cdot}{\rightarrow}  \Sigma^{0,8} s_*\MGL/2 \stackrel{\Sigma^{0,16}\tau^{\M\ZZ}_{4} \cdot}{\rightarrow} \Sigma^{0,16} s_*\MGL/2 \cdots.
\end{equation}
For $\nu \geq 2$, we use the unital pairing on $\MGL/2^{\nu}$ and the element $(\tau_{2^{\nu}}^{\M\ZZ})^{2^{\nu}} \in E^{1}_{0,0,-\e_{\MGL}(2^{\nu})}(\MGL/2^{\nu})$ to get a map:
\begin{equation}\label{eq:invert-2nu}
(\tau_{2^{\nu}}^{\M\ZZ})^{2^{\nu}}\cdot: \Sigma^{0,-\e_{\MGL}(2^{\nu})}\sspt \wedge s_*\MGL/2^{\nu} 
\stackrel{(\tau_{2^{\nu}}^{\M\ZZ})^{2^{\nu}} \wedge id}{\rightarrow}
s_0\MGL/2^{\nu} \wedge s_*\MGL/2^{\nu} \rightarrow s_*\MGL/2^{\nu},
\end{equation} 
which we invert using the same formula as in~\eqref{eq:ss-invert}. 
In any event, 
this gives us a spectral sequence of the form~\eqref{ss-inverted} in all cases. The key point to address in these spectral sequences is their \emph{convergence} which we will provide a proof of.

\subsection{Proof of Theorem~\ref{thm:1} for algebraic cobordism} 
We now have all the ingredients to prove Theorem~\ref{thm:1} for $\MGL$. 
We begin by factoring the unit map.
\begin{lemma} \label{lem:tau-invert} Let $S$ be a $\ZZ[\frac{1}{\ell}]$-scheme, $\ell$ a prime and $\nu \geq 1$ and consider the map:
\[
(\tau^{\MGL}_{\ell^{\nu}})_S: \Sigma^{0, -\e_{\MGL}(\ell^{\nu})}\sspt_S \rightarrow \MGL_S/\ell^{\nu},
\]
classifying the Bott element as in Definition~\ref{def:s-bott}. Then, the map in $\SH_{\et}(S)$
\[
\epsilon^*(\tau^{\MGL}_{\ell^{\nu}})_S: \Sigma^{0, -\e_{\MGL}(\ell^{\nu})}\epsilon^*\sspt_S \rightarrow \epsilon^*\MGL_S/\ell^{\nu}
\] 
is invertible and thus $(\tau^{\MGL}_{\ell^{\nu}})_S$ acts invertibly on $\MGL^{\et}_{*,*}$.
\end{lemma}
%

\begin{proof}  
Over any base $S$, the unit map factors as: 
$$
\MGL_S/{\ell}^{\nu} 
\rightarrow 
\colim (f_q \MGL_S)^{\et}/{\ell}^{\nu}
\rightarrow  
\MGL_S^{\et}/{\ell}^{\nu},
$$ 
where the first map induces a commutative diagram:
\begin{equation} \label{eq:fet}
\xymatrix{
\MGL/{\ell^{\nu}}_{0,-\e_{\MGL}(\ell^{\nu})}(S) \ar[r] \ar[d] 
& (\colim (f_q \MGL)^{\et}/{\ell})_{0,-\e_{\MGL}(\ell^{\nu})}(S) \ar[d]\\ 
\pi_{0,-\e_{\MGL}(\ell^{\nu})}s_0\MGL/{\ell}(S) \iso H^{0,\e_{\MGL}(\ell^{\nu})}(S, \ZZ/\ell^{\nu}) \ar[r] 
& \pi_{0,-\e_{\MGL}(\ell^{\nu})} (s_0\MGL)^{\et}/{\ell}^{\nu} (S)\iso H_{\et}^0(S,\mu_{\ell^{\nu}}^{\otimes \e_{\MGL}(\ell^{\nu})}).
}
\end{equation}
Consider  $(\tau_{\ell^{\nu}}^{\MGL})_S\in\MGL/{\ell^{\nu}}_{0,-\e_{\MGL}(\ell^{\nu})}$; it suffices to prove that it maps to an invertible element via the top horizontal map of diagram~\eqref{eq:fet}. To verify the claim, it suffices to check the case that $S=\Spec\,k$, i.e., the spectrum of a field $k$ (using, for example, the fact that $\SH_{\et}$ has the full six functor formalism by \cite{ayoub}*{Corollaire 4.5.47}). 

By construction $\tau_{\ell^{\nu}}^{\MGL} \in\MGL/{\ell^{\nu}}_{0,-\e_{\MGL}(\ell^{\nu})}(k)$ maps to $\tau_{\ell^{\nu}}^{\M\ZZ}\in H^{0,\e_{\MGL}(\ell^{\nu})}(k, \ZZ/\ell^{\nu})$ via the left vertical arrow.  Since we are over a field, $\tau_{\ell^{\nu}}^{\M\ZZ}$ maps to a periodicity operator in $H_{\et}^0(k,\mu_{\ell^{\nu}}^{\otimes \e_{\MGL}(\ell^{\nu})})$ via the bottom horizontal map. This element acts isomorphically on the \'{e}tale slice spectral sequence. Identifying the \'{e}tale slice spectral sequence as the inverted slice spectral sequence by Proposition~\ref{prop:ssse1}, 
such a periodicity operator survives to an element in $\MGL^{\et}/{\ell^{\nu}}_{0,-\e_{\MGL}(\ell^{\nu})}(k)$ since one can check that the computations in Lemmas~\ref{lem:bottperm2}, 
\ref{lemma:mgl-bott-2powers}, 
and \ref{lem:mgl-bott-oddpowers} are not affected by this inversion procedure. Thus we conclude that this element survives to the value of $\tau_{\ell^{\nu}}^{\MGL}$ via the top vertical map, as desired.

\end{proof}

\subsubsection{Field case}  \label{sec:systemofss} 
We first prove the case when the base is a field. To do this, we first give more details on the construction of the spectral sequence~\eqref{ss-inverted} since we will need some of the notation in our discussion of convergence.
Suppose $\{{^i}E\}_{i \in \NN}$ is a collection of spectral sequences with filtered graded groups $\{ {^i}G \}_{i \in \NN}$ together with maps ${^i}\phi:{^i}E \rightarrow{^{i+1}}E$ 
which are compatible with the maps ${^i}\psi: {^i}G \rightarrow {^{i+1}}G$ in the evident sense.

We say ${^i}\phi$ is \emph{of degree $a$} if it takes an element $x$ of degree $|x|$ to an element of degree $|x|+a$.
By passing to the colimit of the system $(\{{^i}E\},{^i}\phi)_{i \in \NN}$ we obtain a spectral sequence with respect to the evident filtration on the target group $G$: 
\begin{equation}
\label{equation:colimss}
E:= 
\colim\,{^i}E
\Rightarrow 
G:= 
\colim\,{^i}G.
\end{equation}

In our main example of interest, we apply $\pi_{*,*}$ to the diagrams~\eqref{eq:ss-invert} and \eqref{eq:ss-invert2}, 
whence we get ${^i}E := E^*_{p,q,w}(\MGL/\ell^{\nu})$ for all $i$, 
together with degree $(0,-\e_{\MGL}(\ell^{\nu}))$ maps: 
$$
\tau^{\M\ZZ}_{\ell^{\nu}}:{^i}E \rightarrow {^{i+1}}E,
$$  
which are compatible with the multiplication by $\tau_{\ell^{\nu}}^{\MGL}$-map:
$$
\tau_{\ell^{\nu}}^{\MGL} \cdot: \pi_{p,w}\MGL/\ell^{\nu} \rightarrow \pi_{p,-\e_{\MGL}(\ell^{\nu})}\MGL/\ell^{\nu},
$$ 
since $\tau^{\M\ZZ}_{\ell^{\nu}}$ detects $\tau_{\ell^{\nu}}^{\MGL}$ by the definition of the latter element as in Lemmas~\ref{lemma:mgl-bott-2powers} and \ref{lem:mgl-bott-oddpowers}. 

\subsubsection{} 
By the discussion in~\S\ref{sec:gradediso} we can identify the first page of the Bott inverted slice spectral sequence and the \'{e}tale slice spectral sequence.

\begin{proposition} 
\label{prop:ssse1} 
Let $k$ be a field with exponential characteristic prime to $\ell$ and let $f: S\rightarrow \Spec\,k$ be an essentially smooth scheme over $k$. 
The unit map $\MGL_S \rightarrow \epsilon_*\epsilon^*\MGL_S$ induces a natural isomorphism of spectral sequences: 
$$
E^{*}_{p,q,w}(\MGL_S/\ell^{\nu})[(\tau^{\M\ZZ}_{\ell^{\nu}})_S^{-1}] 
\simeq 
E^{*,\et}_{p,q,w}(\MGL_S/\ell^{\nu}).
$$ 
\end{proposition}

\begin{proof} 
By Theorem~\ref{thm:motcohcase3} we obtain equivalences:
\begin{eqnarray*}
s_*\MGL_S/\ell^{\nu} [(\tau^{\M\ZZ}_{\ell^{\nu}}){^{-1}}] & \simeq & \M\ZZ_S/\ell^{\nu} [x_1, ..., x_q, ...][(\tau^{\M\ZZ}_{\ell^{\nu}}})_S{^{-1}]  \\
& \simeq & \M\ZZ_S^{\et}/\ell^{\nu}[x_1, ..., x_q, ...],
\end{eqnarray*}
under the natural maps from $s_*\MGL/\ell^{\nu}_S$. 
Indeed for $\ell$ an odd prime this is straightforward. 
For the prime $2$ and $\nu \geq 1$, 
the first term in the above equivalence is inverted using the $2^{\nu}$-th power of the $\tau_{2^{\nu}}^{\M\ZZ}$ as in~\eqref{eq:invert-2nu} and so Theorem~\ref{thm:motcohcase3} still applies. 
When $\ell=2$ and $\nu=1$, 
we use the inversion procedure described in~\eqref{eq:map-tau-3}, 
which is slightly different from the case described in Theorem~\ref{thm:motcohcase3}. 
However, 
the element $\tau^{\M\ZZ}_4$ acts invertibly in mod-$2$ \'{e}tale cohomology, 
which is all one needs to obtain the second equivalence.

Note that, 
as proved in Lemma~\ref{lem:noambiguity2}, 
the above equivalences are independent of the various choices involved in choosing the Bott elements.  
\end{proof}

\subsubsection{} 
For Proposition~\ref{prop:ssse1} to be useful we need to address convergence results.  
In general, inverting an element in a strongly convergent spectral sequence may destroy strong convergence (see \cite{wilson}*{Section 2} for toy examples.) 
In the notation of~\S\ref{sec:systemofss}, 
assume that ${^i}E$ converges strongly to ${^i}G$, 
where the $q$-th filtration of ${^i}G$ is  denoted by ${^i}F_q$. 
Denote by ${^i}E_q$ the $q$-th filtration degree of ${^i}E.$ We require ${^i}E_q = 0$ for $q<0$. 
%
The convergence result we need follows closely the proof of \cite{wilson}*{Proposition 3} but tailor-made for our needs.
\begin{lemma} 
\label{lem:convergence} 
Suppose that $\{ {^i}E \}$ is a directed system of strongly convergent spectral sequences of bigraded groups ${^i}E^{p,w}_{q}$ where $q$ is the filtration degree such that:
\begin{enumerate}
\item The maps ${^i}\phi:{^i}E \rightarrow {^{i+1}}E$ preserve filtration, 
\item The maps are of degree $(0, a)$ for some $a \in \ZZ$,
\item For fixed $p$, there is an $M(p)>0$ such that for all $i>M(p)$ the group $^i E^{p,w}_{q}=0$ for $q<N(p)$.
\end{enumerate}
In this case, \eqref{equation:colimss} is strongly convergent in the sense of \cite{boardman}*{Definition 5.2}.
\end{lemma}

\begin{proof} 
First we check weak convergence. 
Let $E:= \colim{^i}E$ be the colimit spectral sequence. 
As usual $E^q_{\infty} = \lim_{q} Z^q_{\infty}/B^q_{\infty}$ where $Z^q_{\infty} = \lim_r Z^q_r$ is the group of infinite cycles and ${^i}B^q_{\infty} = \colim_r B^q_r$. 
We want to show there is a natural isomorphism:
$$
\colim_{i } {^i}F_q/{^i}F_{q+1} 
\overset{\iso}{\rightarrow} 
E^q_{\infty}.
$$ 
Since ${^i}E$ is strongly convergent, 
in particular weakly convergent we have:
\begin{equation} 
\label{ptwise-weak-conv}
{^i}F_q/{^i}F_{q+1} \iso {^i}Z^q_{\infty}/{^i}B^q_{\infty},
\end{equation}
where, similarly, ${^i}Z^q_{\infty} = \lim_r {^i}Z^q_r$  and ${^i}B^q_{\infty} = \colim_r {^i}B^q_r$. 
Moreover, 
we have:
\begin{eqnarray*}
\colim_{i } {^i}F_q/{^i}F_{q+1}  
\iso  
\colim_{i} \lim_r Z^q_{r}/{^i}B^q_{r}
\iso 
\lim_r \colim_{i}   {^i}Z^q_{r}/{^i}B^q_{r}
= 
E^q_{\infty}.
\end{eqnarray*}
Here the first isomorphism comes from~\eqref{ptwise-weak-conv}, 
the second comes from assumption $(3)$ (so that the limit term is finite), 
and the last equality is by definition.

Next we check completeness.  We denote by $F_q = \colim_i {^i}F_q$, i.e., the induced filtration on the colimit group. We claim that there is an isomorphism:
$$
\lim_q G/F_q \iso G.
$$ 
Using the definition $G^b/F^b_q:=\colim_i ({^i}G^{b+ ia}/{^i}F_q^{b+ia})$ we obtain: 
$$
\lim_q G^b/ F_q^b
\iso
\lim_q  \colim_i ({^i}G^{b+ ia}/{^i}F_q^{b+ia}).
$$ 
On the other hand, 
by completeness of each ${^i}E$ there is an isomorphism:
$$
G^b 
\iso 
\colim_i ({^i}G^{i+ba}) 
\iso 
\colim_i \lim_q ({^i}G^{i+ba})/({^i}F^{i+ba}_q).
$$ 
Since filtered colimits commute with finite limits of abelian groups it suffices to prove that $\lim_q ({^i}G^{i+ba})/({^i}F^{i+ba}_q)$ is a finite limit for $i\gg 0$.
For fixed $(p,w)$, 
assumption $(3)$ implies that ${^i}F^{i+ba}_q/{^i}F^{i+ba}_{q+1}=0$ whenever $i >M(p)$ and $q \geq N(p)$.
Since the filtration $^iF$ of $^iG$ is Hausdorff, ${^i}F^{i+ba}_q=0$.
Thus the limit $\lim_q ({^i}G^{i+ba})/({^i}F^{i+ba}_q)$ is attained at a finite stage.

Next we check that the filtration $F_q$ of $G$ is \emph{Hausdorff}, that is, $\cap_q F_q = 0$. 
We claim that for every nonzero element $x \in G^b$, there exists some ${^i}x\in {^i}G^{b+ia}$ such that:
\begin{itemize}
\item ${^i}x$ maps to $x$ under the canonical map to ${^i}G \rightarrow G$,
\item ${^i}x$ is detected by an element ${^i}y \in {^i}E$ which is nonzero for all iterated compositions of ${^i}f$.
\end{itemize}
Choose an element ${^j}x'$ that maps to $x\neq 0$. 
Then $^{k+j}\psi(x')\neq 0$ are all nonzero for all $k \geq 0$. 
Let us write ${^k} y \in {^{k+j}}E^{b+k,a}$ for an element that detects $^{k+j}\psi(x')$. 
We may assume $k>M(p)$, so that the filtration degree of ${^k}y$ is less than $N(p)$. 
As $k$ increases, the filtration degree of ${^k}y$ can increase but it must eventually be constant, 
because otherwise ${^k}y$ would be zero and therefore it cannot detect $^{k+j}\psi(x')\neq 0$. 
Where the filtration degree becomes constant is exactly where we find ${^i}x$; in other words we let $i = k+j$ where $k \gg0$.

Note that ${^k}y$ survives to the $E_{\infty}$ page of ${^i}E$ since $^{k+j}g(x')\neq 0$ for all $k\geq 0$. Suppose that $q$ is its filtration degree. 
Since the spectral sequences ${^i}E$ are strongly convergent for all $i$, 
we have that $^{k+j}\psi(x')$ is indeed in filtration $q$ and not in any higher filtration. 
Since this always happens as $i$ tends to $\infty$, we are done.
\end{proof}


\subsubsection{} 
In order to apply Lemma~\ref{lem:convergence} to algebraic cobordism we check for vanishing lines. 

\begin{lemma}  
\label{lem:range} 
Let $k$ be a field with exponential characteristic prime to $\ell$. The group $E^1_{p,q,w}(\MGL_k/\ell^{\nu})=0$ if (1) $p>2q$, (2) $q+w>p$, or (3) $2q>p+\cd_{\ell}(k)$.
\end{lemma}
\begin{proof}
Let us write $L_q$ for the set of degree $q$ monomials in the polynomial generators $x_i$ in the slice calculation \eqref{equation:slicesMGL}.
The $E^1$-term $E^1_{p,q,w}(\MGL_k/\ell^{\nu})$ of the slice spectral sequence for $\MGL_k/\ell^{\nu}$ is given as: 
$$
\pi_{p,w}(s_q\MGL_k/\ell^{\nu}) 
\iso
\oplus_{L_q} \pi_{p,w} (\Sigma^{2q,q} \M\ZZ_k/\ell^{\nu})
\iso 
\oplus_{L_q} H^{2q-p, q-w}(k; \ZZ/\ell^{\nu}). 
$$
Hence (1) and (2) follow from standard vanishing results in motivic cohomology. 
By the Bloch-Kato conjecture,  
away from the region specifed in (2), 
there is an isomorphism
$$
H^{2q-p, q-w}(k; \ZZ/\ell^{\nu}) 
\iso 
H^{2q-p}_{\et}(k, \mu_{\ell^{\nu}}^{\otimes q-w}).
$$ 
Thus (3) follows by the definition of $\cd_{\ell}(k)$.
\end{proof}

\begin{proposition} 
\label{prop:conv-invert} 
The spectral sequence $E^1_{p,q,w}(\MGL_k/\ell^{\nu})[(\tau^{\M\ZZ}_{\ell^{\nu}})_k^{-1}]$ is strongly convergent.
\end{proposition}
\begin{proof} 
The convergence of the original spectral sequence is established in \cite{hopkinsmorelhoyois}*{Theorem 8.12}.  In the notation of Lemma~\ref{lem:convergence} we let ${^i}E^{p,w}_q:= E^1_{p,q,w}$.  
Part (1) of Lemma~\ref{lem:convergence} is satisfied since $\tau^{\M\ZZ}_{\ell^{\nu}}$ is in slice filtration zero. 
Part (2) of Lemma~\ref{lem:convergence} is satisfied since multiplication by $\tau^{\M\ZZ}_{\ell^{\nu}}$ is of degree $(0,\e_{\MGL}(\ell^{\nu}))$. 
Part (3) of Lemma~\ref{lem:range} implies part (3) of Lemma~\ref{lem:convergence} is satisfied for $N(p) = \frac{p+\cd_{\ell}(k)}{2}$. 
\end{proof}

We can now proceed to prove our main theorems for essentially smooth schemes over a field.

\begin{lemma} 
\label{lem:reduction1} 
Suppose $f:T \rightarrow S$ is an essentially smooth morphism of base schemes. 
Then for $\E \in \SH(S)$ there is an equivalence $f^*(\E_{\geq d}) \simeq (f^*E)_{\geq d}$.
\end{lemma}
\begin{proof} 
This is exactly~\cite{hopkinsmorelhoyois}*{Lemma 2.2}.
\end{proof}


\begin{lemma}
\label{lem:reduction2} 
For essentially smooth schemes $S$ and $T$ over a field $k$ and $f: T \rightarrow S$ a smooth morphism over $\Spec\,k$, there are equivalences:
\begin{enumerate}
\item $f^*(\MGL_S/\ell^{\nu}[(\tau_{\ell^{\nu}}^{\MGL})_S^{-1}]) \simeq f^*(\MGL_S/\ell^{\nu})[(\tau_{\ell^{\nu}}^{\MGL})_T^{-1}] \simeq \MGL_T/\ell^{\nu}[(\tau_{\ell^{\nu}}^{\MGL})_{T}^{-1}]$, and
\item $f^* \MGL^{\et}_S/\ell^{\nu} \simeq  \MGL^{\et}_T/\ell^{\nu}$.
\end{enumerate}
\end{lemma}
\begin{proof} Statement (1) follows from the construction of Bott elements.
For (2) we need to prove that $f^*$ commutes with $\epsilon_*$. 
These functors have left adjoints $f_{\#}$ and $\epsilon^*$, respectively, 
and there is an equivalence $f_{\#}\epsilon^* \simeq \epsilon^*f_{\#}$ since both sides agree on representables and preserve colimits. 
\end{proof}

\begin{theorem} 
\label{thm:mgl} 
Suppose $k$ is a field, $\ell$ is prime to the exponential characteristic of $k$ such that $\cd_{\ell}(k) < \infty$, and let $S$ be an essentially smooth $k$-scheme.
Then, for all $\nu \geq 1$ the unit of the adjunction~\eqref{eq:pi-adjunction-paper} induces a natural equivalence in $\SH(S)$:
$$
\MGL/{\ell^{\nu}}[(\tau^{\MGL}_{\ell^{\nu}})_{S}^{-1}] 
\overset{\simeq}{\rightarrow} 
\MGL^{\et}/{\ell^{\nu}}.
$$
\end{theorem}
\begin{proof} 
First assume that the field is perfect. Lemmas~\ref{lem:reduction1} and~\ref{lem:reduction2} reduce to the case of $S = \Spec\,k$. 
Over perfect fields, the motivic homotopy sheaves are strictly $\AA^{1}$-invariant \cite{morel-trieste}*{Remark 5.1.13}, 
and isomorphisms are detected on generic points of smooth $k$-schemes \cite{morel-book}*{Example 2.3, Proposition 2.8}. 
For every finitely generated extension $L$ of $k$ and $p, q \in \ZZ$ we claim there is a naturally induced isomorphism:
$$
\pi_{p,q}\MGL/\ell^{\nu}[(\tau^{\MGL}_{\ell^{\nu}})^{-1}](L) 
\overset{\iso}{\rightarrow} 
\pi_{p,q}\MGL^{\et}/\ell^{\nu}(L).
$$ 
To wit, 
recall the unit map $\MGL/{\ell}^{\nu} \rightarrow \MGL^{\et}/{\ell}^{\nu})$ factors though: 
$$
\eta_{\infty}: 
\MGL/{\ell}^{\nu} \simeq \colim f_q\MGL/{\ell}^{\nu} \rightarrow \colim (f_q\MGL)^{\et}/\ell^{\nu},
$$ 
and: 
$$
\eta^{\infty}:  
\colim (f_q\MGL)^{\et}/{\ell}^{\nu} \rightarrow \epsilon_* \epsilon^* \MGL/{\ell}^{\nu}.
$$ 
Under our assumptions, $\eta^{\infty}$ is an equivalence since $\epsilon_*$ preserves colimits by Corollary~\ref{cor:comp2}.
Lemma~\ref{lem:tau-invert} shows $\eta_{\infty}$ factors through $\MGL/{\ell}^{\nu}[(\tau_{\ell^{\nu}}^{\MGL})^{-1}] \rightarrow \MGL^{\et}/{\ell}$.
By Proposition~\ref{prop:ssse1} the maps between the corresponding spectral sequences are isomorphisms. 
Corollary~\ref{cor:et-complete} (which applies since $\cd_{\ell}(L) < \infty$ by, for example, \cite{shatz}*{Theorem 28}) and Proposition~\ref{prop:conv-invert} 
inform us that these spectral sequences are strongly convergent.
Therefore we have an isomorphism on homotopy groups $\pi_{p,q}$. 

To promote the statement to non-perfect fields we perform a standard continuity argument. So let $k$ be an arbitrary field. It suffices to prove that the comparison map $\MGL_k/{\ell^{\nu}}[(\tau^{\MGL}_{\ell^{\nu}})^{-1}] 
\rightarrow 
\MGL_k^{\et}/{\ell^{\nu}}$ is an equivalence. Choose an essentially smooth morphism $f: \Spec\,k \rightarrow \Spec\,L$ where $L$ is perfect. According to Lemma~\ref{lem:reduction2} the comparison map $\MGL_k/{\ell^{\nu}}[(\tau^{\MGL}_{\ell^{\nu}})^{-1}] 
\rightarrow 
\MGL_k^{\et}/{\ell^{\nu}}$ is $f^*$ of the comparison map over $\Spec\,L$. Write $f$ as a cofiltered limit of maps $f_{\alpha}: S_{\alpha} \rightarrow \Spec\,k$ where  each $S_{\alpha} \rightarrow \Spec\,k$ is smooth. Furthermore any $X \in \Sm_k$ can be written as a limit $X \simeq \lim X_{\alpha}$ where each $X_{\alpha}$ is a smooth $S_{\alpha}$-scheme. Now for any $p, q \in \ZZ$ and any $X \in \Sm_k$, continuity \cite{hopkinsmorelhoyois}*{Lemma A.7} gives us equivalences
\begin{equation} \label{eq:cont-1}
\Maps(\Sigma^{p,q}_TX_{+}, \MGL_k/{\ell^{\nu}}[(\tau^{\MGL}_{\ell^{\nu}})^{-1}]) 
\simeq \colim_{\alpha} \Maps(\Sigma^{p,q}_TX_{\alpha +}, f^*_{\alpha}\MGL_L/{\ell^{\nu}}[(\tau^{\MGL}_{\ell^{\nu}})^{-1}] )
\end{equation}
and
\begin{equation} \label{eq:cont-2}
\Maps(\Sigma^{p,q}_TX_{+}, \MGL_k^{\et}/{\ell^{\nu}}) \simeq 
\colim_{\alpha} \Maps(\Sigma^{p,q}_TX_{\alpha}+, f^*_{\alpha +}\MGL_L^{\et}/{\ell^{\nu}})
\end{equation}
respecting the comparison map which goes $\eqref{eq:cont-1}\rightarrow\eqref{eq:cont-2}$. Since each term of the colimit is an equivalence from the case of smooth schemes over perfect fields, we are done.

\end{proof}

\subsubsection{General case}

We now proceed to prove the general case. We would like to thank Tom Bachmann for help in improving the next theorem to its current level of generality.

\begin{theorem} \label{thm:tau-is-loc} Let $\ell$ be a prime, $S$ be a Noetherian $\ZZ[\frac{1}{\ell}]$-scheme of finite dimension, and assume that for all $x \in S$, $\cd_{\ell}(k(x)) < \infty$. Then for all $\nu \geq 1$, $\MGL_S/{\ell^{\nu}}[(\tau^{\MGL}_{\ell^{\nu}})_{S}^{-1}]$ is \'etale local.
\end{theorem}

\begin{proof} Let $U \in \Sm_S$ and let $U_{\bullet} \rightarrow U$ be a hypercover. After Proposition~\ref{prop:loc}, we need to prove that the map:
\[
\Maps(\Sigma^{p,q}\Sigma^{\infty}_TU_+, \MGL_S/{\ell^{\nu}}[(\tau^{\MGL}_{\ell^{\nu}})_{S}^{-1}]) \rightarrow \lim \Maps(\Sigma^{p,q}\Sigma^{\infty}_TU_{\bullet,+}, \MGL_S/{\ell^{\nu}}[(\tau^{\MGL}_{\ell^{\nu}})_{S}^{-1}]) 
\]
is an equivalence for all $p, q \in \ZZ$.  Fix $p, q \in \ZZ$ and take the cofiber, $C(p,q)$, of the map $\colim \Sigma^{p,q}\Sigma^{\infty}_TU_{\bullet,+} \rightarrow \Sigma^{p,q}\Sigma^{\infty}_TU_+$ in $\SH(S)$. Our goal now is to prove that: 
\[
\Maps(C(p,q), \MGL_S/{\ell^{\nu}}[(\tau^{\MGL}_{\ell^{\nu}})_{S}^{-1}]) \simeq 0.
\]
First let us assume that $\MGL_S/{\ell^{\nu}}$ is a unital ring spectrum so that $\MGL_S/{\ell^{\nu}}[(\tau^{\MGL}_{\ell^{\nu}})_{S}^{-1}]$ is as well (since the functor of $(\tau^{\MGL}_{\ell^{\nu}})_{S}$-inversion is lax monoidal). In this situation we have an isomorphism:
\[
[ C(p,q), \MGL_S/{\ell^{\nu}}[(\tau^{\MGL}_{\ell^{\nu}})_{S}^{-1}]] \cong [ \MGL_S/{\ell^{\nu}}[(\tau^{\MGL}_{\ell^{\nu}})_{S}^{-1}] \wedge C(p,q), \MGL_S/{\ell^{\nu}}[(\tau^{\MGL}_{\ell^{\nu}})_{S}^{-1}]].
\]
We claim that $ \MGL_S/{\ell^{\nu}}[(\tau^{\MGL}_{\ell^{\nu}})_{S}^{-1}] \wedge C(p,q) \simeq 0$. Since $\tau^{\MGL}$-inverted algebraic cobordism is stable under base change by Lemma~\ref{lem:reduction2}, we may check this on fields by an application of Lemma~\ref{lem:conserve}, where we apply Theorem~\ref{thm:mgl} which, in particular, tells us that $\tau^{\MGL}$-inverted algebraic cobordism is \'et-local.

The only case that is not covered by the preceding argument is when $\ell =2$ and $\nu=1$, i.e., $\MGL_S/2$. We argue as follows: since $\MGL_S/4$ is a unital ring spectrum, we have a retract diagram:
\[
\MGL_S/2 \wedge \sspt \stackrel{\id \wedge \eta}{\rightarrow} \MGL_S/2 \wedge \MGL_S/4 \rightarrow \MGL_S/2,
\]
where the second map is induced by Oka's module action~\eqref{equation:okapairing1}. This induces a retract diagram of $\MGL_S/4$-modules
\[
\MGL_S/2[(\tau^{\MGL}_{2})_{S}^{-1}]  \wedge \sspt \stackrel{\id \wedge \eta}{\rightarrow} \MGL_S/2 \wedge \MGL_S/4[(\tau^{\MGL}_{2})_{S}^{-1}]  \rightarrow \MGL_S/2[(\tau^{\MGL}_{2})_{S}^{-1}].
\]
Now, since $\MGL_S/4[(\tau^{\MGL}_{2})_{S}^{-1}]$ is \'etale local by the previous paragraph and $\MGL_S/2[(\tau^{\MGL}_{2})_{S}^{-1}]$ is, by definition, obtained by inverting the Bott element from $\MGL_S/4$ (see Lemma~\ref{lemma:mgl-bott-2powers}), we conclude.


\end{proof}

\begin{theorem} \label{thm:S-noeth-new} Let $\ell$ be a prime, $S$ be a Noetherian $\ZZ[\frac{1}{\ell}]$-scheme of finite dimension, and assume that for all $x \in S$, $\cd_{\ell}(k(x)) < \infty$. Then for all $\nu \geq 1$ the unit of the adjunction~\eqref{eq:pi-adjunction-paper}:
\[
\MGL_S/{\ell^{\nu}}[(\tau^{\MGL}_{\ell^{\nu}})_{S}^{-1}] 
\overset{\simeq}{\rightarrow} 
\MGL_S^{\et}/{\ell^{\nu}}.
\]
is an equivalence in $\SH(S)$. 
\end{theorem}

\begin{proof} Theorem~\ref{thm:tau-is-loc} tells us that $\MGL_S/{\ell^{\nu}}[(\tau^{\MGL}_{\ell^{\nu}})_{S}^{-1}]$ is \'etale local, while Lemma~\ref{lem:tau-invert} tells us that $(\tau^{\MGL}_{\ell^{\nu}})_{S}$ is invertible after applying the \'etale localization endofunctor $\epsilon_*\epsilon^*$. The theorem then follows immediately.
\end{proof}


\begin{theorem} \label{thm:S-noeth-new-module} Let $\ell$ be a prime, $S$ be a Noetherian $\ZZ[\frac{1}{\ell}]$-scheme of finite dimension, and assume that for all $x \in S$, $\cd_{\ell}(k(x)) < \infty$. Let $\nu \geq 1$ and suppose that $\MGL/\ell^{\nu}$ is a unital ring spectrum \footnote{The only case missing is $\ell =2$, $\nu=1$.}.  Then, for any $\E \in \Mod_{\MGL}$ the unit of the adjunction~\eqref{eq:pi-adjunction-paper}:
\[
\E_S/{\ell^{\nu}}[(\tau^{\MGL}_{\ell^{\nu}})_{S}^{-1}] 
\overset{\simeq}{\rightarrow} 
\E_S^{\et}/{\ell^{\nu}}.
\]
is an equivalence in $\SH(S)$. 
\end{theorem}

\begin{proof} Just as in the proof of Theorem~\ref{thm:S-noeth-new}, and following the notation of the proof of Theorem~\ref{thm:tau-is-loc}, it suffices to check that:
\[
[C(p,q), \E_S/{\ell^{\nu}}[(\tau^{\MGL}_{\ell^{\nu}})_{S}^{-1}]]  \simeq 0.
\]
Since $\MGL/\ell^{\nu}$ is a unital ring spectrum, $\MGL_S/{\ell^{\nu}}[(\tau^{\MGL}_{\ell^{\nu}})_{S}^{-1}]$ is as well, and thus we have an equivalence:
\[
[C(p,q), \E_S/{\ell^{\nu}}[(\tau^{\MGL}_{\ell^{\nu}})_{S}^{-1}]]  \simeq [C(p,q) \wedge \MGL_S/{\ell^{\nu}}[(\tau^{\MGL}_{\ell^{\nu}})_{S}^{-1}] , \E_S/{\ell^{\nu}}[(\tau^{\MGL}_{\ell^{\nu}})_{S}^{-1}]].
\]
But then $C(p,q) \wedge \MGL_S/{\ell^{\nu}}[(\tau^{\MGL}_{\ell^{\nu}})_{S}^{-1}] \simeq 0$ by Theorem~\ref{thm:S-noeth-new}.

\end{proof}

Since motivic cohomology is an $\MGL$-algebra, we also note the following improvement to Theorem~\ref{thm:motcohcase3}.
\begin{corollary} 
\label{thm:motcohcase4} 
 Let $\ell$ be a prime, $S$ be a Noetherian $\ZZ[\frac{1}{\ell}]$-scheme of finite dimension, and assume that for all $x \in S$, $\cd_{\ell}(k(x)) < \infty$. Let $\nu \geq 1$.  Then the unit of the adjunction~\eqref{eq:pi-adjunction-paper}:
$$
\M\ZZ_S/\ell^{\nu} \rightarrow \M\ZZ_S^{\et}/\ell^{\nu},
$$ 
induces an equivalence in $\SH(S)$: 
$$
\M\ZZ_S/\ell^{\nu}[(\tau_{\ell^{\nu}}^{\M\ZZ})_S^{-1}] 
\overset{\simeq}{\rightarrow}
\M\ZZ_S^{\et}/\ell^{\nu}.
$$
\end{corollary}	

\begin{proof} We remark that the case not covered by Theorem~\ref{thm:S-noeth-new-module}, i.e., $\ell =2, \nu =1$ can be proved by the same argument as in Theorem~\ref{thm:S-noeth-new} noting that $\M\ZZ/2$ does have an $\mathcal{E}_{\infty}$-ring structure by construction \cite{spitzweck-integral}*{\S 4.1.1}.

\end{proof}

\subsection{Integral statements} 
\label{integral} 
In the following we promote our $\ell$-local results to an integral statement --- at least after inverting some primes.
First, let us investigate the rational results.
\begin{lemma} 
\label{lem:rational} 
If $S$ is a Noetherian scheme of finite dimension, there is a canonical equivalence in $\SH(S)_{\QQ}$: 
$$
L_{\M\QQ}\MGL_S
\overset{\simeq}{\to} 
L_{\M\ZZ^{\et}}\MGL_{S,\QQ}.
$$
\end{lemma}
\begin{proof} 
We show $\MGL_S\rightarrow L_{\M\ZZ^{\et}}\MGL_{S,\QQ}$ is an $\M\QQ$-equivalence and $L_{\M\ZZ^{\et}}\MGL_{S,\QQ}$ is $\M\QQ$-local. 
For the $\M\QQ$-equivalence note that $\MGL_S \rightarrow L_{\M\ZZ^{\et}}\MGL_{S,\QQ}$ factors as: 
$$
\MGL_S 
\rightarrow 
L_{\M\ZZ^{\et}}\MGL_S
\rightarrow 
L_{\M\ZZ^{\et}}\MGL_{S,\QQ}.
$$ 
Upon smashing with $\M\ZZ_S^{\et}$ the first map becomes an equivalence, and upon smashing with $\M{\QQ}$ the second map becomes an equivalence. 
For $\M\QQ$-locality, 
by Morel's rational decomposition theorem, 
$\SH(S)_{\QQ} \simeq \DM(S, \QQ) \times \SH(S)_{\QQ}^{-}$ \cite{cisinski-deglise}*{Theorem 16.2.13}. 
It remains to note that $L_{\M\ZZ^{\et}}\MGL_{S,\QQ} \in \DM(S, \QQ).$ This is because it receives a map of ring spectra from $\MGL_S$ and hence is an oriented theory \cite{cisinski-deglise}*{Theorem 14.2.16} which means that $\eta$ acts by zero, in particular it cannot lie in $\SH(S)_{\QQ}^{-}$ as the first Hopf map $\eta$ acts invertibly on this subcategory of $\SH(S)_{\QQ}.$
\end{proof}

\subsubsection{} 
We can now glue our results to prove the following result up to inverting some primes.
\begin{theorem}
\label{thm:integralresult}
Let $S$ be a Noetherian scheme of finite dimension. Let $J$ be a collection of primes which are all invertible in $S$ such that for $\ell \in J$ and any $x \in S$, $\cd_{\ell}(k(x)) < \infty$. Then, the canonical map
\[
\MGL_S \rightarrow \MGL_S^{\et},
\]
induces an equivalence in $\Mod_{\MGL_{S,(J)}}$: 
$$
L_{\M\ZZ^{\et}}\MGL_{S,(J)} 
\overset{\simeq}{\to} 
\MGL^{\et}_{S,(J)}.
$$
\end{theorem}
\begin{proof} 
To begin with, consider the canonical map $\MGL_{S,(J)} \rightarrow \MGL^{\et}_{S,(J)}$. We claim that it factors through $\MGL_{S,(J)} \rightarrow L_{\M\ZZ^{\et}}\MGL_{S,(J)}$. To do so, we first need to prove that $\MGL^{\et}_{S,(J)}$ is $\M\ZZ^{\et}$-local,
i.e., the canonical map (which exists since the target is \'etale-local)
$\MGL_S^{\et} \rightarrow L_{\M\ZZ^{\et}}\MGL_S^{\et}$ is an equivalence in $\SH(S)$. 

By the arithmetic fracture square in Theorem~\ref{thm:arithmetic} we reduce to checking the equivalence on rationalization and $\ell$-adic completions. 
Lemma \ref{lem:rational} tells us that $L_{\M\QQ}\MGL_S \simeq L_{\M\ZZ^{\et}}\MGL_{S,\QQ}$, 
while motivic Landweber exactness tells us that $\MGL_{S,\QQ}$ is exactly computed as $L_{\M\QQ}\MGL_S \simeq \MGL_S \wedge \M\QQ_S$, 
as explained in \cite{landweber}*{Corollary 10.6}.

For a prime $\ell$, 
Lemma~\ref{lem:modell} reduces to proving that $\MGL_S^{\et}\compl$ is $\M\ZZ_S^{\et}/\ell$-local. 
Since the $\infty$-category of $\M\ZZ_S^{\et}/\ell$-local objects is colocalizing, 
i.e., 
stable and closed under small limits, 
it suffices to check that $\MGL_S^{\et}/\ell^{\nu}$ is $\M\ZZ_S^{\et}/\ell$-local for $\nu \geq 1$. 
Since local objects are closed under cofiber sequences such as $\MGL_S^{\et}/\ell^{\nu} \rightarrow \MGL_S^{\et}/\ell^{\nu-1} \rightarrow \MGL_S^{\et}/\ell$ we may assume $\nu =1$. 
It remains to prove that $\MGL_S^{\et}/\ell$ is $\M\ZZ_S^{\et}/\ell$-local.
By conditional convergence, 
see Corollary~\ref{cor:et-complete}, 
there is an equivalence: 
$$
\lim (f^q\MGL_S)^{\et}/\ell 
\simeq 
\MGL_S^{\et}/\ell.
$$ 
Since $\MGL_S$ is effective,
$(s_0\MGL_S)^{\et}/\ell \simeq (f_{-1}\MGL_S)^{\et}/\ell$ and $(f_{-1}\MGL_S)^{\et}/\ell$ is $(s_0\MGL_S)^{\et}/\ell \simeq \M\ZZ_S^{\et}/\ell$-local, 
which is the first induction step. 
Assuming $(f^{q'}\MGL_S)^{\et}/\ell$ is $\M\ZZ_S^{\et}/\ell$-local for all $q' < q$,  
then $(f^{q}\MGL_S)^{\et}/\ell$ is $\M\ZZ_S^{\et}/\ell$-local by the cofiber sequence: 
$$
(s_q\MGL_S)^{\et}/\ell \rightarrow (f^{q'}\MGL_S)^{\et}/\ell \rightarrow (f^{q-1}\MGL_S)^{\et}/\ell.
$$ 
Indeed the term $(f^{q-1}\MGL_S)^{\et}/\ell$ is $\M\ZZ^{\et}/\ell$-local by hypothesis. 
The slices $s_q\MGL_S/\ell$ are $s_0\MGL_S/\ell$-modules by \cite{operadsslices}*{Section 6 (v)} and the functor $\epsilon_*\epsilon^*$ is lax monoidal, 
and thus preserves $\mathcal{E}_{\infty}$-algebras and modules. 
Hence $(s_q\MGL_S)^{\et}/\ell$ is a module over $\M\ZZ_S^{\et}/\ell$. 
We conclude that the limit is $\M\ZZ_S^{\et}/\ell$-local, 
since the $\infty$-category of $\M\ZZ_S^{\et}/\ell$-local objects is colocalizing and thus closed under small limits.

Therefore, we have a comparison map $L_{\M\ZZ^{\et}}\MGL_{S,(J)} \to 
\MGL^{\et}_{S,(J)}$ which we want to prove is an equivalence. Using the arithmetic fracture square in Theorem \ref{thm:arithmetic} again it suffices to check the equivalence rationally and on $\ell$-adic completions.
Rationally there are equivalences:
$$
L_{\M\ZZ^{\et}}\MGL_{S,\QQ} 
\simeq 
L_{\M\QQ} \MGL_S 
\simeq  
\MGL_{S,\QQ}
\simeq  
\MGL^{\et}_{S,\QQ},
$$ 
where the first is due to Lemma~\ref{lem:rational}, 
the second to \cite{landweber}*{Theorem 10.5}, 
and the third holds since rational oriented theories are \'{e}tale local \cite{cisinski-deglise}*{Corollary 14.2.16, Theorem 14.3.4}.

Now to check at the primes. If $\ell \not \in J$, then $\ell$ is invertible in $\MGL_{S,(J)}$ so that $\MGL_{S,(J)}/\ell^{\nu}$ is zero for all $\nu\geq 1$ and thus the claim follows trivially.
If $\ell \in J$, 
using Lemma~\ref{lem:modell} below, 
it suffices to prove that $
\MGL_{S,(J)}\compl 
\rightarrow
\MGL_{S,(J)}^{\et}\compl
$ 
is an $\M\ZZ_S^{\et}/\ell$-equivalence. 
This claim follows from Lemma~\ref{lem:mz-smash-mgl}, which uses our main Theorem~\ref{thm:S-noeth-new}.

\end{proof}

\begin{lemma} \label{lem:mz-smash-mgl} With the notation of Theorem~\ref{thm:integralresult}, let $\ell \in J$ such that $\MGL_S/\ell^{\nu}$ is a unital ring spectrum. Then for all $\nu \geq 1$, the canonical map:
\[
\MGL_S \rightarrow \MGL_S^{\et},
\]
induces a canonical equivalence:
\begin{equation} \label{eq:mz-smash-mgl}
\M\ZZ_S^{\et} \wedge \MGL_S/\ell^{\nu} 
\overset{\simeq}{\rightarrow} 
\M\ZZ_S^{\et} \wedge \MGL_S^{\et}/\ell^{\nu}.
\end{equation}

\end{lemma}

\begin{proof} The map $\MGL_S \rightarrow \MGL_S^{\et}$ induces a map $\M\ZZ_S^{\et} \wedge \MGL_S/\ell^{\nu} 
\rightarrow 
\M\ZZ_S^{\et} \wedge \MGL_S^{\et}/\ell^{\nu}$, after modding out by $\ell^{\nu}$ and applying $\M\ZZ_S^{\et} \wedge$. To prove the result, we proceed via the following string of equivalences:

\begin{eqnarray*}
\M\ZZ_S^{\et} \wedge \MGL_S/\ell^{\nu} & \simeq & (\MGL_S/(x_1, \cdots x_n,\cdots))^{\et} \wedge \MGL_S/\ell^{\nu} \\
& \simeq & (\colim \MGL_S)^{\et} \wedge \MGL_S/\ell^{\nu}\\
& \simeq & \colim \MGL_S^{\et} \wedge \MGL_S/\ell^{\nu}\\
& \simeq & \colim (\MGL_S/\ell^{\nu})^{\et} \wedge \MGL_S \\
& \simeq & \colim (\MGL_S/\ell^{\nu})[(\tau_{\ell^{\nu}}^{\MGL_S})^{-1}] \wedge \MGL_S\\
& \simeq & \colim ((\MGL_S/\ell^{\nu})\wedge \MGL_S)[(\tau_{\ell^{\nu}}^{\MGL_S})^{-1}] \\ 
& \simeq & (\M\ZZ_S \wedge \MGL_S)^{\et}\\
& \simeq & \M\ZZ_S^{\et} \wedge \MGL_S^{\et}.
\end{eqnarray*}

Here, the first equivalence is due to Theorem~\ref{thm:hmh-iso} which holds under the stated hypotheses, the second is merely rewriting the $\MGL_S/(x_1, \cdots x_n,\cdots)$ term as the colimit it is, the third is because of Theorem~\ref{thm:pi-pres-colimits} which holds under the stated hypotheses, the fourth holds simply because \'etale localization is exact, the fifth is our  Theorem~\ref{thm:S-noeth-new}, the sixth holds because $\tau_{\ell^{\nu}}$-inversion can take place in either $\SH$ or in $\Mod_{\MGL/\ell^{\nu}}$, the seventh is a consequence of Theorem~\ref{thm:S-noeth-new-module}, the eighth is because \'etale localization is closed under $\wedge$.

\end{proof}

%
\begin{lemma} 
\label{lem:modell} 
For $\M \in \SH(S)$ and all primes $\ell$ there is a canonical equivalence of endofunctors: 
$$
(L_{\M})\compl \simeq L_{\M/\ell}: \SH(S) \rightarrow \SH(S).
$$
\end{lemma}
\begin{proof} 
We need to prove that for any $\E \in \SH(S)$, 
$(L_{\M}\E)\compl$ is $\M/\ell$-local and the map $\E \rightarrow (L_{\M}E)\compl$ is an $\M/\ell$-equivalence in $\SH(S)$. 
To prove that $(L_{\M}\E)\compl$ is $\M/\ell$-local, 
we let $\F$ be an $\M/\ell$-acyclic spectrum and first consider $\Maps(\F, L_{\M}\E/\ell)$. 
This is equivalent to $\Maps(\F, L_{\M}\E \wedge \sspt/\ell) \simeq \Maps(\F \wedge (\sspt/\ell)^{\vee}, L_{\M}\E)$. 
Since $L_{\M}\E$ is a left adjoint and compatible with the monoidal structure, 
this is furthermore equivalent to $\Maps(L_{\M}(F) \wedge L_{\M}(1/\ell)^{\vee}, L_{\M}\E)$ which is contractible by the assumption on $\F$. 
By induction we get the same result for $\sspt/\ell^{\nu}$ and thus the result upon completion.  
\end{proof}

\subsubsection{} Applying the same argument as in Theorem~\ref{thm:integralresult} for any $\E \in \Mod_{\MGL}$ and using Theorem~\ref{thm:S-noeth-new-module} we get
\begin{theorem}
\label{thm:integralresult-e}
Let $S$ be a Noetherian scheme of finite dimension. Suppose that $J$ is a collection of primes which are all invertible in $S$ and are such that for any $\ell \in J$ and any $x \in S$, $\cd_{\ell}(k(x)) < \infty$. Then, for any $\E \in \Mod_{\MGL}$ there is a naturally induced equivalence in $\Mod_{\MGL_{S,(J)}}$: 
$$
L_{\M\ZZ^{\et}}\E_{S,(J)} 
\overset{\simeq}{\to} 
\E^{\et}_{S,(J)}.
$$
\end{theorem}

\subsubsection{}\label{thomason-theorem} 
We explain how to recover Thomason's theorem in our setting. 
The motivic spectrum $\KGL$ representing algebraic $K$-theory is Landweber exact since the multiplicative formal group defines a Landweber exact $\MU_*$-algebra \cite{landweber}.
Recall further that for Noetherian scheme $S$ we have for all $p, q\in \ZZ$ a Voevodsky's isomorphism \cite{voevodsky-icm}*{Theorem 6.9}, \cite{cisinski}*{Th\'eor\`eme 2.20}: 
$$
\KGL_{-p,-q}(S) \iso \KGL^{p,q}(S)
\iso 
\mathrm{KH}_{2q-p}(S),
$$ 
where $\mathrm{KH}$ denote Weibel's homotopy $K$-theory \cite{weibel-kh}. When $S$ is furthermore regular $\mathrm{KH}$ is equivalent to Thomason-Trobaugh algebraic $K$-theory. 

Under the Todd genus map $\MGL \rightarrow \KGL$, 
the Bott element $\tau_{\ell^{\nu}}^{\MGL} \in \MGL/\ell^{\nu}_{0,-\e_{\MGL}(\ell^{\nu})}$ maps to $\tau_{\ell^{\nu}}^{\KGL} \in \KGL_{0,-\e_{\MGL}(\ell^{\nu})}$. 
Theorem~\ref{thm:S-noeth-new-module} shows there is an equivalence:
\begin{equation}
\label{equation:KGLet}
\KGL/\ell^{\nu}[(\tau^{\KGL}_{\ell^{\nu}})^{-1}] 
\overset{\simeq}{\to} 
\KGL^{\et}/\ell^{\nu}.
\end{equation}
Applying $\Omega^{\infty}_T$ to \eqref{equation:KGLet} tells us that $\KGL/\ell^{\nu}[(\tau^{\KGL}_{\ell^{\nu}})^{-1}]$ satisfies \'{e}tale hyperdescent since $\Omega^{\infty}_T\KGL^{\et}/\ell^{\nu}$ 
does by the characterization of {\'e}tale local objects in $\SH(S)$ in Lemma \ref{prop:loc}. This recovers Thomason's main theorem from \cite{aktec} when the base scheme $S$ is regular Noetherian and proves the analogous result for homotopy $K$-theory in general.

We sketch how our Bott element specializes to one picked by Thomason for $\ell\neq 2$.
By \cite{eventuallysurjects}*{Bott Elements} the $K$-theoretic Bott elements are also chosen using Bockstein arguments, 
in parallel to our maneuvers in~\S\ref{bott-choice} and so we may assume $\nu=1$. 
In Thomason's case, 
this Bott element is also chosen using Galois descent (again, see \cite{eventuallysurjects}*{Bott Elements}) along the extension $k(\zeta_{\ell})/k$. 
This again parallels what we do in~\S\ref{sec:bott-algc}.

Working over $S=\ZZ[\frac{1}{\ell}, \zeta_{\ell}]$, we pick an $\MGL$-theoretic Bott element in $\MGL/\ell_{0,-1}(S)$ which maps to a $\KGL$-theoretic Bott element in $\KGL/\ell_{0,-1}(S) \simeq K_2/\ell(k)$. 
By its origin from the roots of unity, 
the $\KGL$-theoretic Bott element is clearly an element that maps to $\zeta_{\ell} \in K_1(S)$ under the connecting homomorphism $\beta: K_2/\ell(S) \rightarrow K_1(S)$, 
just as Thomason picks his Bott element in this case \cite{aktec}*{\S A.7}.

\section{Applications and consequences}
\label{sect:applications} 
In this section, we gather some applications and consequences of our theorems above.

\subsection{Descriptions of \'etale-local module categories} 
First,
we generalize the result in \cite{horn-hae} which gives an equivalence between modules over Bott inverted motivic cohomology and \'etale motives over a field. In particular we will describe the stable $\infty$-category of $\MGL$-modules with \'etale descent. To begin, suppose that $\E \in \CAlg(\SH(S))$ and denote by $\Mod_E^{\et}$ the full subcategory spanned by $\E$-modules which are $\et$-local, i.e., its image under the forgetful functor $u_E: \Mod_{\E} \rightarrow \SH(S)$ lands in $\SH_{\et}(S)$. By definition, we have the following diagram of adjoints
\begin{equation} \label{eq:e-et}
\xymatrix{
\Mod_{\E} \ar@/^1.2pc/[rr]^{\epsilon^*}\ar[dd]^{u_\E}
&&   
\Mod_{\E}^{\et} \ar[dd]_{u_\E^{\et}}\ar[ll]_{\epsilon_*}\\\\
\SH(S) \ar@/_1.2pc/[rr]^-{\epsilon^*}\ar@/^1.2pc/[uu]^-{-\wedge \E} 
&&  
\SH_{\et}(S) \ar@/_1.2pc/[uu]_-{(-\wedge \E)^{\et}}\ar[ll]_-{\epsilon_*}.
}
\end{equation}
\begin{proposition} \label{prop:e} Let $S$ be a scheme, $\E \in \CAlg(\SH(S))$. Then we have a canonical adjunction of $\SH(S)$-modules:
\begin{equation} \label{eq:e-et-1}
F_\E:\Mod_{\E^{\et}} \rightleftarrows \Mod^{\et}_\E:G_\E,
\end{equation}
which is an equivalence whenever the composite of the right adjoints in~\eqref{eq:e-et}: 
$$\Mod^{\et}_{\E} \rightarrow \SH(S),$$ 
preserves small colimits.
\end{proposition}

\begin{proof} For this proof, we denote the composite right adjoint of~\eqref{eq:e-et} by $R:\Mod^{\et}_{\E} \rightarrow \SH(S)$ and the left adjoint by $L:\SH(S) \rightarrow \Mod^{\et}_{\E}$. We remark that it factors through $\Mod_{\E^{\et}}$ since it takes the unit object in $\Mod^{\et}_{\E}$ to $\E^{\et}$ and is lax monoidal, being the right adjoint of a strong monoidal functor. The resulting functor $G_{\E}:\Mod^{\et}_\E \rightarrow \Mod_{\E^{\et}}$ is the right adjoint of the desired adjunction~\eqref{eq:e-et}, from which the left adjoints exist by the adjoint functor theorem. 

Now, suppose that the functor $R:\Mod^{\et}_{\E} \rightarrow \SH(S)$ preserve colimits (in other words, it preserves filtered colimits since it is an exact functor of stable $\infty$-categories). By \cite{e-kolderup}*{Theorem 3.6}, it suffices to verify the projection formula \cite{e-kolderup}*{Definition 3.5}: for  any $\M \in \SH(S)$ the canonical map: $$RL(\sspt )\wedge \M \simeq \E^{\et} \wedge \M \rightarrow RL(\M),$$ is an equivalence. To prove this, since both functors preserves colimits in the $\M$-variable, it suffices to prove the claim for $\M = \Sigma^{p,q}\Sigma^{\infty}_+X$ for any $p, q \in \ZZ$. By Remark~\ref{rem:susp-sheaf}, we may assume that $p, q= 0$ we are then left to prove that the map:
$$RL(\sspt )\wedge \Sigma^{\infty}_+X \simeq \E^{\et} \wedge \M \rightarrow RL(\Sigma^{\infty}_+X),$$
is an equivalence for $X$ a smooth $S$-scheme. In this case $X$ is a smooth $S$-scheme with structure map $f$, then we have equivalences in $\SH(S)$:
\begin{eqnarray*}
\E^{\et} \wedge \Sigma^{\infty}_TX_+ & = & \epsilon_*\epsilon^*(\E) \otimes f_{\#}f^*\sspt\\
& \simeq & f_{\#}( f^*\epsilon_*\epsilon^*\E \otimes \sspt)\\
& \simeq & f_{\#}( \epsilon_*\epsilon^*f^*\E) \\
& \simeq & \epsilon_*\epsilon^*(f_{\#}f^*\E)\\
& = & RL(\Sigma^{\infty}_TX_+),\\
\end{eqnarray*}
Here, the first equivalence is the smooth projection formula (see, for example, \cite{cisinski-deglise}*{1.1.26}), the second equivalence follows from the argument of Lemma~\ref{lem:reduction2} for $\SH$ using that $f$ is smooth, and the last equivalence follows from the equivalence of the transformation $f_{\#}\epsilon_*\rightarrow \epsilon_*f_{\#}$ on representables which can be seen to hold unstably from the definitions of these functors since the \'etale topology is subcanonical.
%
%
%
\end{proof}


%

\subsubsection{} 
We freely use the machinery of periodization and localization reviewed in \S\ref{appendix:invert}.
In our context we are interested in $\E \in\CAlg(\SH(S))$ and a map $\alpha: x \rightarrow \E$ in $\Mod_{\E}$ corresponding to a homotopy element in $\E$. 
We consider the full subcategory $P_{\alpha}(\Mod_{\E})\subset\Mod_{\E}$ spanned by $\alpha$-periodic objects. 
By \S\ref{par:alg}, 
$\E[\alpha^{-1}]:=P_{\alpha}\E \in \CAlg(\Mod_{\E})$, 
i.e., 
the periodization remains an $\mathcal{E}_{\infty}$-algebra. 
Therefore, 
we may consider the stable $\infty$-category of modules over $P_{\alpha}\E$ in $\E$-modules:
$$
\Mod_{\E[\alpha^{-1}]} := \Mod_{P_{\alpha}E}(\Mod_{\E}).
$$
First we identify $\alpha$-periodic modules with modules over the $\mathcal{E}_{\infty}$-ring of the $\alpha$-periodization of $\E.$ 
\begin{proposition} 
\label{prop:modules} 
Let $\alpha: x \rightarrow \E$ be a morphism in $\Mod_{\E}$ where $x$ is an invertible object.
Then there is an equivalence of $\infty$-categories: 
$$
P_{\alpha}\Mod_{\E} \simeq \Mod_{\E[\alpha^{-1}]}.
$$
\end{proposition}
\begin{proof} 
According to Proposition~\ref{prop:invert-long} for $\M \in \Mod_{\E},$ we have an equivalence between $P_{\alpha}\M$ and the colimit $Q_{\alpha}\M$ of the diagram: 
$$
\M \rightarrow \mapint(x, \M) \rightarrow \mapint(x^{\otimes 2}, \M) \rightarrow \cdots.
$$ 
Moreover, there are equivalences:
\begin{eqnarray*}
Q_{\alpha}\M & \simeq &\colim (\M \rightarrow \M \otimes x^{\vee} \rightarrow \M \otimes (x^{\vee})^{\otimes 2} \rightarrow \cdots)\\
& \simeq & \colim \M \otimes (1 \rightarrow x^{\vee} \rightarrow (x^{\vee})^{\otimes 2} \rightarrow \cdots)\\
& \simeq & \M \otimes \E[\alpha^{-1}],
\end{eqnarray*}
where the first follows from dualizability (since it is, in fact, invertible) of $x$ and the second holds because $\Mod_{\E}$ is presentably symmetric monoidal, 
i.e., 
the tensor product commutes with colimits. 
Thus every $\alpha$-periodic object is canonically an $\E[\alpha^{-1}]$-module.
In other words, 
$P_{\alpha}$ factors through $P_{\alpha}: \Mod_{\E} \rightarrow \Mod_{\E[\alpha^{-1}]}$. 

Conversely, 
since the full subcategory of $\alpha$-local objects is the essential image of $Q_{\alpha}$, 
every $\alpha$-local object is a module over $\E[\alpha^{-1}]$. 
Thus $P_{\alpha}: \Mod_{\E} \rightleftarrows P_{\alpha}\Mod_{\E}: u$ induces the adjunction: 
$$
- \otimes E[\alpha^{-1}]: \Mod_{\E[\alpha^{-1}]} \rightleftarrows P_{\alpha}\Mod_{\E}:u.
$$ 
Since every object in $P_{\alpha}\Mod_{\E}$ is of the form $Q_{\alpha}\M$, 
the functor is essentially surjective. 
Fully faithfulness follows since $Q_{\alpha}$ is computed by tensoring with $-\otimes E[\alpha^{-1}]$ as shown above.
\end{proof}
In other words, $\alpha$-periodization is a smashing localization.

\subsubsection{} 
In the case of motivic cohomology we generalize the main result of \cite{horn-hae} in a highly coherent setting.
\begin{theorem} 
\label{thm:categorical-mz}
Let $\ell$ be a prime, $S$ be a Noetherian $\ZZ[\frac{1}{\ell}]$-scheme of finite dimension. Further assume that for all $x \in S$, $\cd_{\ell}(k(x)) < \infty$.
There is an equivalence of stable $\infty$-categories:
$$
\epsilon^*:P_{\tau^{\M\ZZ}_{\ell^{\nu}}}\Mod_{\M\ZZ_S/\ell^{\nu}} \rightleftarrows \Mod_{\M\ZZ_S^{\et}/\ell^{\nu}}:\epsilon_*.
$$ 
Suppose further that $S$ is a regular Noetherian scheme of finite dimension, over a field $k$ whose exponential charateristic is coprime to $\ell$, then there is an equivalence:
$$
\epsilon^*: P_{\tau^{\M\ZZ}_{\ell^{\nu}}}\DM(S, \ZZ/\ell^{\nu}) \rightleftarrows \DM_{\et}(S; \ZZ/\ell^{\nu}):\epsilon_*.
$$
\end{theorem}
\begin{proof} 
We remark that, by construction \cite{spitzweck-integral}*{\S 4.1.1}, $\M\ZZ_S^{\et}/\ell^{\nu}$ is an $\mathcal{E}_{\infty}$-ring spectrum. The claim then follows immediately from Theorem~\ref{thm:motcohcase4} and Proposition~\ref{prop:modules}.
The last statement follows from the identification of $\DM$ with  modules over motivic cohomology. This latter result is obtained by a combination of the main theorem of \cite{ro} over characteristic zero, Theorem 5.8 of \cite{hoyois-kelly-ostvaer} over positive characteristics and its generalization in \cite{integral-mixed}*{Theorem 3.1}.
\end{proof}

\subsubsection{} 
The following is a version for $\MGL$ and, more generally, $\mathcal{E}_{\infty}$-algebras in $\MGL$ modules. 
This is where the strength of our integral statement in Theorem~\ref{thm:integralresult} comes into play since the motivic spectrum $L_{\M\ZZ^{\et}}\E_{\M,(J)}$ is an $\mathcal{E}_{\infty}$-ring. 
%
%
%

\begin{theorem} 
\label{thm:categorical-mgl} 
Let $S$ be a Noetherian scheme, and $J$ be a collection of primes which are all invertible in $S$ such that for any $\ell \in J$ and any $x \in S$, $\cd_{\ell}(k(x)) < \infty$. Then, the adjunction:
$$
\epsilon_*:\Mod_{\MGL_S} \rightleftarrows \Mod^{\et}_{\MGL_S}:\epsilon_*, 
$$ 
induces an equivalence of presentably symmetric monoidal stable $\infty$-categories: 
$$
\epsilon^*: \Mod_{L_{\M\ZZ^{\et}}\MGL_{S,(J)}} \rightleftarrows \Mod^{\et}_{\MGL_{S,(J)}}: \pi_{*}.
$$ 
More generally, if $\E_{N}$ is an $\mathcal{E}_{\infty}$-ring in $\Mod_{\MGL_S}$, 
then there is an equivalence of stable $\infty$-categories: 
$$
\epsilon^*: \Mod_{L_{\M\ZZ^{\et}}\E_{N,(J)}} \rightleftarrows \Mod^{\et}_{\E_{N,(J)}}::\pi_{*}.
$$
\end{theorem}
\begin{proof} 
First, note that for every $\E \in \SH(S)$ the localization $L_{\E}$ is compatible with the monoidal structure in the sense of \cite{higheralgebra}*{Definition 2.2.1.6, Remark 2.2.1.7}: 
for any $\E$-equivalence $f: X \rightarrow Y$ and $Z \in \SH(S)$, 
$f \wedge Z: X \wedge Z \rightarrow Y \wedge Z$ is an $\E$-equivalence. 
Therefore, 
by \cite{higheralgebra}*{Proposition 2.2.1.9}, 
$L_{\M\ZZ^{\et}\MGL_{(J)}}$ is an $\mathcal{E}_{\infty}$-algebra and Theorem~\ref{thm:integralresult} gives an equivalence of $\mathcal{E}_{\infty}$-algebras in $\SH(S)$: $L_{\M\ZZ^{\et}}\MGL_{(J)} \simeq \MGL_{(J)}^{\et}$. The claimed result then follows from the second part of Proposition~\ref{prop:e}, whose hypothesis is verified by Theorem~\ref{thm:pi-pres-colimits}.
\end{proof}
\subsection{Base change and six functors} Since $f^*$ and $\epsilon_*$ has no reason to commute, it is not clear that \'etale localization of a motivic spectrum commutes with base-change. Our description of \'etale local spectra as Bott-inverted spectra allows us to prove some base change results.
\label{sect:base-change}

\begin{theorem} \label{thm:base-change} Let $\ell$ be a prime and $T, S$ be a Noetherian $\ZZ[\frac{1}{\ell}]$-scheme of finite dimension and assume that for all $x \in S$, $\cd_{\ell}(k(x)) < \infty$. Let $f: T \rightarrow S$ be a morphism, then the canonical map:
\[
f^*\E_S^{\et}/{\ell^{\nu}} \rightarrow \E_T^{\et}/{\ell^{\nu}},
\]
is an equivalence
\end{theorem}

\begin{proof} The claim follows immediately from Theorem~\ref{thm:S-noeth-new}, the fact that $f^*$ commutes with colimits, and our choice of Bott elements from Definition~\ref{def:s-bott}.
\end{proof}

Now, let $\Sch'_{\ZZ[\frac{1}{\ell}]}$ be the full subcategory of $\Sch_{\ZZ[\frac{1}{\ell}]}$ spanned by Noetherian schemes whose field points have finite $\ell$-cohomological dimension. It follows that:
\begin{corollary} \label{cor:cart-sec} For for any $n \geq 1$:
\[
X \mapsto \MGL_X^{\et}/\ell^{\nu}, \M\ZZ^{\et}_X/\ell^{\nu},
\]
defines Cartesian sections of $\SH \rightarrow \Sch'_{\ZZ[\frac{1}{\ell}]}$.
\end{corollary}

We also obtain an integral statement. Let $J$ is a collection of primes and let $\Sch^{(J)} \subset \Sch$ be the full subcategory of schemes spanned by those Noetherian schemes $X$ such that all primes in $J$ are invertible in $X$, and any $\ell \in J$ and any $x \in X$, $\cd_{\ell}(k(x)) < \infty$.

\begin{corollary} \label{cor:cart-sec-int} The functor:
\[
X \mapsto \MGL^{\et}_{X,(J)},
\]
defines a Cartesian section of $\CAlg(\SH) \rightarrow \Sch^{(J)}$.
\end{corollary}

\begin{proof} By Theorem~\ref{thm:integralresult}, it suffices to prove that for any morphism $f: X \rightarrow Y$ in $\Sch^{(J)}$, the canonical map:
\[
f^*L_{\M\ZZ^{\et}}\MGL_{(J), Y} \rightarrow L_{\M\ZZ^{\et}}\MGL_{(J), X},
\]
is an equivalence. We note that for any scheme $S$, $ L_{\M\ZZ^{\et}_{(J)}}\MGL_{(J), S} \simeq L_{\M\ZZ^{\et}}\MGL_{(J), S}$. Hence, by Lemma~\ref{lem:bousfieldvspull}, it then suffices to prove that $\M\ZZ^{\et}_{(J)}$ defines a Cartesian section of $\SH \rightarrow \Sch^{(J)}$. This follows from Corollary~\ref{cor:cart-sec}.
\end{proof}

Thanks to Corollary~\eqref{cor:cart-sec}, it is standard (see, for example \cite{spitzweck-integral}*{\S 10}) that we can form a stable homotopy $2$-functor:
\begin{equation} \label{eq:mgl-6}
\ho(\Mod_{\MGL/\ell^{\nu}}):\Sch^{'\op}_{\ZZ[\frac{1}{\ell}]} \rightarrow \TriCat,
\end{equation}
in the sense of \cite{ayoub}*{Definition 1.4.1} which thus satisfies the full six functors formalism by \cite{ayoub}*{Scholie 1.4.2}. Corollary~\eqref{cor:cart-sec} also feeds into the formalism of premotivic categories which we have already encountered. By \cite{cisinski-deglise}*{Theorem 2.4.50}, we also obtain the six functor formalism for $\Mod_{\MGL/\ell^{\nu}}$. The same remark applies to:
\begin{equation} \label{eq:mgl-6}
\Mod_{\MGL^{\et}_{(J)}}:(\Sch^{(J)})^{\op} \rightarrow \Cat_{\infty},
\end{equation}
using Corollary~\ref{cor:cart-sec-int}. Combining this with Theorem~\ref{thm:categorical-mgl}, we obtain

\begin{corollary} \label{cor:mgl-et} The functor
\begin{equation} \label{eq:mgl-6-int}
\Mod^{\et}_{\MGL_{(J)}}:(\Sch^{(J)})^{\op} \rightarrow \Cat_{\infty},
\end{equation}
satisfies the full six functors formalism.
\end{corollary}

\subsection{The \'{e}tale hyperdescent spectral sequence}
\label{sect:hyp-ss} 
After the discussion in~\S\ref{sect:descent-ss-for-mot}, we obtain an immediate reward in the form of descent spectral sequences.
When $N = \KU_*$, 
so that $\E_N = \KGL$ is the algebraic $K$-theory spectrum, 
we get back Thomason's result \cite{aktec}*{Theorem 4.1}.
\begin{theorem} 
\label{thm:hypdesc}  
Let $\E \in \SH(S)$ be a Landweber exact spectum.
There exist strongly convergent spectral sequences of the form: 
$$
H^p_{\et}(S, \underline{\pi}^{\et}_{-q, -t}(\E^{\et}_{N}/\ell^{\nu})) 
\Rightarrow 
\Hom_{\SH(S)}(\sspt_S, \Sigma^{0,t}\E^{\et}_{N}/\ell^{\nu}[(\tau_{\ell^{\nu}}^{\MGL})_{S}^{-1}][p+q]),
$$ 
$$
H^p_{\et}(S, \underline{\pi}^{\et}_{-q, -t}((\E^{\et}_{N})\compl)) 
\Rightarrow 
\Hom_{\SH(S)}(\sspt_S, \Sigma^{0,t}(\E^{\et}_{N})\compl[(\tau_{\ell^{\nu}}^{\MGL})_{S}^{-1}][p+q]),
$$ 
and:
$$
H^p_{\et}(S, \underline{\pi}^{\et}_{-q, -t}(L_{\M\ZZ^{\et}}\E_{N,{(J)}})) 
\Rightarrow 
\Hom_{\SH(S)}(\sspt_S, \Sigma^{0,t}L_{\M\ZZ^{\et}}\E_{N,{(J)}}[p+q]).
$$ 
\end{theorem}

\subsubsection{} 
In order to access the spectral sequence above, 
we need to compute the \'{e}tale homotopy sheaves: 
$$
\underline{\pi}^{\et}_{p,q}(\E^{\et}_{N}/\ell^{\nu}):\Sm^{\op}_S \rightarrow \Ab,
\underline{\pi}^{\et}_{p,q}((\E^{\et}_{N})\compl):\Sm^{\op}_S \rightarrow \Ab.
$$
These \'etale sheaves also assemble into \'{e}tale sheaves of graded abelian groups: 
$$
\oplus_{p,q} \underline{\pi}^{\et}_{p,q}(\E^{\et}_{N}/\ell^{\nu}):\Sm^{\op}_S \rightarrow \gr\Ab,
\oplus_{p,q} \underline{\pi}^{\et}_{p,q}((\E^{\et}_{N})\compl): \Sm^{\op}_S \rightarrow \gr\Ab.
$$ 
We claim these sheaves are locally constant (see, for example, \cite{mvw}*{Definition 6.7}). 
There is an adjunction between the small and big \'{e}tale (discrete) topoi over $S$:  
$$
\rho^*:\Shv_{\et}(\Et_S) \rightleftarrows  \Shv_{\et}(\Sm_S): \rho_*.
$$ 
This gives rise to an adjunction between sheaves of abelian groups: 
$$
\rho^*: \Ab\Shv_{\et}(\Et_S) \rightleftarrows \Ab\Shv_{\et}(\Sm_S):\rho_*,
$$
and its graded variant:
$$
\rho^*: \gr\Ab\Shv_{\et}(\Et_S) \rightleftarrows \gr\Ab\Shv_{\et}(\Sm_S):\rho_*.
$$


\begin{proposition} 
\label{prop:sheaves}  
For any $p, q \in \ZZ$, the counit map yields an isomorphism: 
$$
\rho^*\rho_*\underline{\pi}^{\et}_{p,q}(\E^{\et}_N/\ell^{\nu}) \overset{\iso}{\rightarrow}  \underline{\pi}^{\et}_{p,q}(\E^{\et}_N/\ell^{\nu}).
$$ 
Consequently the counit map on $\ell$-completions: 
$$
\rho^*\rho_* \underline{\pi}^{\et}_{p,q}((\E^{\et}_N)\compl) \rightarrow \underline{\pi}^{\et}_{p,q}((\E^{\et}_N)\compl),
$$ 
and the graded variants: 
$$
\rho^*\rho_*\oplus_{p,q} \underline{\pi}^{\et}_{p,q}(\E^{\et}_N/\ell^{\nu}) \rightarrow \oplus_{p,q} \underline{\pi}^{\et}_{p,q}(\E^{\et}_N/\ell^{\nu}),
$$
$$
\rho^*\rho_*\oplus_{p,q} \underline{\pi}^{\et}_{p,q}((\E^{\et}_N)\compl) \rightarrow \oplus_{p,q} \underline{\pi}^{\et}_{p,q}((\E^{\et}_N)\compl),
$$
are isomorphisms.
\end{proposition}
\begin{proof} 
The counit map $$\rho^*\rho_*\underline{\pi}^{\et}_{p,q}(\E^{\et}_N/\ell^{\nu}) \rightarrow  \underline{\pi}^{\et}_{p,q}(\E^{\et}_N/\ell^{\nu})$$ is a map of \'{e}tale sheaves on $\Sm_S$. 
To check the isomorphism, we consider stalks. 
Since the residue field of the strict Henselization of the local ring of a smooth $S$-scheme is separably closed, 
we may assume that $S$ is the spectrum of a separably closed field. 
In this case, 
$\underline{\pi}^{\et}_{p,q}(\E^{\et}_N/\ell^{\nu})$ is the locally constant sheaf for the abelian group $\underline{\pi}^{\et}_{p,q}(\E^{\et}_N/\ell^{\nu})(\Spec\,k)$ by \cite{horn-yag}*{Theorem 0.3}. 
The other cases follow readily.
\end{proof}

\subsubsection{} 
As the proof of Proposition~\ref{prop:sheaves} indicates, it will be useful to know the value of the homotopy sheaves over a separably closed field. 
First we have the following computation:

\begin{proposition} 
\label{prop:mgl-et-closed} 
If $k$ is an algebraically closed field, we have isomorphisms of graded abelian groups: 
$$
\oplus_{p,q} \MGL^{\et}/\ell^{\nu}_{p,q}(k) \iso  \oplus_{p} \MU/\ell^{\nu}_{p}[(\tau_{\ell^{\nu}}^{\MGL})_{S}^{-1}],$$
and:
$$
\oplus_{p,q} ((\MGL^{\et})\compl)_{p,q}(k) \iso  \oplus_{p} (\MU\compl)_{p}[(\tau_{\ell^{\nu}}^{\MGL})_{S}^{-1}].
$$
\end{proposition}
\begin{proof} 
From the computation of $s_*\MGL$ reviewed in~\S\ref{subsection:slicesofMGL}, 
and the collapse of the Bott inverted spectral sequence over an algebraically closed field we have an isomorphism:
\begin{equation} \label{eqn:mgl-k-closed}
\oplus_{p,q} \MGL/\ell^{\nu}_{p,q}(k)[(\tau_{\ell^{\nu}}^{\MGL})_{S}^{-1}]  
\iso 
\oplus_{p} \MU/\ell^{\nu}_{p}[(\tau_{\ell^{\nu}}^{\MGL})_{S}^{-1}].
\end{equation}
The desired isomorphisms follow as an instance of Theorem~\ref{thm:mgl}.
\end{proof}

%

More generally, we have:
\begin{proposition} 
\label{prop:lw-et-closed} 
If $k$ is an algebraically closed field, 
we have isomorphisms of graded abelian groups: 
$$\
\oplus_{p,q} \E_N^{\et}/\ell^{\nu}_{p,q}(k) \iso  \oplus_{p}N_p/\ell^{\nu} [(\tau_{\ell^{\nu}}^{\MGL})_{S}^{-1}],
$$ 
and:
$$\oplus_{p,q} ((\E_N^{\et})\compl)_{p,q}(k) \iso  \oplus_{p} (N\compl)_p [(\tau_{\ell^{\nu}}^{\MGL})_{S}^{-1}].
$$
\end{proposition}
\begin{proof} 
There is an isomorphism of graded abelian groups $\E_{N {*,*}} \iso \MGL_{*,*} \otimes_{\MU_*}N_*$ \cite{landweber}*{\S 7}.  
Therefore we have isomorphisms:
\begin{eqnarray*}
\oplus_{p,q} \E_N/\ell^{\nu}_{p,q}(k)[(\tau_{\ell^{\nu}}^{\MGL})_{S}^{-1}] & \iso & (\oplus_{p,q} \MGL/\ell^{\nu}_{p,q}(k)[(\tau_{\ell^{\nu}}^{\MGL})_{S}^{-1}]) \otimes_{\MU_*}N \\
& \iso & (\oplus_{p} \MU/\ell^{\nu}_{p}[(\tau_{\ell^{\nu}}^{\MGL})_{S}^{-1}]) \otimes_{\MU_*}N_*\\
& \iso & \oplus_{p}N_p/\ell^{\nu} [(\tau_{\ell^{\nu}}^{\MGL})_{S}^{-1}]. 
\end{eqnarray*}
The desired isomorphisms follow as an instance of Theorem~\ref{thm:mgl}.
\end{proof}

Propositions \ref{prop:sheaves} and \ref{prop:lw-et-closed} show the coefficients of the descent spectral sequences displayed in Theorem~\ref{thm:hypdesc} are locally constant sheaves 
on a periodized version of $N_*$. 

\subsection{Cellularity} 
Recall the $\infty$-category $\SH(S)^{\cell}$ of \emph{cellular spectra} is the full localizing subcategory of $\SH(S)$ generated by $\Sigma^{p,q}\unit_S$ for $p, q \in \ZZ$ \cite{dugger-isaksen.cell}. 
While the left adjoint functor $\epsilon^*$ preserves cellularity, this isn't quite clear for its right adjoint $\epsilon_*$.
\begin{theorem} 
\label{corollary:cell} 
Let $\ell$ be a prime and $S$ be a Noetherian $\ZZ[\frac{1}{\ell}]$-scheme of finite dimension and assume that for all $x \in S$, $\cd_{\ell}(k(x)) < \infty$.
Suppose that $\E \in\SH(S)$ is an $\MGL$-module which is cellular, then $\E^{\et}/\ell^{\nu}$ is cellular.
\end{theorem}  
\begin{proof} 
Follows from Theorem~\ref{thm:S-noeth-new} because cellular objects are closed under colimits and thus $\E/{\ell^{\nu}}[(\tau^{\MGL}_{\ell^{\nu}})_{S}^{-1}]$ is cellular.
\end{proof}
\begin{remark}
It is unclear whether the $\ell$-completion $\E\compl$ is actually cellular. 
This is a delicate issue as cellular objects are not closed under infinite limits, see \cite{totaro}*{\S5}. 
\end{remark}

%

\section{Appendix A: Periodization and localizations}
\label{appendix:invert}

\subsection{Periodization} \label{per}
We begin by reviewing the process of periodization/inversion of objects in a presentably symmetric monoidal $\infty$-category. 
Our main reference is \cite{hoyois-cdh}*{Section 3}. 
The set-up is as follows: 
$\C$ is a presentably symmetric monoidal $\infty$-category and $S$ a collection of objects of $\C_{/\unit}$
i.e., 
objects $x$ equipped with a map $\alpha: x \rightarrow \unit$. 
Let $\M$ be a $\C^{\otimes}$-module.

\subsubsection{}
\label{sect:incoh} 
Let us first recall the periodization procedure which does not take into account coherence. 
We have already used this procedure in~\S\ref{sec:no-ambiguity} and in other places throughout the paper. 
Here, we consider the category $\Ho(\C)$ as a symmetric monoidal category. 

Suppose $m \in \Ho(\C)$ admits a unital multiplication (which is not necessarily associative) and $x \in \Ho(\C)$ acquires a unital pairing $m \otimes x \rightarrow x$. 
If $m$ is an associative monoid in $\Ho(\C)$ and $x$ is an $m$-module, 
we may take the module action as this pairing. 
Let $\alpha: y \rightarrow m$ be a map from an $\otimes$-invertible object $y$. 
For example if $\C = \SH(S)$, 
then $y$ could be $\Sigma^{p,q}\sspt$ and $\alpha: \Sigma^{p,q}\sspt \rightarrow m$ is a homotopy element. 
We define $\alpha \cdot: x \rightarrow y^{-1}\otimes x$ by applying $y^{\otimes-1}$ to the ``multiplication by $\alpha$ map": 
$$
y \otimes x \stackrel{\alpha}{\rightarrow} m \otimes x \rightarrow x.
$$
The $\alpha$-inversion $x[\alpha^{-1}]$ of $x$ is defined as the colimit of the diagram:
\begin{equation} 
\label{incoh-mult}
x \stackrel{\alpha \cdot}{\rightarrow} y^{-1}\otimes x \stackrel{\id \otimes \alpha \cdot}{\rightarrow} y^{-1} \otimes y^{-1} \otimes x \cdots.
\end{equation}

\subsubsection{} Let us now tackle the homotopy coherent situation.
\begin{definition} 
The $\infty$-category $P_S\M$ of \emph{$S$-periodic objects in $\M$} is the full subcategory of $\M$ spanned by objects $m \in M$ such that for all $\alpha: x \rightarrow \unit$ in $S$, 
the map $\alpha^*: \Hom(\unit, m) \simeq m \rightarrow \Hom(x, m)$ is an equivalence.
\end{definition}

\begin{example} 
Let $n: \unit \rightarrow \unit$ be the multiplication by $n$ map in $\SH$, the stable $\infty$-category of spectra.
A spectrum is $\{n:\unit\rightarrow\unit\}$-periodic if and only if $n$ acts invertibly.
\end{example}

\begin{proposition} 
Suppose the domains of elements in $S$ are $\kappa$-compact. 
Then the fully faithful embedding of $P_S\M$ into $\M$ admits a left adjoint $P_S: \M \rightarrow P_S \M$ which witnesses the latter as an $\kappa$-accessible localization at $\{ \id_m \otimes \alpha \}$, 
where $m \in M$, 
and $\alpha: x \rightarrow \unit \in S$.
\end{proposition}
\begin{proof}
It is clear that limits of periodic objects are also periodic. 
To see that $P_S\M \hookrightarrow \M$ preserves $\kappa$-filtered colimits, 
note that for all $\alpha: x \rightarrow \unit \in S$ we have, by compactness of $x$,:
$$
\Maps(x, \colim_{\alpha} E_{\alpha}) 
\simeq  
\colim_{\alpha}  \Maps(x,E_{\alpha}) 
\simeq  
\colim_{\alpha} E_{\alpha},
$$ 
and thus the left adjoint exists. 
By the adjunction $\Hom(m \otimes x, n) \simeq \Hom(m, \Hom(x, n))$, we see that we are exactly inverting the maps $\id_m \otimes x: m \otimes x \rightarrow m \otimes \unit \simeq m$.
\end{proof}

\subsubsection{}  
\label{par:alg} 
The localization turns out to be compatible with the monoidal structure in the sense of \cite{higheralgebra}*{Proposition 2.2.1.9}. 
Thus $A \in \CAlg(\C)$ implies $P_SA\in\CAlg(\C)$, 
and the map $A \rightarrow P_SA$ is one of algebras.

\subsubsection{} 
Let $S_0$ denote the set of domains of the morphisms in $S$. 
By the machinery explained in Theorem~\ref{thm:robalo-stab1} we can consider the monoidal inversion $\M[S_0^{-1}]$ for which there is an equivalence of symmetric monoidal $\infty$-categories 
\cite{hoyois-cdh}*{Proposition 3.2}: 
$$
P_S\M[S_0^{-1}] \rightleftarrows P_S\M.
$$ 
Thus every $S$-periodic object in $\CAlg(\C)$ uniquely defines an $S$-periodic object in $\CAlg(\C[S_0^{-1}])$.

\subsubsection{} 
An explicit formula for periodization is given in \cite{hoyois-cdh}*{Section 3}. 
We are interested in inverting a single map $\alpha: x \rightarrow \unit$, 
so the formula explained prior to Theorem 3.8 of \cite{hoyois-cdh} simplifies to:
$$
Q_{\alpha}m := \colim m \stackrel{\alpha^*}{\rightarrow} \mapint(x, m) \stackrel{\alpha^*}{\rightarrow} \mapint(x, \mapint(x,  m)) \simeq \mapint(x^{\otimes 2}, m)\cdots.
$$ 
According to \cite{hoyois-cdh}*{Theorem 3.8}, this formula computes exactly the $\alpha$-periodiziation in certain cases. 
For our purposes it suffices to prove a simpler result.

Suppose $\Gamma$ is a symmetric monoidal $1$-groupoid, and $\D$ is a $1$-category.
The $1$-category $\D^{\Gamma}$ of \emph{$\Gamma$-graded objects in $\D$} is the functor category $\Fun(\Gamma, \D)$. 
For $F: \Gamma \rightarrow \D$ we set $F_{\gamma} := F(\gamma)$ and for $\gamma, \gamma' \in \Gamma$ we set $F_{\gamma + \gamma'} := F(\gamma \otimes \gamma')$. 
If $\Gamma$ is the groupoid of integers with a unique invertible morphism and monoidal structure given by addition, 
then $\D^{\Gamma}$ is the category of graded objects in $\D$. 
If $\Gamma := \Ho(\Pic(\SH(S))$, 
where $\Pic(\SH(S))$ denotes $\otimes$-invertible objects in $\SH(S)$, 
then $\Ab^{\Gamma}$ corresponds to homotopy groups graded by $\otimes$-invertible objects.
The point of the above definition is that we often have a functor $\pi: \C \rightarrow \D$ such that the collection of functors $\{ \pi(\mapint(\gamma, -)) \}_{\gamma \in \Gamma}$ 
is conservative for $\Gamma \subset \C$ a $1$-subgroupoid. 
This reduces the checking of certain higher coherences to simpler $1$-categorical coherences.

\begin{proposition} \label{prop:invert-long} 
Let $\C$ be a presentably symmetric monoidal $\infty$-category and let $\alpha: x \rightarrow \unit$ be a morphism. 
Suppose we are given the following data:
\begin{enumerate}
\item $\Gamma \subset \C$ a $1$-subgroupoid of $\C$ containing the morphism $\alpha$ and the permutation isomorphisms $\sigma_n: x^{\otimes n} \rightarrow x^{\otimes n}$ for all $n \geq 1$.
\item A filtered colimit preserving functor: 
$$
\pi: \C \rightarrow \D,
$$ 
so that the functors $\{\pi( \mapint(\gamma, -)) =: \pi_{\gamma}(-)  \}_{\gamma \in \Gamma}$ form a conservative family.
\item For any $m \in \C$, consider the diagram $\pi( \mapint(-, m)) \in \D^{\Gamma}$. 
Then there exists $n\gg 0$ such that for any $m \geq n$ there is a commutative diagram in $\D$:
\begin{equation}
\xymatrix{
\pi( \mapint(-, m))_{\Sigma_m\,x} \ar[r]^{\id} \ar[d]_{\sigma_{n,*}} & \pi( \mapint(-, m))_{\Sigma_m\,x} \\
\pi( \mapint(-, m))_{\Sigma_m\,x}. \ar[ur]_{\id} & }
\end{equation}
\end{enumerate}
Then for $x$ any $n$-symmetric object and $n\geq 2$, we have $Q_{\alpha} \simeq P_{\alpha}$.
\end{proposition}

\begin{proof} 
For $m \in \C$ we claim that $Q_{\alpha}m$ is $\alpha$-periodic, 
i.e., 
$\alpha^*:Q_{\alpha}m \rightarrow \mapint(x, Q_{\alpha}m)$ is an equivalence in $\C$. 
In effect, 
consider the (solid) commutative diagram:
$$
\xymatrix{
m \ar[r] \ar[d] 
& \mapint(x, m) \ar[r] \ar[d] 
& \mapint(x^{\otimes 2}, m) \ar[r] \ar[d] 
& \cdots \ar[r] \ar[d] 
& Q_{\alpha}m \ar[d]_{\alpha^*} \\
\mapint(x, m) \ar[r]  \ar@{-->}[ur]& \mapint(x,\mapint(x, m)) \simeq \mapint(x^{\otimes 2}, m)  \ar[r] \ar@{-->}[ur] 
&  \mapint(x,\mapint(x^{\otimes 2}, m)) \ar[r]  \ar@{-->}[ur] 
& \cdots \ar[r] 
& \mapint(x, Q_{\alpha}m). 
}
$$
We construct a quasi-inverse to $\alpha^*$ by constructing the indicated dotted arrows rendering the diagrams commutative. 
Letting the dotted arrows be identities, we see that while the top triangles commute, the bottom triangles \emph{do not necessarily commute}. 
The problem is that the horizontal arrow is given by $\alpha^*$, 
while the vertical arrow is given by $\mapint(\id, \alpha^*)$.  
Hence, 
after adjunction, 
the two maps differ by a cyclic permutation.
However, 
for $\gamma \in \Gamma$, 
applying $\pi_{\gamma}$ to the above solid diagram yields a diagram in the $1$-category $\D$:
$$
\xymatrix{
\pi_{\gamma}m \ar[r] \ar[d] 
& \pi_{\gamma + x}m \ar[r] \ar[d]
& \pi_{\gamma + x + x}m \ar[r] \ar[d] 
& \cdots  \ar[r] \ar[d]
& \colim_n \pi_{\gamma+ nx}m \ar[d]_{\pi_{\gamma}\alpha^*} \\
\pi_{\gamma+x}m \ar[r]  \ar@{-->}[ur]
& \pi_{\gamma+ x+ x}m \ar[r] \ar@{-->}[ur] 
& \pi_{\gamma+ x+ x+x}m\ar[r]  \ar@{-->}[ur] 
& \cdots \ar[r] & \colim_n \pi_{\gamma +nx}\mapint(x, Q_{\alpha}m). 
}
$$
Here we have identified the last terms with $\pi_{\gamma}(Q_{\alpha} m)$ using the fact that $\pi$ preserves filtered colimits. 
By the conservativity of $\pi_{\gamma}$, we need only prove that $\pi_{\gamma}(\alpha^*)$ is an isomorphism. 
To see this, we examine the diagram:
$$
\xymatrix{
& \pi_{\gamma + kx} m \ar[d]\\
\pi_{\gamma + {k}x} m \ar[r] \ar@{-->}[ur]^{\id} & \pi_{\gamma + (k+1)x}m.
}
$$
If this diagram commutes for $k \gg 0$, then we can produce the desired inverse. The commutativity of the above diagram is thus given by the third hypothesis. 
 
It remains to apply \cite{hoyois-cdh}*{Lemma 3.3} which states that if $Q_{\alpha}m$ is $\alpha$-periodic then it must coincide with $P_{\alpha}m$.
\end{proof}
Note that the third hypothesis is a weaker form of the condition that the map $\alpha: x \rightarrow 1$ is $n$-symmetric since we need only check a commutative diagram of $1$-categories. According to the coherence results of Dugger in \cite{dugger-coherence}*{Proposition 4.20-21} for invertible objects, the conclusion of Proposition~\ref{prop:invert-long} holds for $\C = \SH(S)$ and $\Gamma = \Ho(\Pic(\SH(S)).$

\subsubsection{} 
We shall apply Proposition \ref{prop:invert-long} to $\C = \SH(S)$ (resp.~$\Mod_{\E}$ for $\E \in \CAlg(\SH(S)$ or a presheaf of $\mathcal{E}_{\infty}$-ring spectra), 
and to $\Gamma$ the $1$-subgroupoid of $\SH(S)$ (resp.~$\Mod_{\E}$) spanned by $\Sigma^{2n,n} \Sigma^{\infty}_{T}X_+$, $n \in \ZZ$ (resp.~$\E \wedge \Sigma^{2n,n} \Sigma^{\infty}_{T}X_+$, $n \in \ZZ$) 
and all invertible $1$-morphisms. 
We let $\pi:= [S^0, -]$ so that $\{\pi(\mapint(\Sigma^{2n,n} \Sigma^{\infty}_{T}X_+, -))\}_{\Gamma}$ (resp.~$\{\pi(\mapint(\E \wedge \Sigma^{2n,n} \Sigma^{\infty}_{T}X_+, -))\}_{\Gamma}$) forms a conservative family.

\subsubsection{Examples}
We review some examples relevant for the main body of the paper.

\begin{example} ($K$-theory) 
We work with $\C = \Mod_K$. 
As explained in \cite{hoyois-cdh}*{Example 3.4} there is a map:
$$
\gamma: \Sigma^{\infty}( (\PP^{1} \setminus 0) \coprod_{\GG_m} \AA^{1}) \rightarrow \Sigma^{\infty} \Sigma(\GG_m,1) \rightarrow K,
$$ 
in presheaves of spectra on $\Sm_S$, 
which defines a map $\gamma: K \wedge  \Sigma^{\infty} (\PP^{1} \setminus 0 \coprod_{\GG_m}) \AA^{1} \rightarrow K$ in $\Mod_K$. 
The nonconnective Bass-Thomason-Trobaugh $K$-theory is the presheaf of $\mathcal{E}_{\infty}$-ring spectra defined by the periodization $K^B:=Q_{\gamma}K$ in $\Mod_K$. 
Taking $\LL_{\mot}$ of $K$ (actually $\LL_{\AA^{1}}$ suffices) $\gamma$-periodization coincides with $\beta$-periodization, 
where $\beta$ is the second map defining $\gamma$. 
Using Proposition~\ref{prop:invert-long}, 
we obtain the motivic $\mathcal{E}_{\infty}$-ring spectrum representing algebraic $K$-theory $\KGL$ such that $\Omega^{\infty}\KGL \simeq \LL_{\mot}K^B$. 
This is the main content of \cite{cisinski}. See also \cite{gepner-snaith} and \cite{ostvaer-spitzweck} for another perspective of how Bott inversion on the infinite projective space gives $\KGL$. In the language of this paper, this is also discussed in \cite{hoyois-cdh}*{Section 5}, in the equivariant setting.
\end{example}

\begin{example} (\'{E}tale Cohomology) 
\label{ex:et}
We work with $\C = \DM^{\eff}_{\et}(S; \ZZ/\ell)$. 
If $\ell$ is prime to all the residue characteristics of $S$, 
the presheaf $\mu_{\ell}$ is a homotopy invariant \'{e}tale sheaf with transfers;
thus $\mu_{\ell}\in\C$. 
The choice of an $\ell$-th root of unity is a map $\ZZ/\ell^{\et} \rightarrow \mu_{\ell}$; 
since $\mu_{\ell}$ is invertible in $\C$, 
this amounts to a map $\tau: \mu_{\ell}^{\otimes -1} \rightarrow \ZZ/\ell^{\et}$. 
Then $\H\mu_{\ell}:=Q_{\tau}\ZZ/\ell^{\et}$ represents \'{e}tale cohomology with $\mu_{\ell}$-coefficients. 
By Proposition~\ref{prop:invert-long}, $\H\mu_{\ell}$ is an object of $\DM_{\et}(S; \ZZ/\ell)$. 
Recall that $\DM_{\et}(S; \ZZ/\ell) \simeq \C$ and $\ZZ/\ell^{\et}(1) \simeq \mu_{\ell}$.
Since $\tau: \ZZ/\ell^{\et} \rightarrow \Hom(\mu_{\ell}, \ZZ/\ell^{\et})$ is an equivalence, 
the unit object in $\C$ is $\tau$-periodic and $\H\mu_{\ell} \simeq \ZZ/\ell^{\et}$ in $\DM_{\et}(S;\ZZ/\ell)$. 
The spectrum $\H_{\et}\mu_{\ell}:=u_{\tr}\H\mu_{\ell}$ in $\SH_{\et}(S)$ represents \'{e}tale cohomology with $\mu_{\ell}$-coefficients in the sense that 
$[\Sigma^{p,q} \Sigma^{\infty}_{T}X_+, \H_{\et}\mu_{\ell}] \cong H_{\et}^{-p}(X, \mu_{\ell}^{\otimes -q})$.
\end{example}

\begin{example} (Inverting elements in $\SH(S)$) 
Suppose that $\E \in \CAlg(\SH(S))$ so that $\Mod_{\E}$ is a symmetric monoidal $\infty$-category (inverting elements for a noncommutative algebra is trickier). 
Given $\F\rightarrow \E$ in $\SH(S)$, we get $\alpha: \E \wedge \F\rightarrow \E$ in $\Mod_{\E}$ and set $\E[\alpha^{-1}] := Q_{\alpha}E$. 
More generally, 
given $\alpha: \F\rightarrow \E$, 
we may $\alpha$-periodize any $\E$-module $\M$ by setting $\M[\alpha^{-1}]:= P_{\alpha}\M$.
\end{example}

\subsection{Localization}
A theory for Bousfield localization of motivic spectra, in the language of stable model categories was carried out in \cite{ro2}. 
We quickly adopt their results to $\infty$-categories; the results stated here are surely well-known to experts.
Throughout $(\C,\otimes,\sspt)$ is a presentably symmetric monoidal stable $\infty$-category.

\begin{definition} 
For $\M, \E, \F \in \C$ we define:
\begin{enumerate}
\item $f: \E \rightarrow \F$ in $\C$ to be an \emph{$\M$-equivalence} if $f \otimes id_{\M}$ is an equivalence,
\item $\E$ to be \emph{$\M$-acyclic} if $\E \otimes \M \simeq 0$,
\item $\F$ to be \emph{$\M$-local} if for any $\M$-acyclic object $\E$, the mapping space $\Maps(\E,\F)$ is contractible.
\end{enumerate}
\end{definition}

In the case that $\C = \SH$, 
the stable $\infty$-category of spectra, 
this is just a homotopy-theoretic refinement of the definition of \cite{bousfield}. 
In the case of $\C = \SH(S)$, 
which is our primary case of interest, 
this is just an $\infty$-categorical reformulation of the definitions of \cite{ro2}*{Appendix A}.

From now on we assume in addition that $\C$ is presentably symmetric monoidal. 
In many cases $\M$ will at least be an $\mathcal{E}_1$-algebra object in $\C$. 
\begin{proposition} 
\label{prop:algebras} 
Any $\M$-module is $\M$-local.
\end{proposition}
\begin{proof} 
Suppose that $\F$ is an $\M$-acyclic object. 
Then, for any $\M$-module $\E$, we have:
\begin{eqnarray*}
\Maps_{\C}(\F, \E) 
& \simeq & \Maps_{\Mod_{\M}}(\F \wedge \M,\E) \\
& \simeq & \Maps_{\Mod_{\M}}(0,\E) \\
& \simeq & 0.
\end{eqnarray*}
\end{proof}

\subsubsection{} 
Let $\C_{\M} \subset \C$ denote the full subcategory of $\C$ spanned by the $\M$-local objects. 
We would like to construct a \emph{Bousfield localization functor}: 
$$
L_{\M}: 
\C \rightarrow \C_{\M},
$$ 
as a left adjoint to the inclusion $\C_{\M} \subset \C$. In this situation, 
we also call $L_{\M}$ an \emph{$\M$-completion} functor. 
Such a functor is produced in \cite{htt}*{Proposition 5.5.4.15}. 
Indeed, 
the collection of all morphisms $f$ such that $f \otimes \id_{\M}$ is an equivalence forms a strongly saturated class by Proposition~\cite{htt}*{Proposition 5.5.4.16}, 
and thus the inclusion $\C_{\M}$ has a left adjoint.
The following lemma is a slight refinements of this existence result.
\begin{lemma} 
\label{lem:construct} 
Let $\C$ be a presentable stable $\infty$-category. 
Suppose $\ell:\C_0^{\perp} \subset \C$ is a full subcategory, and let $r:\C_0 \subset \C$ be the full subcategory spanned by objects $X$ in $\C$ for which 
$\Maps(Y, X) \simeq 0$ for all $Y \in \C_0^{\perp}$. 
Then the following are equivalent:
\begin{enumerate}
\item the inclusion $r: \C_0 \subset \C$ admits a left adjoint $L: \C \rightarrow \C_0$ such that $r \circ L: \C \rightarrow \C$ is exact and $\kappa$-accessible;
\item the full subcategory $r:\C_0 \subset \C$ is presentable, stable, and closed under $\kappa$-filtered colimits under sufficiently large $\kappa$.
\end{enumerate}
Furthermore, the following equivalent statements:
\begin{enumerate}
\item [(3)]the inclusion $\ell: \C_0^{\perp} \subset \C$ admits a right adjoint $R: \C \rightarrow \C_0^{\perp}$ such that $\ell \circ R: \C \rightarrow \C$ is exact and accessible;
\item [(4)]the full subcategory $\ell: \C_0^{\perp} \subset \C$ is presentable and stable,
\end{enumerate}
imply the above two statements.
\end{lemma}
\begin{proof} 
Assuming $(1)$ we get $\C_0 \simeq L\C$. 
By \cite{higheralgebra}*{Lemma 1.1.3.3}, we need only show that $\C_0$ is stable under cofibers and translations. 
Stability under cofibers follows since $L$ is a left adjoint and hence preserves cofibers in $\C$. 
To see stability under translations, it suffices to prove that $X[-1] \in \C_0$ for any $X \in \C_0$. 
We have a cofiber sequence in $\C$, $X[-1] \rightarrow 0 \rightarrow X$. 
Then $rL(X[-1]) \rightarrow rL(0) \simeq 0 \rightarrow rL(X) \simeq X$ is still a cofiber sequence in $\C$ by exactness of $r \circ L$, 
and is thus a fiber sequence by the stability of $\C$. 
Now since the inclusion $r$ is fully faithful and a right adjoint, 
$r$ creates limits and thus $L(X[-1]) \rightarrow 0 \rightarrow X$ is a fiber sequence in $\C_0$, 
and therefore the shift $X[-1]$ is indeed in $\C_0$.
The presentability of $\C_0$ follows from \cite{htt}*{Proposition 5.5.4.15}, by letting $S$ be all morphisms of the form $rL(f)$. 
In particular $\C_0$ is $\kappa$-accessible for some large enough cardinal $\kappa$, and thus admits $\kappa$-filtered colimits.

Assume that $(2)$ holds. Since $\C$ and $\C_0$ are presentable, 
$r$ preserves small limits and $\kappa$-filtered colimits for some cardinal $\kappa$. 
We deduce from the adjoint functor theorem \cite{htt}*{Corollary 5.5.2.9} that $r$ admits a left adjoint $L: \C \rightarrow \C_0$. 
By \cite{htt}*{Proposition 5.4.7.7}, we see that both $L$ and $r$ are accessible. 
By \cite{higheralgebra}*{Proposition 1.1.4.1} both $r$ and $L$ are exact, and thus their composite is accessible and exact.

The equivalence between  $(3)$ and $(4)$ follows analogously by replacing left with right adjoints and using the adjoint functor theorem \cite{htt}*{Corollary 5.5.2.9} with right adjoints.

Next, 
assuming $(3)$ we show $(1)$. 
The cofiber of the counit transformation $\ell R \rightarrow \id$ yields an endofunctor $L': \C \rightarrow \C$. 
Since $\ell R$ and $\id$ are exact and accessible, $L'$ is exact and accessible. 
We claim that its essential image lies in $\C_0$ and that the resulting functor $L: \C \rightarrow \C_0$ is the desired left adjoint. 
For any $X \in \C$, we show that $L'(X) \in \C_0$. 
If $Y \in \C_0^{\perp}$ there is a cofiber sequence of spectra: 
$$
\Maps(Y, \ell R(X)) \rightarrow \Maps(Y, X) \rightarrow \Maps(Y, L'(X)).
$$
By the assumption on $Y$, 
$\Maps(Y, \ell R(X)) \simeq \Maps(Y, X)$ and therefore the final term is contractible and indeed the functor lies in the essential image. 
To check that it is left adjoint, 
it suffices to show that for any $X' \in \C_0$, 
the map $L'(X) \rightarrow X$ induces an equivalence $\Maps(X, X') \rightarrow \Maps(L'(X), X')$. 
To this end we consider the defining cofiber sequence $\ell R(X) \rightarrow X \rightarrow L'(X)$.
It induces a cofiber sequence of spectra:
$$
\Maps(L'(X), X') \rightarrow \Maps(X, X') \rightarrow \Maps(\ell R(X), X'),
$$ 
where the final term is contractible since $X' \in \C_0$.
\end{proof}

\begin{remark} 
The following standard warning applies: 
even though $L$ preserves colimits, 
the composite $r \circ L$ may not. An example is the process of \'etale sheafification: if $r \circ L$ preserves filtered colimits then representables would remain compact but, as discussed earlier, this is not the case without finiteness of cohomological dimension. In the case of localization at an $\mathcal{E}_1$-ring spectrum $\E$, $r \circ L$ preserves colimits if and only if it is computed at a spectrum $\F$ by smashing $\F$ with $r \circ L(\sspt)$.
\end{remark}

%
%

\subsubsection{Example: $\ell$-completion} 
Let $\M = \unit/\ell$ for $1 < \ell \in\NN$. 
In this case $\M$ is \emph{not} an $\mathcal{E}_1$-object in $\SH(S)$. 
The Bousfield localization functor: 
$$
L_{\unit/\ell}: \SH(S) \rightarrow \SH(S)_{\unit/\ell},
$$ 
admits an explicit formula. 
Define the \emph{$\ell$-adic completion} functor as: 
$$
(-)\compl: \SH(S) \rightarrow \SH(S);
\,\,\,
\E\compl:= \lim (\E \leftarrow \E/\ell \leftarrow \E/\ell^2 \leftarrow\cdots),
$$
where the transition maps $\E/\ell^{k+1} \rightarrow \E/\ell^k$ are induced by the natural equivalence at the bottom row of the following diagram of cofiber sequences:
\begin{equation}
\xymatrix{
E \ar[r]^{\times \ell} \ar[d]_{\id} & E \ar[d]^{\times \ell^k} \ar[r] & \E/\ell \ar[r] \ar[d] & \Sigma^{1,0} E \ar[d]^-{\Sigma^{1,0} \id} \\
E  \ar[r]^{\times \ell^{k+1}} \ar[d] & E \ar[r] \ar[d] & \E/\ell^{k+1} \ar[r] \ar[d] & \Sigma^{1,0} E \ar[d]\\
0 \ar[r] & \E/\ell^k \ar[r]^-{\simeq} & C \ar[r] & 0.
}
\end{equation}
\begin{proposition} 
The Bousfield localization functor $L_{\unit/\ell}: \SH(S) \rightarrow \SH(S)_{\unit/\ell}$ coincides with the $\ell$-adic completion functor $(-)\compl: \SH(S) \rightarrow \SH(S)$.
\end{proposition}
\begin{proof} 
The proof in \cite{ro2}*{Example 3.5} applies in the $\infty$-categorical setting.
\end{proof}
Thus the $\infty$-category $\SH(S)_{\unit/\ell}$ is the \emph{$\ell$-adic completion} of $\SH(S)$;
we write $\SH(S)\compl:= \SH(S)_{\unit/\ell}$.

\subsubsection{Example: localization}
Applying Lemma~\ref{lem:construct} to $\E \in \SH(S)$ we can construct an adjunction: 
$$
L_{\E}: \SH(S) \rightleftarrows \SH(S)_{\E}: r_{\E},
$$ 
where $\SH(S)_{\E}$ is the full subcategory of $\E$-local objects.
\begin{proposition} 
\label{prop:exists} 
If $\M \in \SH(S)$, the inclusion $i: \SH(S)_{\M} \subset \SH(S)$ admits a left adjoint $L_{\M}: \SH(S) \rightarrow \SH(S)_{\M}$.
\end{proposition}
\begin{proof} 
Let $\SH(S)_{\M}^{\perp}$ denote the $\M$-acyclic objects. 
To verify condition (4) of Lemma~\ref{lem:construct}, 
the nontrivial thing to check is that $\SH(S)_{\M}^{\perp}$ is $\kappa$-accessible for some large cardinal $\kappa$. 
In other words, it is of the form $\Ind_{\kappa}(\C_0)$ for a small $\infty$-category $\C_0$. 
For $\kappa$ we choose the minimal cardinal larger than $\bigcup_{p,q, X \in \Sm_S} |\M_{p,q}(X)|$, 
where $|-|$ denotes the cardinality of a set. 
Let $\C_0$ be the full subcategory of $\SH(S)_{\M}$ spanned by objects $A$ such that $|\Hom(\Sigma^{p,q}X_{+},A)| < \kappa$. 
The objects of $\C_0$ form a set and is thus a small $\infty$-category. 
The verification that this choice of $\kappa$ works is in \cite{ro2}*{Appendix A}.
\end{proof}

\subsubsection{Fracture squares}
We record Bousfield localization fracture squares in the context of motivic homotopy theory \cite{oo}*{Theorem A.1}.
\begin{theorem} 
Let $\E,\F,\G \in \SH(S)$, and suppose that $\E \wedge L_{\F}\G \simeq \E \wedge L_{\F}L_{\E}\G \simeq 0$.
Then there is a Cartesian square in $\SH(S)$:
$$
\xymatrix{
L_{\E \vee\F}\G \ar[r]^{\eta_{\E}} \ar[d]_{\eta_{\F}} & L_{\E}\G \ar[d]\\
L_{\F}\G \ar[r]^-{L_{\F} \eta_{\E}} & L_{\F} L_{\E}\G.
}
$$
\end{theorem}

A special case we will use repeatedly is the so-called arithmetic fracture square.
\begin{theorem} 
\label{thm:arithmetic} 
For any $\E \in \SH(S)$ there is a Cartesian square in $\SH(S)$:
$$
\xymatrix{
\E \ar[r]^-{\eta_{\E}} \ar[d] & \prod_{\ell} \E \compl \ar[d]\\
\E_{\QQ} \ar[r]^-{L_{\F} \eta_{\E}} & (\prod_{\ell} \E\compl)_{\QQ}.
}
$$
\end{theorem}

\subsubsection{} 
Finally, we record a lemma about localization and the six functors formalism. 
\begin{lemma} 
\label{lem:bousfieldvspull} 
Suppose $\M$ is a Cartesian section of $\SH(-)$ and $f: S \rightarrow T$ is a map of schemes. 
Then for $\E \in \SH(T)$ there is a canonical equivalence: 
$$
f^*L_{\M}\E \overset{\simeq}{\to} L_{\M}f^*\E.
$$ 
\end{lemma}
\begin{proof} 
First we show that $f_*$ preserves $\M$-local objects. 
If $\M \in \SH(S)$ is $\M$-local and $\F \in \SH(T)$ is $\M$-acyclic, 
then $f^*\F$ is $\M$-acyclic since $\M \wedge f^*\F \simeq f^*(\M \wedge \F) \simeq 0$ (here we use that $\M$ is Cartesian and $f^*$ is monoidal). 
As a result, 
$[\F,f_*\E] \simeq [f^*\F, \E] \simeq 0$ and thus the following diagram of right adjoints commute:
$$
\xymatrix{
\SH(S)_{\M} \ar[d]_{f_*} \ar[r] & \SH(S) \ar[d]^{f_*} \\
\SH(T)_{\M} \ar[r] & \SH(T).
}
$$
Thus the diagram with the corresponding left adjoints commutes too, 
i.e., 
$f^*L_{\M} \simeq L_{\M}f^*$. 
\end{proof}

\section{Appendix B: hyperdescent}\label{sect:hyp-remove}

In this paper, we have meant for $\SH_{\et}(S)$ to be the version of stable \'etale motivic homotopy theory built from hypercovers. There is a version, which we denote by $\widetilde{\SH}_{\et}(S)$ built from \v{C}ech covers: this is obtained by inverting 
\[
 \{ \Sigma^{p,q}\Sigma^{\infty}_{T,+}\LL_{\mot,\et}(U \hookrightarrow X):X \in \Sm_S, U \hookrightarrow X\text{ is an $\et$-sieve of $X$}, p,q \in \ZZ \}.
\]
We also have an adjunction:
\[
\widetilde{\pi}^*:\SH(S) \rightleftarrows \widetilde{\SH}_{\et}(S):\widetilde{\pi}_*,
\]
and we denote $\widetilde{\pi}^*\widetilde{\pi}_*\E := \widetilde{\E}^{\et}$. We have a comparison map $\widetilde{\E}^{\et} \rightarrow \E^{\et}$.

\begin{lemma} \label{lem:hyper} Let $\ell$ be a prime and $S$ be a Noetherian $\ZZ[\frac{1}{\ell}]$-scheme of finite dimension and assume that for all $x \in S$, $\cd_{\ell}(k(x)) < \infty$. Then for all $\nu \geq 1$, the map $\widetilde{\MGL_S}^{\et}/{\ell^{\nu}} \rightarrow \MGL_S^{\et}/{\ell^{\nu}}$ is an equivalence.
\end{lemma}

\begin{proof} The lemma asserts that $\widetilde{\MGL_S}^{\et}/{\ell^{\nu}}$ is, in fact, hypercomplete. By the criterion of \cite{clausen-mathew}*{Theorem 4.37}, it then suffices to prove the claim for fields. Unpacking the proof for the field case \S\ref{sec:systemofss}, our arguments assert that Bott inverted $\MGL_S/{\ell^{\nu}}$ is filtered by the \'etale version of the slice tower and the associated graded are spectra representing \'etale cohomology each of which has hyperdescent. Since hypercomplete sheaves are stable under limits (it is a Bousfield localization of the category of sheaves), an induction along the Postnikov tower tells us that Bott inverted $\MGL_S/{\ell^{\nu}}$ is, in fact, hypercomplete and is equivalent to both $\widetilde{\MGL_S}^{\et}/{\ell^{\nu}}$ and $\MGL_S^{\et}/{\ell^{\nu}}$.
\end{proof}

Using Lemma~\ref{lem:hyper} we can replace our results expressing various Bott-inverted motivic spectra as the localization at \'etale hypercovers by the analogous statement for \'etale covers instead.

\section{Appendix C: $\E$-slice completeness and $\E$-locality} \label{appendix:e-locals}
Here we record the following result related to \cite{koras-russell}*{Lemma 4.1} which we will not need directly but inspired the proof of Theorem~\ref{thm:integralresult}, 
and is generally useful.
\begin{proposition} 
\label{prop:eff-mz-loc} 
Let $\E$ be a motivic $\mathcal{E}_{\infty}$-ring spectrum. 
Suppose $\M \in \Mod_{\E}$ is $\E$-slice complete and eventually effective, 
i.e., 
there exists $q \in \NN$ such that $\Sigma_T^d\E \in \Mod_{\E}^{\eff}(S)$. 
Then $\M$ is $s^{\E}_0\E$-local.
\end{proposition}
\begin{proof} 
By the discussions in~\S\ref{slice-complete} $\M$ is $\E$-slice complete if it is equivalent to $\lim f^q_{\E}\M$. 
Since the $\infty$-category of $s^{\E}_0\E$-local objects is colocalizing, 
thus closed under limits, 
it suffices to prove $f^q_{\E}\M$ is $s^{\E}_0\E$-local for all $q \in \ZZ$. 
Without loss of generality we may assume that $\M$ is $\E$-effective and so we may assume $q \geq -1$. 
In this case, 
$f_0^{\E}\M \simeq \M$ and thus the cofiber sequence $f^{\E}_1\M \rightarrow f^{\E}_0\M \simeq \E \rightarrow f_{\E}^{-1}\M$ implies that  $f_{\E}^{-1}\M \simeq s^{\E}_0\M$. 
Since $s^{\E}_0\M$ is a module over the $\mathcal{E}_{\infty}$-ring spectrum $s^{\E}_0\E$ \cite{operadsslices}*{Section 6 (v)}, we have that $f_{\E}^{-1}\M$ is $s^{\E}_0\E$-local. 
Assume now that $f_{\E}^{q'}\M$ is $s^{\E}_0\E$-local for all $q' < q$.
We claim $f_{\E}^q\M$ is $s^{\E}_0\E$-local. 
To see this, 
we use the cofiber sequence $s^{\E}_q\M \rightarrow f_{\E}^{q}\M \rightarrow f_{\E}^{q-1}\M$. 
Since $s^{\E}_q\M$ is an $s^{\E}_0\M$-module, 
the claim follows from the inductive hypothesis.
\end{proof}
\begin{remark} 
Suppose $S = \Spec\,k$ is a perfect field of finite cohomological dimension and exponential characteristic $c$.  
Then according to \cite{levine-converge}*{Theorem 4} objects in $\SH(k)^{\proj}[\frac{1}{c}]$ are slice complete and thus $s_0(\sspt[\frac{1}{c}]) \simeq \M\ZZ[\frac{1}{c}]$-local.
\end{remark}

\section{Appendix D: Transfers for $\MGL$}
\label{mgl-trsfs}
In this appendix we prove the following statement about transfers for $\MGL$-modules:

\begin{proposition} 
\label{prop:trsf-mgl} 
Suppose that $A \subset B$ is a finite Galois extension of commutative rings of degree $d$ with Galois group $G:=\Aut_A(B)$, i.e., the map $A \rightarrow B$ is finite \'etale and $A \cong B^{G}$. Denote by $p: \Spec\,B \rightarrow \Spec\,A$ the map on prime spectra. Then there exists a transfer map:
\[
p_*:\E^{*,*}(\Spec\,B) \rightarrow \E^{*,*}(\Spec\,A)
\]
such that
\begin{enumerate}
\item The map $p_*$ is $G$-equivariant.
\item The composite $p_*p^*$ is multiplication by $|G|=d$.
\item The composite $p^*p_* = \sum_{g \in G} g_*$.
\end{enumerate}
\end{proposition}

\begin{proof} According to~\cite{deglise-bivariant}*{Definition 3.3.2} we have the \emph{Gysin morphism} $f_*:\E^{*,*}(X) \rightarrow \E^{*,*}(Y)$ for any proper morphism $f: X \rightarrow Y$. In the situation of finite \'etale extensions, this map preserves degrees and this is our transfer map of interest. This map is constructed via cupping with the fundamental class, see \cite{deglise-bivariant}*{Section 2} for a construction which is functorial along gci morphisms, whence the map $f_*$ is $G$-equivariant; in fact in the situation of \'etale extensions the construction of this fundamental class is easy and clearly independent of any factorization of morphisms \cite{deglise-bivariant}*{Example 2.1.2}. This proves property (1).

Property (2) is a consequence of Proposition~\cite{deglise-bivariant}*{Proposition 3.3.9}; indeed since $A \rightarrow B$ is finite \'etale we can let $L_0/S$ in \emph{loc. cit.} to be trivial and hence $p^*p_* = |G| = d$. To prove property (3), we have the following Cartesian diagram
\begin{equation} 
\xymatrix{
\underset{g\in G}\coprod \Spec\,B \ar[d]_{\pi_2} \ar[rr]^{\pi_1} & & \Spec\,B \ar[d]_{p}\\
\Spec\,B \ar[rr]^{p} & & \Spec\,A
}
\end{equation}
of schemes. By base change we have an equality $p^*p_* = \pi_{2*}\pi_1^*$, whence we compute the composite
\[
\E^{*,*}(\Spec\,B) \stackrel{\pi_1^*}{\rightarrow} \E^{*,*}(\underset{g\in G}{\coprod} \Spec\,B) \cong \underset{g\in G}{\bigoplus} \E^{*,*}(\Spec\,B) \stackrel{\pi_{2*}}{\rightarrow} \E^{*,*}(\Spec\,B).
\]
Now, we note that $\pi_{2*}$ identifies with the fold map, while $\pi_1^*(x) = (x, (g_1)_*x, \cdots, (g_{i*})x, \cdots)$ for $g_i \in G$, whence we have the desired formula.
\end{proof}

\section{Appendix E: Transfinite constructions}
\label{sect:transfinite}
Let $\tau$ be a topology on $\Sm_S$ where $S$ is a qcqs scheme. 
We review two transfinite constructions used in this paper.

\subsubsection{Motivic localization}
Recall that if $\tau = \Nis$, 
then the motivic localization functor $\LL_{\Nis,\mot}: \P(\Sm_S) \rightarrow \H(S)$ is obtained by a filtered colimit:
\begin{equation} 
\label{eqn:model}
\LL_{\Nis,\mot} \simeq \colim_{n \rightarrow \infty} (\LL_{\Nis} \circ \Sing^{\AA^1})^{\circ n} \simeq (\LL_{\Nis}\Sing^{\AA^1})^{\circ \NN}.
\end{equation}
Here $\Sing^{\AA^1}$ is the construction introduced in \cite{morel-voevodsky} for an arbitrary site with an interval; 
this formula in our situation is reviewed in \cite{bootcamp}*{Definition 4.23}. 
The following is simply an $\infty$-categorical formulation of \cite{morel-voevodsky}*{Lemma 3.21}.

\begin{proposition} 
\label{prop:transfinite-motivic} 
Let $\tau$ be a topology on $\Sm_S$ and $\kappa$ be regular cardinal such that the endofunctor: 
$$
\P(\Sm_S) \stackrel{\LL_{\tau}}{\rightarrow} \P_{\tau}(\Sm_S) \hookrightarrow \P(\Sm_S),
$$ 
preserves $\kappa$-filtered colimits. 
Then we have an equivalence of endofunctors on $\P(\Sm_S)$:
\begin{equation} 
\label{eqn:model-tau}
\LL_{\tau,\mot} \simeq (\LL_{\tau}\Sing^{\AA^1})^{\circ \kappa}.
\end{equation}
\end{proposition}
\begin{proof} 
We sketch a modification of the proof in \cite{bootcamp}*{Theorem 4.27}.  
Let $\Phi := \LL_{\tau}\Sing^{\AA^1}$.  
We need to check that $(1)$ $\Phi^{\circ \kappa}X$ is $\tau$-local for any $X \in \P(\Sm_S)$, 
$(2)$ $\Phi^{\circ \kappa}X$ is $\AA^1$-invariant, 
and $(3)$ $\Phi^{\circ \kappa}X$ preserves $\LL_{\tau}$ and $\LL_{\AA^1}$-equivalences. 
The proofs for $(2)$ and $(3)$ are the same as for the Nisnevich topology in \cite{bootcamp}*{Theorem 4.27} with $\NN$ replaced by the cardinal $\kappa$. 
The only difference is $(1)$: 
by definition the presheaf $\LL_{\tau}\Sing^{\AA^1}X$ is $\tau$-local. 
The presheaf $\Phi^{\circ \kappa}X$ is a $\kappa$-filtered colimit of $\tau$-local presheaves taken in $\P(\Sm_S)$. 
Hence we are done if the inclusion $\P_{\tau}(\Sm_S) \hookrightarrow \P(\Sm_S)$ preserves $\kappa$-filtered colimits. 
This is true by hypothesis. 
%
\end{proof}

\begin{corollary} 
\label{eqn:kappa-compactness} 
For $X \in \Sm_S$ the objects $X \in \H_{\tau}(S), \Sigma^{\infty}_{S^1}X_+ \in \SH_{\tau}^{S^1}(S),$ and $\Sigma^{\infty}_{T}X_+ \in \SH_{\tau}(S)$ are $\kappa$-compact.
\end{corollary}
\begin{proof} 
The case of $X$ follows from \eqref{eqn:model-tau} since $\LL_{\tau, \mot}$ is a transfinite composite of endofunctors in $\P(\Sm_S)$ that preserve $\kappa$-compact objects 
($\Sing^{\AA^1}$ preserves compact objects since the category of $\AA^1$-invariant presheaves are closed under filtered colimits). 
To check that $\Sigma^{\infty}_{S^1}X_+$ (resp. $\Sigma^{\infty}_{T}X_+$) is $\kappa$-compact, 
it suffices to check that $\Omega^{\infty}_{S^1}$ (resp. $\Omega^{\infty}_{T}$) is $\kappa$-compact. 
This follows since $\Maps_{\H_{\tau}(S)_{\bullet}}(S^1, -)$ (resp. $\Maps_{\H_{\tau}(S)_{\bullet}}(T, -)$) is $\kappa$-compact because $S^1$ (resp. $T$) is the pushout of $\kappa$-compact objects, 
and hence $\kappa$-compact.
\end{proof}

\subsubsection{Spectrification} 
In the notation of~\S\ref{sect:gm-pre}, there is an adjunction: 
$$
Q: \SH^{\pre}_{\tau}(S) \rightleftarrows \SH_{\tau}(S):u.
$$ 
Intuitively, the functor $Q$ should be computed as:
\begin{equation} 
\label{eqn:spectrify-1}
Q(\E)_i \simeq \colim_n E_i 
\stackrel{\epsilon}{\rightarrow} 
\Omega_{\GG_m}\E_{i+1} \stackrel{\Omega\epsilon}{\rightarrow} \Omega^2_{\GG_m}\E_{i+2} \cdots =  \colim_{n} \Omega^{n}_{\GG_m}\E_{i+n}.
\end{equation} 
This is only true whenever the endofunctor $\Omega_{\GG_m}:\SH^{S^1}_{\tau}(S) \rightarrow \SH^{S^1}_{\tau}(S)$ preserves filtered colimits. 
In general, 
we have to perform a transfinite construction. 
We define an endofunctor on $\SH^{\pre}_{\tau}(S)$ as:
\begin{equation} \label{e:the-functor-s}
s: \SH^{\pre}_{\tau}(S) \rightarrow \SH^{\pre}_{\tau}(S), 
\, 
s(\E)_i = \Maps_{\SH^{S^1}_{\tau}(S)}(\GG_m, \E_{i+1}).
\end{equation}

\begin{proposition} 
\label{prop:spectrification} 
Let $\tau$ be a topology on $\Sm_S$ and $\kappa$ be a regular cardinal such that the functor $\LL_{\tau}: \P(\Sm_S) \rightarrow \P_{\tau}(\Sm_S)$ preserves $\kappa$-filtered colimits. 
Then  $Q: \SH^{\pre}_{\tau}(S) \rightleftarrows \SH_{\tau}(S)$ is computed on $\E \in \SH^{\pre}_{\tau}(S)$ as:
\begin{equation} \label{eqn:model-tau-spectrify}
Q(\E)_i \simeq s^{\circ \kappa}(\E)_i.
\end{equation}
\end{proposition}
\begin{proof}  
The formula in~\eqref{eqn:model-tau-spectrify} is given in \cite{hoyois-cdh}*{Page 7} using the fact that $\GG_m$ is $\kappa$-compact in $\SH^{S^1}_{\tau}(S)$; 
it is the cofiber of the map of $\kappa$-compact objects, $S^0 \rightarrow \Sigma^{\infty}_{S^1}\AA^1 \setminus 0_+$.
\end{proof}

\begin{remark} 
If a site has $< \tau$ arrows, then the $\LL_{\tau}$ preserves $\tau$-filtered colimits as explained in \cite{borceux}*{Proposition 3.4.16}. So, for example, when $\tau = \et$, setting $\kappa = \aleph_1$ suffices. 
\end{remark}

\section{Appendix F: Complements on the rational \'etale motivic sphere spectrum} \label{et-q} In this section, we piece together known facts about the rational motivic sphere spectrum and describe $\sspt^{\et}_{S,\QQ}$, its rational counterpart. 

\subsubsection{} Recall that $\D_{\AA^1}(S;\QQ)$ denote the \emph{$\AA^1$-derived category} (see \cite{morel-book}*{Section 5.2} for a reference). As an $\infty$-category $\D_{\AA^1}(S;\QQ)$ is constructed in parallel as in $\SH(S)_{\QQ}.$ First we consider $\D^{S^1}_{\AA^1}(S;\QQ)$, which is the full subcategory of presheaves of complexes of rational vector spaces on $\Sm_S,$ $\P_{\D(\QQ)}(\Sm_S),$ spanned by objects which are $\AA^1$-invariant and satisfy Nisnevich descent. The $\infty$-category $\D^{S^1}_{\AA^1}(S;\QQ)$ is presentably symmetric monoidal and stable. Recall that there is an equivalence
\[
\Spt_{\QQ} \simeq \D(\QQ)
\]
between the $\infty$-category of rational spectra and the $\infty$-category of category of (unbounded) complexes of $\QQ$-vector spaces; this is a result of the basic fact that the rationalized sphere spectrum, $S^0_{\QQ}$, is equivalent to the rationalized Eilenberg-MacLane spectrum $\H\QQ$ and $\D(\QQ)$ is the $\infty$-category of $\H\QQ$-modules; see \cite{shipley} for details. Therefore we have an equivalence of $\infty$-categories:
\begin{equation} \label{s1-q}
\SH^{S^1}(S)_{\QQ} \simeq \D^{S^1}_{\AA^1}(S)_{\QQ}.
\end{equation}
In the same way as discussed in~\S\ref{sect:SHtau}, we can invert the object $\QQ(\PP^1, \infty)$ of $\D^{S^1}_{\AA^1}(S)_{\QQ}$ to obtain the $\infty$-category $\D_{\AA^1}(S)_{\QQ}$ or the \emph{rational $\AA^1$-derived category}. The equivalence~\eqref{s1-q} persists after $\otimes$-inverting $\Sigma^{\infty}_{S^1}(\PP^1, \infty)$ on the left and $\QQ(\PP^1, \infty)$ on the right; see the discussion preceding \cite{cisinski-deglise}*{(5.3.35.2)}:

\begin{proposition} \label{prop:q-sh-da1} For any qcqs base scheme $S$ there is a canonical equivalence:
\begin{equation}
\SH(S)_{\QQ} \simeq \D_{\AA^1}(S, \QQ).
\end{equation}
Hence there is an equivalence after \'etale localizations:
\begin{equation}
\SH_{\et}(S)_{\QQ} \simeq \D_{\et,\AA^1}(S, \QQ).
\end{equation}
\end{proposition}

\subsubsection{} Any story on rational motivic homotopy theory starts with the decomposition of $\SH(S)_{\QQ}.$ According to Morel \cite{morel-trieste}*{Section 6.1}, there is an involution on the sphere spectrum, i.e. an endomorphism:
\[
\mathrm{sw}: \sspt_S \rightarrow \sspt_S
\]
such that $\mathrm{sw}^2 = \id.$ This gives a decomposition of the motivic sphere spectrum with $2$-inverted into $+$ and $-1$-eigenspectra:
\[
\sspt_S[\frac{1}{2}] \simeq \sspt^+_S[\frac{1}{2}] \oplus \sspt^-_S[\frac{1}{2}],
\]
where $\tau$ acts on $\sspt^+_S[\frac{1}{2}]$ (resp. $\sspt^-_S[\frac{1}{2}]$) by $+1$ (resp. $-1$). This decomposition of the motivic sphere spectrum then induces a decomposition:
\[
\SH(S)[\frac{1}{2}] \simeq \SH(S)[\frac{1}{2}]^- \times \SH(S)[\frac{1}{2}]^+,
\]
and this a decomposition:
\[
\SH(S)_{\QQ} \simeq \SH(S)_{\QQ}^- \times \SH(S)_{\QQ}^+.
\]
Hence we also obtain a decomposition upon \'etale localization:
\[
\SH_{\et}(S)_{\QQ} \simeq \SH_{\et}(S)_{\QQ}^- \times \SH_{\et}(S)_{\QQ}^+.
\]

\subsubsection{} The main result of rational \'etale motives is:

\begin{theorem} [Cisinski-D\'eglise] For any Noetherian, geometrically unibranch scheme $S$ of finite dimension there is an equivalence of $\infty$-categories:
\[
\DM(S)_{\QQ} \simeq  \SH_{\et}(S)_{\QQ} \simeq \SH(S)_{\QQ}^+.
\]
In particular, $\SH_{\et}(S)_{\QQ}^-$ vanishes.
\end{theorem}

\begin{proof} This is a matter of parsing definitions. In \cite{cisinski-deglise}*{Theorem 16.2.18} it is proved that the rational \'etale $\AA^1$-derived category is equivalent to ``Beilinson motives" for any Noetherian scheme of finite dimension. Under \'etale localization on both categories, the equivalence of Proposition~\ref{prop:q-sh-da1} tells us that $\SH_{\et}(\QQ)$ is naturally equivalent to the rational \'etale $\AA^1$-derived category. On the other hand in \cite{cisinski-deglise}*{Theorem 16.1.4}, it is proved that the category ``Beilinson motives" is equivalent to $\DM(S)_{\QQ}$ supposing further that $S$ is geometrically unibranch. The last statement follows from \cite{cisinski-deglise}*{Lemma 16.2.19} where it is proved that $\sspt_{\QQ}^-$ vanishes in $\SH_{\et}(\QQ)$ and thus the negative part of $\SH_{\et}(S)_{\QQ}$ vanishes.

\end{proof}

\vspace{20pt}
\scriptsize
\noindent
Elden Elmanto\\
Harvard University\\
One Oxford Street\\
MA 02138\\
USA\\
\texttt{eldenelmanto@gmail.com}

\vspace{10pt}
\noindent
Marc Levine\\
Universit\"at Duisburg-Essen\\
Fakult\"at f\"ur Mathematik\\
Thea-Leymann-Strasse 9\\
45127 Essen\\
Germany\\
\texttt{levine@uni-due.de}

\vspace{10pt}
\noindent
Markus Spitzweck\\
Universit\"at Osnabr\"uck\\
Fakult\"at f\"ur Mathematik\\
Albrechtstr. 28a\\
49076 Osnabr\"uck\\
Germany\\
\texttt{markus.spitzweck@uni-osnabrueck.de }

\vspace{10pt}
\noindent
Paul Arne \O stv\ae r\\
University of Oslo\\
Department of Mathematics\\
0316 Oslo\\
Norway\\

\noindent
Department of Mathematics Federigo Enriques\\
University of Milan\\
20133 Milan, Italy\\
\texttt{paularne@math.uio.no}
\texttt{paul.oestvaer@unimi.it}

\end{document}